\documentclass{article}%
\usepackage{amsmath}
\usepackage{graphicx}
\usepackage{multirow}
\usepackage{array}
\usepackage{amsfonts}%
\usepackage{amssymb}
\usepackage[mathscr]{eucal}
\newtheorem{theorem}{Theorem}

\newtheorem{corollary}[theorem]{Corollary}

\newtheorem{lemma}[theorem]{Lemma}

\newtheorem{proposition}[theorem]{Proposition}

\newtheorem{remark}[theorem]{Remark}

\newenvironment{proof}[1][Proof]{\textbf{#1.} }{\ \rule{0.5em}{0.5em}}
\begin{document}

\title{Nourdin-Peccati analysis on Wiener and Wiener-Poisson space for general distributions}
\author{Richard Eden \ \ \ \ \ \ \ \ \ \ \ \ Juan V$\acute{\text{\i}}$quez\\ \\Department of Mathematics, Purdue University\\150 N. University St., West Lafayette, IN 47907-2067, USA}
\date{\ }
\maketitle

\begin{abstract}
Given a reference random variable, we study the solution of its Stein equation
and obtain universal bounds on its first and second derivatives. We then
extend the analysis of Nourdin and Peccati by bounding the Fortet-Mourier and
Wasserstein distances from more general random variables such as members of
the Exponential and Pearson families. Using these results, we obtain
non-central limit theorems, generalizing the ideas applied to their analysis
of convergence to Normal random variables. We do these in both Wiener space
and the more general Wiener-Poisson space. In the former, we study conditions
for convergence under several particular cases and characterize
when two random variables have the same distribution. As an example, we apply this tool to bilinear functionals of Gaussian subordinated fields where the underlying process is a fractional Brownian motion with Hurst parameter bigger than 1/2. In the latter space we give sufficient conditions for a sequence of multiple (Wiener-Poisson) integrals to converge to a Normal random variable.

\end{abstract}

\section{Introduction}

Recent years have seen exciting research on combining Stein's method with
Malliavin calculus in proving central and non-central limit theorems. The
delicate combination of these tools can be attributed to Nourdin and Peccati who intertwined an integration by parts formula from
Malliavin calculus with an ordinary differential equation called a Stein
equation. Much work has been done to compare Normal or Gamma random variables (r.v.'s) with another r.v. (having unknown distribution). See
\cite{NorN}, \cite{NPnonclt}, \cite{NuO}, \cite{NuP} for results on
the convergence of multiple (Wiener) integrals to a standard Normal or Gamma
law. \cite{BT} and \cite{Tudor} discuss Cramer's theorem for Normal and Gamma
distributions applied to multiple integrals. \cite{Viens} gives probability tail bounds in terms of the Normal probability tail, with \cite{EV} applying the
same techniques to give tail bounds in terms of the probability tail of other r.v.'s (e.g. Pearson distributions).

\bigskip

In \cite{NP}, Nourdin and Peccati found a clever link between Stein's method
and Malliavin calculus. This was used to derive the Nourdin-Peccati upper
bound (NP bound) on the Wasserstein, Total Variation, Fortet-Mourier and
Kolmogorov distances of a generic r.v. from a Normal r.v., and lay the
groundwork for comparisons to a more general r.v. (with such results leading
to non-central limit theorems). These authors and Reinert (see \cite{NPR})
applied this NP bound to obtain a second order Poincar\'{e}-type inequality
useful in proving central limit theorems (CLTs) in Wiener space. Specifically,
they proved CLTs for linear functionals of Gaussian subordinated fields.
Particular instances are when the subordinated process is fractional Brownian
motion (fBm) or the solution to the Ornstein-Uhlenbeck (O-U) SDE driven by
fBm. They also characterized convergence in distribution to a Normal r.v. for
multiple stochastic integrals.

\bigskip

Later in \cite{PUST} these ideas were applied to prove the NP bound in Poisson
space (pure jump processes), which was used to obtain Berry-Ess$\acute{\text{e}}$en bounds for
arbitrary tensor powers of O-U kernels. Keeping in line with attempts to
extend these results as far as possible, \cite{Viquez} proved an NP bound in
Wiener-Poisson space. The author applied similar ideas found in \cite{NPR} to
derive a second order Poincar\'{e}-type inequality and use it to prove CLTs
for a continuous average of a product of two O-U processes (one in Wiener
space and the other in Poisson space) which lives in the second chaos of
Wiener-Poisson space. Also, it was proved that under mild conditions, the
small jumps part of a functional in the first Poisson chaos is approximately
equal in law to a functional in the first Wiener chaos with the same kernel
(useful when simulating a fractional L\'{e}vy process as a process with
finitely many jumps plus a fBm). All these results show the importance of this
NP bound and the potential it has as an effective tool in proving non-central
limit theorems, CLTs and characterizations.

\bigskip

Let $Z$ be absolutely continuous with respect to Lebesgue measure. For our purposes, we can think of $Z$ as a ``well-behaved'' r.v. (e.g. it lives in Wiener space and we know its density). Typical instances are when $Z$ is Normal, Gamma, or another member of the Pearson family of distributions. $X$ is another r.v. whose properties are not as easy to determine as with $Z$, our ``target'' r.v.. We may have a hunch that $X$ has the same distribution as $Z$, or in the case of sequences, a belief that $\left\{X_{n}\right\}  $ converges in law to the distribution of $Z$. We thus want to
compare $X$ with $Z$. How different are the laws of $X$ and $Z$ for instance
(and we need to make precise the sense in which they are different)? What
conditions will ensure that $X$ has the same law as $Z$? For a sequence
$\left\{  X_{n}\right\}  $, what sufficient conditions ensure convergence to
$Z$ in distribution? In this regard, we wish to measure the distance between (the laws of) $X$ and $Z$ by a metric $d_{\mathcal{H}}$ which induces a topology that is equal to or stronger than the topology of convergence in distribution:\ if $d_{\mathcal{H}}\left(  X_{n},Z\right) \rightarrow0$, then $X_{n} \rightarrow Z$ in distribution.

\bigskip

The motivation for this paper is to find the widest generalization of the NP
bound by applying it to a target r.v. which is neither Normal nor Gamma, and in
both Wiener space and Wiener-Poisson space. This is worked out in \cite{KT}
but the conditions needed to apply the NP bound are quite restrictive (it was
also carried out only in Wiener space). The conditions we are introducing here
are more general, and are still wide enough in scope to cover a $Z$ belonging
to the Exponential family or the Pearson family. We point out that
Wiener-Poisson space is more inclusive than Wiener space (which can be
identified with a subspace of the former). In fact, it includes processes with
jumps, and therefore considers Poisson space too as a subspace (also by
identification). Nevertheless, even if Wiener space is less general, we can
apply our techniques to a wider class of target r.v. $Z$ than in
Wiener-Poisson space (which requires boundedness of the second derivative of
the solution of the Stein's equation, something not needed in Wiener space).

\bigskip

Our main results are the NP bounds on $d_{\mathcal{H}}\left(  X,Z\right)  $ in Wiener space
and in Wiener-Poisson space. The main result in Wiener space (Theorem
\ref{thm_boundnew}) is%
\begin{align*}
d_{\mathcal{H}}\left(  X,Z\right)   &  \leq k\mathbb{E}\left|  g_{\ast}\left(  X\right)
-g_{X}\right| \\
&  \leq k\sqrt{\left|  \mathbb{E}\left[  g_{\ast}\left(  X\right)
^{2}\right]  -\mathbb{E}\left[  g_{\ast}\left(  Z\right)  ^{2}\right]
\right|  +\left|  \mathbb{E}\left[  XG_{\ast}\left(  X\right)  \right]
-\mathbb{E}\left[  ZG_{\ast}\left(  Z\right)  \right]  \right|  +\left|
\mathbb{E}\left[  g_{X}^{2}\right]  -\mathbb{E}\left[  g_{Z}^{2}\right]
\right|  }\text{.}%
\end{align*}
The main result in Wiener-Poisson space (Theorem \ref{thm_boundnewWP}) is%
\[
d_{\mathcal{H}}\left(  X,Z\right)  \leq k\biggl(\mathbb{E}\left|  g_{\ast}\left(  X\right)
-g_{X}\right|  +\mathbb{E}\left[  \left\langle \left|  x\left(  DX\right)
^{2}\right|  ,\left|  -DL^{-1}X\right|  \right\rangle _\mathfrak{H}\right]  \biggr
)\text{.}%
\]
Here, $g_{X}:=\mathbb{E}[\langle DX,-DL^{-1}X\rangle_\mathfrak{H}|X]$ and
$g_{Z}:=\mathbb{E}[\langle DZ,-DL^{-1}Z\rangle_\mathfrak{H}|Z]$ are random variables
defined using Malliavin calculus operators, specifically, the Malliavin
derivative $D$ and the inverse of the infinitesimal generator $L$ of the O-U
semigroup. The definitions parallel each other, but it would be helpful to
think of $g_{X}$ as an object belonging exclusively to $X$, and $g_{Z}$ to
$Z$. On the other hand, $g_{\ast}(\cdot):=\mathbb{E}[\langle DZ,-DL^{-1}%
Z\rangle_\mathfrak{H}|Z=\cdot]$ is a function whose support is the support of $Z$,
taking on nonnegative numbers as values. It will depend only on the density of
$Z$, and is independent of the structure of $X$. As such, it is an object
belonging solely to $Z$. In the second term of the first bound above,
$G_{\ast}$ is an antiderivative of $g_{\ast}$, provided it exists. We can make
sense of the NP bounds in the following way: if we want to know how different
the laws of $X$ and $Z$ are, then we need to know how different (in the
$L^{1}$ sense) $g_{X}=\mathbb{E}[\langle DX,-DL^{-1}X\rangle_\mathfrak{H}|X]$ and
$g_{\ast}(X)=\mathbb{E}[\langle DZ,-DL^{-1}Z\rangle_\mathfrak{H}|Z=X]$ are. In
Wiener-Poisson space, we consider in addition how close the jump part
$\mathbb{E}[\langle|x\left(  DX\right)  ^{2}|,|-DL^{-1}X|\rangle_\mathfrak{H}]$ is to
$0$, which makes sense since $Z$ belongs to Wiener space (subspace of the
Wiener-Poisson space without jumps).

\bigskip

In our bounds above, $k$ is a constant that does not depend on $X$ but on $Z$
and the metric we are using. For convergence problems, we do not need its
specific value since the convergence will follow from the convergence of
$\mathbb{E}\left|  g_{\ast}\left(  X_{n}\right)  -g_{X_{n}}\right|  $ to $0$.
This presupposes we have such a constant $k$. This constant appears as a bound
($\left\|  \phi\right\|  _{\infty}\leq k$) for some function $\phi$, which is
related to the solution of the underlying Stein equation. In particular, since
we have a Stein equation for each $Z$ (the target r.v.), $k$ depends on $Z$.
Finding such a bound $k$ is easy when $Z$ is Normal: $g_{\ast}$ is
constant, and consequently, the Stein equation is simpler. If $g_{\ast}$
vanishes at a finite endpoint of the support of $Z$, the challenge now is to find a
bound for $\left\|  \phi/g_{\ast}\right\|  _{\infty}$. To the best of our
knowledge, \cite{KT} (Kusuoka and Tudor) presents the first attempt to find
such a sup\ norm bound when $Z$ is not Normal. Their result is presented below
as Lemma \ref{prop_tudorbound}. In Theorem \ref{prop_fprimebound} we
improve their result, and this paves the way for the needed bound we stated
for the general non-Normal case.

\bigskip

The paper is organized as follows. In Section \ref{sec_Malliavin}, we review
the operators we need from Malliavin calculus. We also define the functions
$g_{\ast}$ and $G_{\ast}$ as well as the random variables $g_{X}$ and $g_{Z}$, 
studying carefully their properties (needed in the subsequent sections).
Section \ref{sec_Stein} contains preliminaries on Stein's method. Here we find
universal bounds on the first and second derivatives of the solution of the
general Stein equation. Our main result in Wiener space is in Section
\ref{sec_Wresults}, where we give a tractable upper inequality which is easier
to compute. We also characterize when the law of $X$ is the same as that of
$Z$. Said result is applied to specific cases when $g_{\ast}$ is a polynomial
and when $\left\{  X_{n}\right\}  $ is a sequence of multiple integrals. As an example, we prove the convergence of a bilinear functional of a Gaussian subordinated field to a $\chi^2$ r.v. by computing some moments and showing their convergence to desired values. In Section \ref{sec_WPresults}, we extend the main result to the more general Wiener-Poisson space. Here, we work out some sufficient conditions for
convergence to a Normal law and convergence of the fourth moment.



\section{\label{sec_Malliavin}Elements of Malliavin calculus and tools}

For the sake of completeness, we include here a brief survey of the needed
Malliavin calculus objects. The r.v. $\left\langle DF,-DL^{-1}%
F\right\rangle_\mathfrak{H} $, to be constructed for $F=Z$ and for $F=X$, is a key element
that bridges Stein's method and Malliavin calculus. $D$ is the Malliavin
derivative operator and $L$ is the generator of the Ornstein-Uhlenbeck semigroup.

\subsection{Wiener space}

Nualart presents in Chapter 1 of \cite{Nualart} a very good exposition on
Malliavin calculus in Wiener space. We mention here the elements that we need.
Let $\mathfrak{H}$ be a real separable Hilbert space. Assume a probability space $\left(
\Omega,\mathcal{F},\mathbb{P}\right)  $ over which $W=\left\{  W\left(
h\right)  :h\in \mathfrak{H}\right\}  $ is an isonormal Gaussian process. By definition,
this means $W$ is a centered Gaussian family such that $\mathbb{E}\left[
W\left(  h_{1}\right)  W\left(  h_{2}\right)  \right]  =\left\langle
h_{1},h_{2}\right\rangle_\mathfrak{H}$. We may also assume that $\mathcal{F}$ is the
$\sigma$-field generated by $W$. The white noise case is when $\mathfrak{H}
=L^{2}\left( T, \mathcal{B}, \mu \right)$ where $(T, \mathcal{B})$ is a measurable space and $\mu$ is a $\sigma$-finite atomless measure. The Gaussian process $W$ is then characterized by the family of r.v.'s $\left\{  W\left(  A\right)
:A\in\mathcal{B}\left[  0,T\right],\mu\left(  A\right)  <\infty\right\}  $ where $W\left(A\right)  =W\left(  \mathbf{1}_{A}\right)  $; we write $W_{t}=W\left(
\mathbf{1}_{\left[  0,t\right]  }\right)  $. We can then think of $W$ as an
$L^{2}\left(  \Omega,\mathcal{F},\mathbb{P}\right)  $ random measure on $\left(  T,\mathcal{B}\right)  $. This is called the white noise measure based on $\mu$.

\bigskip

The $q^{\text{th}}$ Hermite polynomial $H_{q}$ is given by $H_{q}\left(
x\right)  =\left(  -1\right)  ^{q}e^{x^{2}/2}\frac{d^{q}}{dx^{q}%
}\left(  e^{-x^{2}/2}\right)  $ for $q\geq1$ and $H_{0}\left(  x\right)  =1$.
The $q^{\text{th}}$ Wiener chaos $\mathcal{H}_{q}$ is defined as the subspace
of $L^{2}\left(  \Omega\right)  =L^{2}\left(  \Omega,\mathcal{F}%
,\mathbb{P}\right)  $ generated by the r.v.'s $\left\{  H_{q}\left(
W\left(  h\right)  \right)  :h\in \mathfrak{H},\left\|  h\right\|  _\mathfrak{H}=1\right\}  $.\ In
the white noise case $\mathfrak{H}=L_{\mu}^{2}\left(  \left[  0,1\right]  \right)  $,
each Wiener chaos consists of iterated multiple (Wiener) integrals%
\[
I_{q}\left(  f\right)  :=n!\int_{0}^{1}\int_{0}^{t_{1}}\cdots\int_{0}%
^{t_{q-1}}f\left(  t_{1},t_{2},\ldots,t_{q}\right)  dW_{t_{q}}\cdots
dW_{t_{2}}dW_{t_{1}}%
\]
with respect to $W$, where $f\in \mathfrak{H}^{\odot q}$ is a symmetric nonrandom kernel. When $f$ is nonsymmetric, we let $\widetilde{f}$ denote its symmetrization, and $I_q(f) = I_q(\widetilde{f})$.

\bigskip

All elements of $\mathcal{H}_{1}$ are Gaussian and all elements of
$\mathcal{H}_{0}$ are deterministic. It is well-known that $L^{2}\left(
\Omega\right)  $ can be decomposed into an infinite orthogonal sum of the
Wiener chaoses, i.e. $L^{2}\left(  \Omega\right)  =\oplus_{q=0}^{\infty
}\mathcal{H}_{q}$. In the white noise case, any $F\in L^{2}\left(
\Omega\right)  $ admits a Wiener chaos decomposition of multiple integrals
\begin{equation}
F=\sum_{q=0}^{\infty}I_{q}\left(  f_{q}\right)  \label{eqn_chaosdecom}%
\end{equation}
where each symmetric $f_{q}\in \mathfrak{H}^{\odot q}=L_{\mu}^{2}\left(  T^{q}\right)  $
is uniquely determined by $F$. Note that $I_{0}\left(  f_{0}\right)
=f_{0}=\mathbb{E}\left[  F\right]  $ and $\mathbb{E}\left[  I_{q}\left(
f_{q}\right)  \right]  =0$ for $q\geq1$.

\bigskip

Consider an orthonormal system $\left\{  e_{k}:k\geq1\right\}  $ in $\mathfrak{H}$. For
$f\in \mathfrak{H}^{\otimes p}$ and $g\in \mathfrak{H}^{\otimes q}$, the contraction of order
$r\leq\min\left\{  p,q\right\}  $ is the element $f\otimes_{r}g\in
\mathfrak{H}^{\otimes\left(  p+q-2r\right)  }$ defined by%
\[
f\otimes_{r}g=\sum_{i_{1},\ldots,i_{r}}^{\infty}\left\langle f,e_{i_{1}%
}\otimes\cdots\otimes e_{i_{r}}\right\rangle _{\mathfrak{H}^{\otimes r}}\left\langle
g,e_{i_{1}}\otimes\cdots\otimes e_{i_{r}}\right\rangle _{\mathfrak{H}^{\otimes r}%
}\text{.}%
\]
Even if $f$ and $g$ are symmetric, $f\otimes_{r}g$ may be nonsymmetric so we
denote its symmetrization by $f\widetilde{\otimes}_{r}g$. In the white noise case
$\mathfrak{H}=L_{\mu}^{2}\left(  T\right)  $, the contraction is given by integrating out
$r$ variables. Thus, if $f\in L_{\mu}^{2}\left(  T^{p}\right)  $ and $q\in
L_{\mu}^{2}\left(  T^{q}\right)  $, we have $f\otimes_{r}g\in L_{\mu}%
^{2}\left(  T^{p+q-2r}\right)  $ and%
\[
\left(  f\otimes_{r}g\right)  \left(  t_{1},\ldots,t_{p+q-2r}\right)
=\int_{T^{r}}f\left(  t_{1},\ldots,t_{p-r},s_{1},\ldots,s_{r}\right)  g\left(
t_{p+1},\ldots,t_{p+q-r},s_{1},\ldots,s_{r}\right)  d\mu\left(  s_{1}\right)
\cdots d\mu\left(  s_{r}\right)  \text{.}%
\]

The product of two multiple integrals is
\begin{equation}
I_{q}\left(  f\right)  I_{p}\left(  g\right)  =\sum_{r=0}^{p\wedge q}%
r!\binom{p}{r}\binom{q}{r}I_{q+p-2r}\left(  f\otimes_{r}g\right)  \text{.}
\label{eqn_productW}%
\end{equation}

\bigskip

The Malliavin derivative of a random variable $F\in L^{2}\left(
\Omega\right)  $ is an $\mathfrak{H}$-valued random variable denoted by $DF$. In the
white noise case $\mathfrak{H}=L_{\mu}^{2}\left(  T\right)  $, if $F=I_{1}\left(
f\right)  =\int_{T}f\left(  t\right)  dW_{t}$, then $D$ maps $F$ to an
$L_{\mu}^{2}\left(  T\right)  $-valued element: $D_{r}F=f\left(  r\right)  $
for $r\in T$. In general, if $F\in L^{2}\left(  \Omega\right)  $ admits the
decomposition (\ref{eqn_chaosdecom}), then
\begin{equation}
D_{r}F=\sum_{q=1}^{\infty}qI_{q-1}\left(  f_{q}\left(  r,\cdot\right)
\right)  \text{.} \label{eqn_chaosderiv}%
\end{equation}
We denote by $\mathbb{D}^{1,2}$ the domain of $D$ in $L^{2}\left(
\Omega\right)  $. $F$ with the above decomposition is in $\mathbb{D}^{1,2}$ if
and only if $\mathbb{E}\left[  \left\|  DF\right\|  _{L_{\mu}^{2}\left(
T\right)  }^{2}\right]  =\sum_{q=1}^{\infty}q\cdot q!\left\|  f_{q}\right\|
_{L_{\mu}^{2}\left(  T^{q}\right)  }^{2}<\infty$. $D$ satisfies the chain rule
formula: $D\left(  f\left(  F\right)  \right)  =f^{\prime}\left(  F\right)
DF$ when $F\in\mathbb{D}^{1,2}$ and $f$ is differentiable with bounded
derivative. One may relax this to almost everywhere differentiability of $f$
as long as $F$ has an absolutely continuous law.

\bigskip

$D$ has an adjoint, the divergence operator $\delta$, so that if $F\in
\operatorname*{Dom}\delta\subset L^{2}\left(  \Omega;\mathfrak{H}\right)  $, then 
$\delta\left(  F\right)  \in L^{2}\left(  \Omega\right)  $ and $\mathbb{E}%
\left[  \delta\left(  F\right)  G\right]  =\mathbb{E}\left[  \left\langle
F,DG\right\rangle _\mathfrak{H}\right]  $ for any $G\in\mathbb{D}^{1,2}$. In the white
noise case, $\delta$ is called the Skorohod integral: for $F\in
\operatorname*{Dom}\delta\subset L_{\mu\times\mathbb{P}}^{2}\left(
T\times\Omega\right)  $ with chaos representation $F\left(  t\right)
=\sum_{q=0}^{\infty}I_{q}\left(  f_{q}\left(  t,\cdot\right)  \right)  $ where
each $f_{q}\in L_{\mu^{\otimes\left(  q+1\right)  }}^{2}$ is symmetric in the
last $q$ variables, $\delta\left(  F\right)  =\sum_{q=0}^{\infty}%
I_{q+1}\left(  \widetilde{f}_q\right)  $ if $\sum_{q=0}^{\infty}\left(
q+1\right)  !\left\|  \widetilde{f}_q\right\|  _{L_{\mu^{\otimes\left(
q+1\right)  }}^{2}}^{2}<\infty$, i.e. $F\in\operatorname*{Dom}\delta$.

\bigskip

One other operator we need, $L$, acts on $F$ as in (\ref{eqn_chaosdecom}) in this
way: $LF=-\sum_{q=1}^{\infty}qI_{q}\left(  f_{q}\right)  $. Its domain
consists of $F$ for which $\sum_{q=1}^{\infty}q^{2}\cdot q!\left\|
f_{q}\right\|  _{L_{\mu}^{2}\left(  T^{q}\right)  }^{2}<\infty$. $L$ also
happens to be the infinitesimal generator of the Ornstein-Uhlenbeck semigroup
$T_{t}$, defined by $T_{t}F=\sum_{q=0}^{\infty}e^{-qt}I_{q}\left(
f_{q}\right)  $. One important relation is $\delta DF=-LF$. More than $L$, we
need its pseudo-inverse $L^{-1}$ defined by $L^{-1}F=-\sum_{q=1}^{\infty}%
\frac{1}{q}I_{q}\left(  f_{q}\right)  $. It easily follows that $L^{-1}%
LF=F-\mathbb{E}\left[  F\right]  $.

\subsection{Wiener-Poisson space}

Assume a complete probability space $\left(  \Omega,\mathcal{F},\mathbb{P}%
\right)  $ over which $\mathcal{L}=\left\{  \mathcal{L}_{t}\right\}  _{t\geq
0}$ is a L\'{e}vy process. By definition, this means $\mathcal{L}$ has
stationary and independent increments, is continuous in probability, and
$\mathcal{L}_{0}=0$. Suppose $\mathcal{L}$ is cadlag, centered, and
$\mathbb{E}\left[  \mathcal{L}_{1}^{2}\right]  <\infty$. We may also assume
$\mathcal{F}$ is generated by $\mathcal{L}$. Let $\mathcal{L}$ have L\'{e}vy
triplet $\left(  0,\sigma^{2},\nu\right)  $ and thus, L\'{e}vy-It\^{o}
decomposition $\mathcal{L}_{t}=\sigma W_{t}+\int\int_{\left[  0,t\right)
\times\mathbb{R}_{0}}xd\widetilde{N}\left(  s,x\right)  $ where $W=\left\{
W_{t}\right\}  _{t\geq0}$ is a standard Brownian motion, $\widetilde{N}$ is the
compensated jump measure (defined in terms of $\nu$) and $\mathbb{R}%
_{0}=\mathbb{R}-\left\{  0\right\}  $. See \cite{Apple} and \cite{Sato} for
more about L\'{e}vy processes.

\bigskip

Consider now the measure $\mu$ on $\mathcal{B}\left(  \mathbb{R}^{+}%
\times\mathbb{R}\right)  $ where $\mathbb{R}^{+}=\left\{  t:t\geq0\right\}  $
and
\[
d\mu\left(  t,x\right)  =\sigma^{2}dt\delta_{0}\left(  x\right)  +x^{2}%
dtd\nu\left(  x\right)  \left(  1-\delta_{0}\left(  x\right)  \right)
\text{.}%
\]
Analogous to a Gaussian process $W$ being extended to a random measure (which we also denoted by $W$) in Wiener space, $\mathcal{L}$ can
be extended to a random measure $M\,$(see \cite{Ito}) on $\left(
\mathbb{R}^{+}\times\mathbb{R},\mathcal{B}\left(  \mathbb{R}^{+}%
\times\mathbb{R}\right)  \right)  $. This is used to construct (in an
analogous way to the It\^{o} integral construction) an integral on step
functions, and then by linearity and continuity, extended to $L_{\mu^{\otimes
q}}^{2}=L^{2}\left(  \left(  \mathbb{R}^{+}\times\mathbb{R}\right)
^{q},\mathcal{B}\left(  \mathbb{R}^{+}\times\mathbb{R}\right)  ^{q}%
,\mu^{\otimes q}\right)  $. We also denote it by $I_{q}$. As in Wiener space,

\begin{enumerate}
\item $I_{q}\left(  f\right)  =I_{q}\left(  \widetilde{f}\right)  $;

\item $I_{q}$ is linear;

\item $\mathbb{E}\left[  I_{q}\left(  f\right)  I_{p}\left(  g\right)
\right]  =\mathbf{1}_{\{q=p\}}q!\int_{\left(  \mathbb{R}^{+}\times\mathbb{R}%
\right)  ^{q}}\widetilde{f}\widetilde{g}d\mu^{\otimes q}$.
\end{enumerate}

Thus, when $F=I_q(f)$, $\mathbb{E}[F^2]=\mathbb{E}[I_q(f)^2]=q!||\widetilde{f}||^2_{L_\mu^{\otimes q}}$.

Contractions are defined slightly differently. Suppose $f\in L_{\mu^{\otimes
q}}^{2}$ and $g\in L_{\mu^{\otimes p}}^{2}$. Let $r\leq\min\left\{
q,p\right\}  $ and $s\leq\min\left\{  q,p\right\}  -r$. The contraction
$f\otimes_{r}^{s}g\in L_{\mu^{\otimes\left(  q+p-2r-s\right)  }}^{2}$ is
defined by integrating out $r$ variables and sharing $s$ of the remaining
variables:%
\[
\left(  f\otimes_{r}^{s}g\right)  \left(  z,u,v\right)  =\left(  \prod
_{i=1}^{s}x_{i}\right)  \left\langle f\left(  \cdot,z,u\right)  ,g\left(
\cdot,z,v\right)  \right\rangle _{L_{\mu^{\otimes r}}^{2}}%
\]
where $z\in\left(  \mathbb{R}^{+}\times\mathbb{R}_{0}\right)  ^{s}$,
$z_{i}=\left(  t_{i},x_{i}\right)  $, $u\in\left(  \mathbb{R}^{+}%
\times\mathbb{R}_{0}\right)  ^{q-r-s}$ and $v\in\left(  \mathbb{R}^{+}%
\times\mathbb{R}_{0}\right)  ^{p-r-s}$. Its symmetrization is $f\widetilde
{\otimes}_{r}^{s}g$. We need the following product formula later (see \cite{LS}
for the proof):
\begin{equation}
I_{q}\left(  f\right)  I_{p}\left(  g\right)  =\sum_{r=0}^{p\wedge q}%
\sum_{s=0}^{p\wedge q-r}r!s!\binom{p}{r}\binom{q}{r}\binom{p-r}{s}\binom
{q-r}{s}I_{q+p-2r-s}\left(  f\otimes_{r}^{s}g\right)  \text{.}
\label{eqn_productWP}%
\end{equation}

We may think of this as a more general version of the product formula (\ref{eqn_productW}) where we only consider $s=0$ since there are no jump components to be shared (which appear in the definition of $f\otimes_{r}^{s}g$).

\bigskip

We have briefly narrated a setup parallel to what was done in Wiener space.
See \cite{SUV} for a more detailed exposition. This time though, we have only
considered $\mathfrak{H}=L^{2}\left(  \mathbb{R}^{+}\times\mathbb{R},\mathcal{B}\left(
\mathbb{R}^{+}\times\mathbb{R}\right)  ,\mu\right)  $ as underlying Hilbert
space, with inner product $\left\langle f,g\right\rangle _\mathfrak{H}=\int
_{\mathbb{R}^{+}\times\mathbb{R}}f\left(  z\right)  g\left(  z\right)
d\mu\left(  z\right)  $. There is as yet no Malliavin calculus theory
developed for a more general abstract Hilbert space. While we don't have a
chaos decomposition via orthogonal polynomials (like Hermite polynomials in
Wiener space; see \cite{DiNunno}), we still have a comparable decomposition
proved by It\^{o} (Theorem 2, \cite{Ito}): for $F\in L^{2}\left(
\Omega,\mathcal{F},\mathbb{P}\right)  $,%
\begin{equation}
F=\sum_{q=0}^{\infty}I_{q}\left(  f_{q}\right)  \text{ where }f_{q}\in
L_{\mu^{\otimes q}}^{2}\text{.} \label{eqn_chaosdecomWP}%
\end{equation}

With this decomposition, we can define the Malliavin derivative operator and
Skorohod integral operator. Define $\operatorname*{Dom}D$ as the set of $F\in
L^{2}\left(  \Omega\right)  $ for which $\sum_{q=1}^{\infty}qq!\left\|
f_{q}\right\|  _{L_{\mu^{\otimes q}}^{2}}^{2}<\infty$ and
\[
D_{z}F=\sum_{q=1}^{\infty}qI_{q-1}\left(  f_{q}\left(  z,\cdot\right)
\right)  \text{.}%
\]

It is instructive to consider the derivatives $D_{t,0}$ and $D_{z}$ where
$z=\left(  t,x\right)  $ has $x\neq0$. This will enable us to better
understand the similarities, and where they end, between the Malliavin
calculus of Wiener space and that of Wiener-Poisson space. See \cite{SUV} and
\cite{SUV2}\ for more details on the following discussion. We consider two
spaces on which we can embed $\operatorname*{Dom}D$. For $F\in L^{2}\left(
\Omega\right)  $, we say $F\in\operatorname*{Dom}D^{0}$ iff $\sum
_{q=1}^{\infty}qq!\int_{\mathbb{R}^+}\left\|  f_{q}\left(  \left(  t,0\right)
,\cdot\right)  \right\|  _{L_{\mu^{\otimes\left(  q-1\right)  }}^{2}}%
^{2}dt<\infty$ and $F\in\operatorname*{Dom}D^{J}$ iff $\sum_{q=1}^{\infty
}qq!\int_{\mathbb{R}^{+}\times\mathbb{R}_{0}}\left\|  f_{q}\left(
z,\cdot\right)  \right\|  _{L_{\mu^{\otimes\left(  q-1\right)  }}^{2}}^{2}%
d\mu\left(  z\right)  <\infty$. In fact, $\operatorname*{Dom}%
D=\operatorname*{Dom}D^{0}\cap\operatorname*{Dom}D^{J}$. Since $W$ and
$\widetilde{N}$ are independent, we can think of $\Omega$ as a cross product of
the form $\Omega_{W}\times\Omega_{J}$ where $\Omega_{W}=\mathcal{C}\left(
\mathbb{R}^{+}\right)  $ and $\Omega_{J}$ consists of the sequences $\left(
\left(  t_{1},x_{1}\right)  ,\left(  t_{2},x_{2}\right)  ,\ldots\right)
\in\left(  \mathbb{R}^{+}\times\mathbb{R}_{0}\right)  ^{\mathbb{N}}$ (with a
few other technical conditions).

\begin{itemize}
\item The derivative $D_{t,0}$ can be interpreted as the derivative with
respect to the Brownian motion part. In fact, if $\nu=0$, then $D_{t,0}%
F=\frac{1}{\sigma}D_{t}^{W}F$ where $D^{W}$ is the classical Malliavin
derivative (defined in Wiener space); the $\frac{1}{\sigma}$ comes from the
fact that we are differentiating with respect to $\sigma W_{t}$ and not just
$W_{t}$. From the isometry $L^{2}\left(  \Omega\right)  \simeq L^{2}\left(
\Omega_{W};L^{2}\left(  \Omega_{J}\right)  \right)  $, consider $F\in
L^{2}\left(  \Omega\right)  $ as an element of $L^{2}\left(  \Omega_{W}%
;L^{2}\left(  \Omega_{J}\right)  \right)  $. A smooth $F$ then has form
$F=\sum_{i=1}^{n}G_{i}H_{i}$ where each $G_{i}$ is a smooth Brownian random
variable and $H_{i}\in L^{2}\left(  \Omega_{J}\right)  $. We can then define
$D^{W}$ by $D^{W}F=\sum_{i=1}^{n}\left(  D^{W}G_{i}\right)  H_{i}$, where
$D^{W}G_{i}$ is the classical Malliavin derivative. It can be shown that this
definition can be extended to a subspace $\operatorname*{Dom}D^{W}%
\subset\operatorname*{Dom}D^{0}$, so that for $F\in\operatorname*{Dom}D^{W}$, as
expected,
\[
D_{t,0}F=\frac{1}{\sigma}D_{t}^{W}F\text{.}%
\]
For functionals of the form $F=f\left(  G,H\right)  \in L^{2}\left(
\Omega\right)  $ having $G\in\operatorname*{Dom}D^{W}$, $H\in L^{2}\left(
\Omega_{J}\right)  $, and such that $f$ is continuously differentiable with bounded partial
derivatives in the first variable, we have a chain rule result: $F\in
\operatorname*{Dom}D^{0}$ and $D_{t,0}F=\frac{1}{\sigma}\frac{\partial
f}{\partial x}\left(  G,H\right)  D_{t}^{W}G$. We may loosen the restriction
on $f$ to a.e. differentiability if $G$ is absolutely continuous.

\item The derivative $D_{z}$, $z=\left(  t,x\right)  $ with $x\neq0$, is a
difference operator: for $F\in\operatorname*{Dom}D^{J}$%
\[
D_{z}F=\frac{F\left(  \omega_{t,x}\right)  -F\left(  \omega\right)  }{x}%
\]
where, if $\Psi_{z}F$ is the right-hand expression, then
$\mathbb{E}\left[  \int_{\mathbb{R}^{+}\times\mathbb{R}_{0}}\left(  \Psi
_{z}F\right)  ^{2}d\mu\left(  z\right)  \right]  <\infty$. The idea is to
introduce a jump of size $x$ at time $t$ which is captured by the realization
$\omega_{t,x}$. For $\omega=\left(  \omega^{W},\omega^{J}\right)  $, we define
$\omega_{t,x}$ by simply adding the time-jump pair $\left(  t,x\right)  $ to
$\omega^{J}$. For $F=f\left(  G,H\right)  \in L^{2}\left(  \Omega\right)  $
with $G\in L^{2}\left(  \Omega_{J}\right)  $, $H\in$ $\operatorname*{Dom}%
D^{J}$ and $f$ continuous, we have this chain rule result:
\[
D_{z}F=\frac{f\left(  G,H\left(  \omega_{t,x}\right)  \right)  -f\left(
G,H\left(  \omega\right)  \right)  }{x}=\frac{f\left(  G,xD_{z}H+H\left(
\omega\right)  \right)  -f\left(  G,H\left(  \omega\right)  \right)  }%
{x}\text{.}%
\]
If $f$ is differentiable, then by the mean value theorem, for some random
$\theta_{z}\in\left(  0,1\right)  $,
\[
D_{z}F=\frac{\partial f}{\partial y}\left(  G,\theta_{z}xD_{z}H+H\left(
\omega\right)  \right)D_zH  \text{.}%
\]
\end{itemize}

The following unified chain rule will be very useful (see Proposition 2 in
\cite{Viquez}): If $F\in\operatorname*{Dom}D^{W}\cap\operatorname*{Dom}D^{J}$,
$DF\in L_{\mu}^{2}$, $f\in\mathcal{C}^{k-1}$ has bounded first derivative 
(or $f'$ may be unbounded if $F$ is absolutely continuous with respect to Lebesgue measure) and $f^{(k-1)}$ is Lipschitz, then for
$z\in\left(  t,x\right)  \in\mathbb{R}^{+}\times\mathbb{R}$,%

\begin{equation}
D_zf(F) = \sum_{n=1}^{k-1}\frac{f^{(n)}(F)}{n!}x^{n-1}(D_zF)^n+\int_0^{D_zF}\frac{f^{(k)}(F+xu)}{(k-1)!}x^{k-1}(D_zF-u)^{k-1}du\label{eqn_chainWP_L}.%
\end{equation}

In the case where $f^{(k-1)}$ is differentiable everywhere, the chain rule is
\begin{align}
D_zf(F) = \sum_{n=1}^{k-1}\frac{f^{(n)}(F)}{n!}x^{n-1}(D_zF)^n+\frac{f^{(k)}(F+\theta_zxD_zF)}{k!}x^{k-1}(D_zF)^k\label{eqn_chainWP}
\end{align}
for some function $\theta_z\in(0,1)$ for all $z=(t,x)\in\mathbb{R}^+\times\mathbb{R}$.

\bigskip

We now define the adjoint of $D$ (see \cite{SUV} again). Suppose $F\in$
$L^{2}\left(  \mathbb{R}^{+}\times\mathbb{R}\times\Omega,\mathcal{B}\left(
\mathbb{R}^{+}\times\mathbb{R}\right)  \times\mathcal{F},\mu\times
\mathbb{P}\right)  $ with $F\left(  z\right)  =\sum_{q=0}^{\infty}I_{q}\left(
f_{q}\left(  z,\cdot\right)  \right)  $ where each $f_{q}\in L_{\mu
^{\otimes\left(  q+1\right)  }}^{2}$ is symmetric in the last $q$ variables.
In this case, the Skorohod integral of $F$ is $\delta\left(
F\right)  =\sum_{q=0}^{\infty}I_{q+1}\left(  \widetilde{f_{q}}\right)  $ where
$\sum_{q=0}^{\infty}\left(  q+1\right)  !\left\|  \widetilde{f_{q}}\right\|
_{L_{\mu^{\otimes\left(  q+1\right)  }}^{2}}^{2}<\infty$, i.e. $F\in
\operatorname*{Dom}\delta$ (by definition). Furthermore, 
$\mathbb{E}\left[  \delta\left(  F\right)  G\right]  =\mathbb{E}\left[
\left\langle F,DG\right\rangle _{L_{\mu}^{2}}\right]  $ for any $G\in
\operatorname*{Dom}D$.

\bigskip

Finally, we define as before $L=-\delta D$: for $F$ as in
(\ref{eqn_chaosdecomWP}), $LF=-\sum_{q=1}^{\infty}qI_{q}\left(  f_{q}\right)
$. The pseudo-inverse is defined by $L^{-1}F=-\sum_{q=1}^{\infty}\frac{1}%
{q}I_{q}\left(  f_{q}\right)  $. We have again $L^{-1}LF=F-\mathbb{E}\left[
F\right]  $.

\bigskip

\begin{remark}
\label{equivalencyW-WP} Write $\vec{z}_{q}=(z_{1},\dots,z_{q})$, with
$z_{i}=(t_{i},x_{i})$ for all $i$. Define
\[
\widetilde{W}=\left\{  F=\sum_{q=0}^{\infty}I_{q}(f_{q})\in\operatorname*{Dom}%
D^{0}\ :\ f_{q}\in L_{\mu^{\otimes q}}^{2}\text{, and for every }q\text{,
}f_{q}(\vec{z}_{q})=0\text{ if }x_{i}\neq0\text{ for some }i\right\}  \text{.}%
\]
Notice from the previous discussion that if $f_{q}(\vec{z}_{q})=0\text{
because }x_{i}\neq0$, then $I_{q}(f_{q})$ coincides with an iterated multiple
(Wiener) integral. Therefore, Wiener space can be seen as a subspace of
Wiener-Poisson space (similarly for Poisson space as a subspace). Moreover, $\widetilde{W}$
coincides with the subspace $\mathbb{D}^{1,2}$ (through embedding). The
relevance of these facts is that if we have a r.v. $F\in\widetilde{W}$, then
the chain rule formula and the Malliavin calculus operators are exactly (up to
a constant) the same as those in Wiener space (as explained earlier in this
subsection). Furthermore, the results (from other papers) in Wiener space can
be replicated in $\widetilde{W}$ and so the conclusions will hold in
Wiener-Poisson space, but within $\widetilde{W}$. From now on, $\mathbb{D}%
^{1,2}$ will mean the subspace $\mathbb{D}^{1,2}$ in Wiener space or the
respective embedding $\widetilde{W}$ in Wiener-Poisson space.
\end{remark}

\subsection{The random variables $\boldsymbol{g_X}$, $\boldsymbol{g_Z}$ and the functions $\boldsymbol{g_*}$, $\boldsymbol{G_*}$}

From this point on, $\mathfrak{H}$ will be taken as $L^{2}\left(  \mathbb{R}^{+}\times\mathbb{R},\mathcal{B}\left(
\mathbb{R}^{+}\times\mathbb{R}\right)  ,\mu\right)$ if we are in Wiener-Poisson space. Now suppose $F$ has mean $0$.  We have the following integration by parts formulas.%
\begin{itemize}

\item If $F\in\operatorname*{Dom}D^{W}\cap\operatorname*{Dom}D^{J}$ and $f\in\mathcal{C}^1$ with Lipschitz first derivative assumed to be bounded (or $f'$ may be unbounded if $F$ has a density),
\begin{align}
\hspace{-1cm} \mathbb{E}\left[  Ff\left(  F\right)  \right]   =\mathbb{E}\left[  \left\langle -DL^{-1}F,DF\right\rangle _\mathfrak{H}f^{\prime
}\left(  F\right)  \right]  +\mathbb{E}\left[  \left\langle -DL^{-1}%
F,\int_0^{DF}f''(F+xu)x(DF-u)du\right\rangle _\mathfrak{H}\right].
\label{eqn_integpartsWP_L}%
\end{align}

\item If $F\in\operatorname*{Dom}D^{W}\cap\operatorname*{Dom}D^{J}$ and $f$ is twice differentiable with bounded first derivative (or $f'$ may be unbounded if $F$ has a density),
\begin{align}
\hspace{-1cm} \mathbb{E}\left[  Ff\left(  F\right)  \right]   =\mathbb{E}\left[  \left\langle -DL^{-1}F,DF\right\rangle _\mathfrak{H}f^{\prime
}\left(  F\right)  \right]  +\mathbb{E}\left[  \left\langle -DL^{-1}%
F,\frac{f''(F+\theta_{\cdot}xDF)}{2}x(DF)^2\right\rangle _\mathfrak{H}\right].
\label{eqn_integpartsWP}%
\end{align}

\item If $F\in\mathbb{D}^{1,2}$ (see Remark \ref{equivalencyW-WP}) and $f$ is differentiable with
bounded derivative (or $f$ is at least a.e. differentiable if $F$ has a density),
\begin{equation}
\mathbb{E}\left[  Ff\left(  F\right)  \right]  =\mathbb{E}\left[  \left\langle
-DL^{-1}F,DF\right\rangle _\mathfrak{H}f^{\prime}\left(  F\right)  \right]  \text{.}
\label{eqn_integpartsW}%
\end{equation}
\end{itemize}

These formulas provide the link to using Malliavin calculus techniques in
solving problems related to Stein's method. Since $F=LL^{-1}F=-\delta
DL^{-1}F$, we have
\[
\mathbb{E}\left[  Ff\left(  F\right)  \right]  =\mathbb{E}\left[  -\delta
DL^{-1}F\cdot f\left(  F\right)  \right]  =\mathbb{E}\left[  \left\langle
-DL^{-1}F,Df\left(  F\right)  \right\rangle _\mathfrak{H}\right]  \text{.}%
\]
A direct application of the chain rule for Wiener-Poisson space, choosing
$k=2$ in (\ref{eqn_chainWP_L}) and (\ref{eqn_chainWP}), yields (\ref{eqn_integpartsWP_L}) and (\ref{eqn_integpartsWP}) respectively, and an
application of the respective chain rule in Wiener space yields
(\ref{eqn_integpartsW}).

\bigskip

Recall that in the Wiener-Poisson case (by Remark \ref{equivalencyW-WP}), if
$F\in\mathbb{D}^{1,2}$ then the chain rule formula (\ref{eqn_chainWP}) will
reduce to the corresponding one in Wiener space. Note that in this paper, we are assuming the target r.v. $Z$ is in $\mathbb{D}^{1,2}$. Consequently, the points we will make about $Z$ will be valid whether we are working in Wiener space or Wiener-Poisson space. The discussion of this subection then applies to both spaces.

\begin{description}
\item[Assumption A] $Z\in\mathbb{D}^{1,2}$ has mean $0$ and support $\left(  l,u\right)  $ with $-\infty\leq l<0<u\leq\infty$. The density $\rho_{\ast}$ of $Z$ is known, and it is continuous in its support. $X$ is either in $\mathbb{D}^{1,2}$ (Wiener space case) or in
$\operatorname*{Dom}D^{W}\cap\operatorname*{Dom}D^{J}$ (Wiener-Poisson space
case), and it also has mean $0$.
\end{description}

\noindent\textbf{Caution:} Notice that in the previous subsection we used
$x\in\mathbb{R}$ to denote the jump component of $z\in\mathbb{R}^{+}%
\times\mathbb{R}$ in our state space. On the other hand, we are using $Z$ to denote the 
target r.v. and $X$ the r.v. with unknown distribution. A confusion
may arise in the usage of $x$ and $X$, or $z$ and $Z$. However, we will stick with current notation for consistency with existing literature. In this regard, we urge
the reader to keep in mind that $x$ represents the size of the jump while $X$
is a random variable \underline{not} (directly) related to $x$. On the other
hand, $z$ is a jump (time of the jump, size of the jump) while $Z$ is the
target r.v. which has no jumps.

\begin{remark}
In some results, we will consider instead of $X$ a sequence $\left\{
X_{n}\right\}  $ of random variables. In this case, we have the same
assumptions (and corresponding functionals, defined below) for each $X_{n}$.
Note that for $Z\in\mathbb{D}^{1,2}$, the support necessarily has to be an
interval (see Theorem 3.1 \cite{NV}, Proposition 2.1.7 \cite{Nualart}), a
consequence that carries over to Wiener-Poisson space. The continuity assumption of the density $\rho_*$ is not strong at all, since general processes like solutions of stochastic differential equations driven by Brownian motion or (under mild conditions) fractional Brownian motion (for example see \cite{FB}) have continuous densities.
\end{remark}

Define the random variable $g_{F}=\mathbb{E}\left[  \left.  \left\langle
DF,-DL^{-1}F\right\rangle _\mathfrak{H}\right|  F\right]  $ for any Malliavin
differentiable r.v. $F$ (we will work with $F=Z$ and $F=X$ later).
Nourdin and Peccati proved that $g_{Z}\geq0$ almost surely (Proposition 3.9,
\cite{NP}). Closely related is the function
\begin{equation}
g_{\ast}\left(  z\right)  :=\mathbb{E}\left[  \left.  \left\langle
DZ,-DL^{-1}Z\right\rangle _\mathfrak{H}\right|  Z=z\right]  \text{.}
\label{eqn_gstarfirst}%
\end{equation}
Trivially, $g_{Z}=g_{\ast}\left(  Z\right)  $. Nourdin and Viens (Theorem 3.1
\cite{NV}) proved that $Z$ has a density if and only if $g_{\ast}\left(
Z\right)  >0$ a.s. Therefore, $g_{\ast}\left(  Z\right)  >0$ a.s. (Assumption
A) and $g_{\ast}\left(  z\right)  >0$ for a.e. $z\in\left(  l,u\right)  $.
Equation (\ref{eqn_integpartsW}) implies $\mathbb{E}\left[  Zf\left(
Z\right)  \right]  =\mathbb{E}\left[  g_{Z}f^{\prime}\left(  Z\right)
\right]  $. In the same manner, in Wiener space ($X\in\mathbb{D}^{1,2}$) we
have $\mathbb{E}\left[  Xf\left(  X\right)  \right]  =\mathbb{E}\left[
g_{X}f^{\prime}\left(  X\right)  \right]  $, while in Wiener-Poisson space
($X\in\operatorname*{Dom}D^{W}\cap\operatorname*{Dom}D^{J}$), this changes to
$\mathbb{E}\left[  Xf\left(  X\right)  \right]  =\mathbb{E}\left[
g_{X}f^{\prime}\left(  X\right)  \right]  +\mathbb{E}[\langle-DL^{-1}%
X,\frac{f^{\prime\prime}(X+\theta_{z}xDX)}{2}x(DX)^{2}\rangle_{H}]$. One needs
to be careful not to write $\mathbb{E}\left[  Xf\left(  X\right)  \right]
=\mathbb{E}\left[  g_{\ast}\left(  X\right)  f^{\prime}\left(  X\right)
\right]  $ or $g_{\ast}\left(  X\right)  =g_{X}$ a.s., both false, since
$g_{\ast}$ is derived from the law of $Z$ (we would need the corresponding
$g_{\ast}$ of $X$ to make it true).

\bigskip

Nourdin and Viens proved (same Theorem 3.1 \cite{NV}) that%
\begin{equation}
g_{\ast}\left(  z\right)  =\frac{%
{\displaystyle\int_{z}^{u}}
y\rho_{\ast}\left(  y\right)  dy}{\rho_{\ast}\left(  z\right)  }=-\frac{%
{\displaystyle\int_{l}^{z}}
y\rho_{\ast}\left(  y\right)  dy}{\rho_{\ast}\left(  z\right)  }\hspace{.5cm} \text{ for
a.e. }z\in\left(  l,u\right)  \text{.} \label{eqn_gstargee}%
\end{equation}

In their proof, they pointed out that $\varphi(z) := \int_z^u y\rho_{\ast}\left(  y\right)  dy = - \int_l^z y\rho_{\ast}\left(  y\right)  dy >0$ for all $z\in (l,u)$. Since $\rho_{\ast}$ is (necessarily) bounded (Assumption A), $\varphi(z)/\rho_{\ast}(z)$ is strictly positive (inside the support). From this point on, we will take $g_{\ast}$ to be either (\ref{eqn_gstarfirst}) or the version (\ref{eqn_gstargee}), whichever suits our purposes. Furthermore, we can assume that $g_{\ast}\left(  z\right)>0$ for every (and not just for almost every) $z\in\left(  l,u\right)  $. Notice that using this definition of $g_*$ we can conclude that $\bigl(g_*(z)\rho_*(z)\bigr)'=\varphi'(z)=-z\rho_*(z)$.

\bigskip

Given the density $\rho_{\ast}$ of $Z$, we can compute $g_{\ast}$ using
(\ref{eqn_gstargee}). Some examples of known distributions with their
$g_{\ast}$ are given in Table \ref{table_geestar}. Recall that $g_{\ast}(z)=0$
outside the support.

\bigskip

Conversely, one can retrieve the density $\rho_{\ast}$ given $g_{\ast}$ using the following noteworthy density formula Nourdin and Viens \cite{NV} proved:
\begin{equation}
\rho_{\ast}\left(  z\right)  =\frac{\mathbb{E}\left|  Z\right|  }{2g_{\ast
}\left(  z\right)  }\exp\left(  -\int_{0}^{z}\frac{y}{g_{\ast}\left(
y\right)  }dy\right)  \text{.} \label{eqn_density}%
\end{equation}

\begin{proposition}
\label{prop_gstarintegral}
$g_{\ast}$ necessarily satisfies the following:
\begin{equation}
\int_{l}^{0}\frac{y}{g_{\ast}\left(  y\right)  }dy=-\infty\qquad\int_{0}%
^{u}\frac{y}{g_{\ast}\left(  y\right)  }dy=\infty\text{.} \label{eqn_ggrowth}%
\end{equation}
\end{proposition}
\begin{proof}
With $\varphi(z)$ defined as before, Nourdin and Viens (Theorem 3.1 \cite{NV}) showed that 
\begin{align*}
\int_0^z \frac{y}{g_{\ast}(y)}dy = \ln \frac{\varphi(0)}{\varphi(z)}.
\end{align*} 
Since $\varphi(z) \to 0$ as $z\to u$ and as $z\to l$, the result follows.
\end{proof}

\begin{remark}
\label{rem_geepower} The necessary conditions in Proposition \ref{prop_gstarintegral} are
actually not new. Stein (Lemma VI.3 \cite{Stein}) has pointed out that these
are necessary for (\ref{eqn_gstargee}) to hold. 

\begin{itemize}
\item Suppose $g_{\ast}\left(  x\right)  =\alpha\left(  x-l\right)  ^{p}$ for some constant $\alpha>0$ and the support of $Z$ is $\left(  l,\infty\right)  $. Then $\int_{0}^{\infty
}\frac{x}{g_{\ast}\left(  x\right)  }dx=\infty$ if and only if $p\leq2$, and
$\int_{l}^{0}\frac{x}{g_{\ast}\left(  x\right)  }=-\infty$ if and only if
$1\leq p$. Similarly, if $g_{\ast}\left(  x\right)  =\alpha\left(  u-x\right)
^{q}$ over the support $\left(  -\infty,u\right)  $, $1\leq q\leq2$
necessarily. Also, if $g_{\ast}\left(  x\right)  =O\left(  x^{p}\right)  $ and
the support is $\left(  -\infty,\infty\right)  $, then $p\leq2$. If $g_{\ast
}\left(  x\right)  =\alpha\left(  u-x\right)  ^{q}\left(  x-l\right)  ^{p}$
over the support $\left(  l,u\right)  $, then $p\geq1$ and $q\geq1$ necessarily.

\item Not every $g_{\ast}$ satisfying (\ref{eqn_ggrowth}) will belong to a
random variable in $\mathbb{D}^{1,2}$. For instance, suppose $Z$ is Inverse
Gamma (see Table \ref{table_geestar}) having support $\left(  l,\infty\right)
$. For this random variable, $g_{\ast}\left(  z\right)  =\alpha\left(
z-l\right)  ^{2}$ for some $\alpha>0$. $Z$ can be shown to have finite
variance if and only if $\alpha<1$. Thus, when $\alpha\geq1$,
(\ref{eqn_ggrowth}) is satisfied but $Z\notin L^{2}\left(  \Omega\right)  $ so
$Z\notin\mathbb{D}^{1,2}$. This means that the coverage of our method is not
as extensive as we would like it to be. Hence, it is important that we check
if the target r.v. belongs to $\mathbb{D}^{1,2}$.
\end{itemize}
\end{remark}

\begin{table}[ptbh]
{
\setlength{\extrarowheight}{2.5pt}
\begin{tabular}
[c]{|l|l|l|}\hline
$Z$, with support and parameters & $\rho_{\ast}\left(  z\right)  $ & $g_{\ast
}\left(  z\right)  $\\\hline\hline
Normal & \multirow{2}{*}{$\dfrac{1}{\sqrt{2\pi}\sigma}\exp\left(
-\dfrac{z^{2}}{2\sigma}\right)  $} & \multirow{2}{*}{$\sigma^{2}$}\\
$\left(  -\infty,\infty\right)  $: $\sigma>0$ &  & \\\hline
Gamma & \multirow{2}{*}{$\dfrac{1}{s^{r}\Gamma\left(  r\right)  }\left(
z-l\right)  ^{r-1}\exp\left(  -\dfrac{z-l}{s}\right)  $} & \multirow{2}%
{*}{$s\left(  z-l\right)  $}\\
$\left(  l,\infty\right)  $: $l=-rs$, $r>0$, $s>0$ &  & \\\hline
$\chi^2$ & \multirow{2}{*}{$\dfrac{1}{2^{v/2}\Gamma\left(  v/2\right)
}\left(  z-l\right)  ^{\frac{v}{2}-1}\exp\left(  -\dfrac{z-l}{2}\right)  $} &
\multirow{2}{*}{$2\left(  z-l\right)  $}\\
$\left(  l,\infty\right)  $: $l=-v$, d.f. $v>0$ &  & \\\hline
Exponential & \multirow{2}{*}{$\lambda\exp\left(  -\lambda\left(  z-l\right)
\right)  $} & \multirow{2}{*}{$\dfrac{1}{\lambda}\left(  z-l\right)  $}\\
$\left(  l,\infty\right)  $: $l=-\dfrac{1}{\lambda}$, $\lambda>0$ &  &
\\\hline
Beta & \multirow{2}{*}{$\dfrac{1}{\beta\left(  r,s\right)  }\left(
z-l\right)  ^{r-1}\left(  1+l-z\right)  ^{s-1}$} & \multirow{2}{*}{$\dfrac
{1}{r+s}\left(  z-l\right)  \left(  1+l-z\right)  $}\\
$\left(  l,u\right)  $: $l=-\dfrac{r}{r+s}$, $u=1+l$, $r,s>0$ &  & \\\hline
Pearson Type IV & \multirow{2}{*}{$C\left(  1+\left(  z-t\right)  ^{2}\right)
^{-r}e^{ s\arctan\left(  z-t\right)  }$} & \multirow{2}{*}{$\dfrac{1}{2\left(
r-1\right)  }\left(  1+\left(  z-t\right)  ^{2}\right)  $}\\
$\left(  -\infty,\infty\right)  $: $t=-\dfrac{s}{2\left(  r-1\right)  }$,
$r>\dfrac{3}{2}$ &  & \\\hline
Student's T & \multirow{2}{*}{$\dfrac{\Gamma\left(  \frac{v+1}{2}\right)
}{\sqrt{v\pi}\Gamma\left(  \frac{v}{2}\right)  }\left(  1+\dfrac{z^{2}}
{v}\right)  ^{-\frac{v+1}{2}}$} & \multirow{2}{*}{$\dfrac{v}{v-1}\left(
1+\dfrac{z^{2}}{v}\right)  $}\\
$\left(  -\infty,\infty\right)  $: d.f. $v>2$ &  & \\\hline
Inverse Gamma & \multirow{2}{*}{$\dfrac{s^{r-1}}{\Gamma\left(  r-1\right)
}\left(  z-l\right)  ^{-r}\exp\left(  -\dfrac{s}{z-l}\right)  $} &
\multirow{2}{*}{$\dfrac{1}{r-2}\left(  z-l\right)  ^{2}$}\\
$\left(  l,\infty\right)  $: $l=-\dfrac{s}{r-1}$, $r>3$, $s>0$ &  & \\\hline
Uniform & \multirow{2}{*}{$\dfrac{1}{2u}$} & \multirow{2}{*}{$\dfrac{1}%
{2}\left(  u^{2}-z^{2}\right)  $}\\
$\left(  l,u\right)  $: $u=-l>0$ &  & \\\hline
Pareto & \multirow{2}{*}{$\dfrac{c\left(  -l\right)  ^{c}\left(  c-1\right)
^{c}}{\left(  z-cl\right)  ^{c+1}}$} & \multirow{2}{*}{$\dfrac{1}{c-1}\left(
z-l\right)  \left(  z-cl\right)  $}\\
$\left(  l,\infty\right)  $: $c>2$, $l<0$ &  & \\\hline
Laplace & \multirow{2}{*}{$\dfrac{c}{2}\exp\left(  -c\left|  z\right|
\right)  $} & \multirow{2}{*}{$\dfrac{1}{c^{2}}\left(  1+c\left|  z\right|
\right)  $}\\
$\left(  -\infty,\infty\right)  $: $c>0$ &  & \\\hline
Lognormal & $\dfrac{-l}{\sqrt{2\pi}\sigma e^{2\delta}}\exp\left(  -\dfrac
{1}{2}\left[  p\left(  z\right)  +\sigma\right]  ^{2}\right)  $ & $\sigma
e^{2\delta}\exp\left(  \frac{1}{2}\left(  p\left(  z\right)  +\sigma\right)
^{2}\right)  $\\
$\left(  l,\infty\right)  $: $l=-\exp\left(  \delta+\frac{1}{2}\sigma
^{2}\right)  $ & where $p\left(  z\right)  =\dfrac{\ln\left(  z-l\right)
-\delta}{\sigma}$ & $\times{\displaystyle\int_{p\left(  z\right)  -\sigma
}^{p\left(  z\right)  }} e^{-s^{2}/2}ds$\\\hline
\end{tabular}
}
\caption{Common random variables $Z$ with their $\rho_{\ast}$ and $g_{\ast}$}
\label{table_geestar}
\end{table}

Parallel to $g_{\ast}$ of $Z$, we may define a corresponding object for $X$
but we will have no use for it. In fact, we typically won't have access to the
density of $X$; if it was known otherwise, one may characterize $X$ without
having to approximate it by $Z$. We will only assume properties (for $X$)
amenable to the Malliavin calculus that would allow us to define $g_{X}$.

\bigskip

Let $G_{\ast}\left(  z\right)  =\int_{l}^{z}g_{\ast}\left(
y\right)  dy$ be the indefinite integral of $g_{\ast}$ (assuming $g_{\ast}\in L^1(l,u)$).
Consider the Wiener space case ($X\in\mathbb{D}^{1,2}$) and suppose $\left\|
g_{\ast}\right\|  _{\infty}<\infty$ or $X$ has a density. If we take $f=G_{\ast}$ in (\ref{eqn_integpartsW}), then%
\begin{equation}
\mathbb{E}\left[  g_{X}g_{\ast}\left(  X\right)  \right]  =\mathbb{E}\left[  G_{\ast
}\left(  X\right)  X\right]  \text{.} \label{eqn_biggstar}%
\end{equation}

\begin{description}
\item[Assumption A$^{\prime}$] Along with Assumption A, either $\left\|
g_{\ast}\right\|  _{\infty}<\infty$ or $X$ has a density.
\end{description}

\begin{proposition}\label{lem_expec}(Moments formula)
\begin{itemize}

\item {\bf Wiener space:} $\mathbb{E}\left[  F^{r+1}\right]  =r\mathbb{E}\left[
F^{r-1}g_{F}\right]$, provided the expectations exist.
\begin{enumerate}

\item If $g_*(Z)=g_{Z}$ is a polynomial in $Z$, i.e. $g_*(z)=\sum_{k=0}^{m}a_{k}z^{k}%
$, then $\mathbb{E}\left[  Z^{r+1}\right]  =\sum_{k=0}^{m}ra_{k}%
\mathbb{E}\left[  Z^{r+k-1}\right]  $.

\item If $X=I_{q}\left(  g\right)  $, then $\mathbb{E}\left[  X^{r+1}\right]
=\frac{r}{q}\mathbb{E}\left[  X^{r-1}\left\|  DX\right\|  _\mathfrak{H}^{2}\right]  $.

\end{enumerate}

\item {\bf Wiener-Poisson space:} $\mathbb{E}\left[  F^{r+1}\right]  = r\mathbb{E}\left[  F^{r-1}g_F\right]  +\frac{r\left(r-1\right)}{2}\mathbb{E}\left[\langle-DL^{-1}F,x(F+\theta_{\cdot}xDF)^{r-2}(DF)^2\rangle_\mathfrak{H}\right]$
\begin{enumerate}

\item If $X=I_{q}\left(  g\right)  $, then $\mathbb{E}\left[  X^{r+1}\right]  =\frac{r}{q}\mathbb{E}\left[  X^{r-1}\left\|  DX\right\|  _\mathfrak{H}^{2}\right]  +\frac{r\left(  r-1\right)  }{2q}\mathbb{E}\left[  \left\langle x\left(  DX\right)  ^{3},\left(X+\theta_{\cdot}xDX\right)  ^{r-2}\right\rangle _\mathfrak{H}\right]  \text{.}$

\end{enumerate}

\end{itemize}
\end{proposition}
\begin{proof}
Simply let $f\left(y\right)=y^{n}$ in (\ref{eqn_integpartsWP}) and (\ref{eqn_integpartsW}), taking into account the
fact that $-DL^{-1}F=\frac{1}{q}DF$ when $F=I_q(g)$.
\end{proof}

\section{\label{sec_Stein}Stein's method and the Stein equation}

Stein's method is a set of procedures that is often used to measure distances
between random variables such as $X$ and $Z$. More precisely, we're measuring
the distance between the laws of $X$ and $Z$. These distances take the form
\begin{equation}
d_{\mathcal{H}}\left(  X,Z\right)  =\sup_{h\in\mathcal{H}}\left|
\mathbb{E}\left[  h\left(  X\right)  \right]  -\mathbb{E}\left[  h\left(
Z\right)  \right]  \right|  \label{eqn_dmetric}%
\end{equation}
where $\mathcal{H}$ is a suitable family of functions. If we take
$\mathcal{H}_{W}=\left\{  h:\left\|  h\right\|  _{L}\leq1\right\}  $ where
$\left\|  \cdot\right\|  _{L}$ is the Lipschitz seminorm, then $d_{W}%
=d_{\mathcal{H}_{W}}$ is called Wasserstein distance. The bounded Wasserstein
(Fortet-Mourier) distance corresponds to $\mathcal{H}_{FM}=\left\{  h:\left\|
h\right\|  _{L}+\left\|  h\right\|  _{\infty}\leq1\right\}  $. Clearly,
$d_{FM}\leq d_{W}$. $d_{FM}$ is important because it metrizes convergence in
distribution: $d_{FM}\left(  X_{n},Z\right)  \rightarrow0$ if and only if
$X_{n}\overset{\text{Law}}{\longrightarrow}Z$. $d_{W}$ on the other hand
induces a topology stronger than that of convergence in distribution. 

\bigskip

Nourdin and Peccati \cite{NP} mentioned other useful metrics. We have the Total Variation distance when $\mathcal{H}_{TV}=\left\{  \mathbf{1}_B: B \text{ is Borel}\right\}$ and the Kolmogorov distance when $\mathcal{H}_{K}=\left\{  \mathbf{1}_{(-\infty,z]}: z\in\mathbb{R}\right\}$. The latter for example is suited for the analysis of probability tails. However, in this paper, we will only consider $d_W$ and $d_{FM}$ as we try to find bounds for $d_{\mathcal{H}}\left(  X,Z\right)$ by exploiting properties of Lipschitz functions $h\in\mathcal{H}$.

\bigskip

A Stein equation is at the root of Stein's method. Given $Z$ and a test
function $h$, the Stein equation is the differential equation
\begin{equation}
g_{\ast}\left(  x\right)  f^{\prime}\left(  x\right)  -xf\left(  x\right)
=h\left(  x\right)  -\mathbb{E}\left[  h\left(  Z\right)  \right]
\label{eqn_steineqn}%
\end{equation}
having solution $f=f_{h}$. If the law of $X$ is ``close'' to the law of $Z$,
then we expect $\mathbb{E}\left[  h\left(  X\right)  \right]  -\mathbb{E}%
\left[  h\left(  Z\right)  \right]  $ to be close to $0$, for $h$ belonging to
a large class of functions. Consequently, $\mathbb{E}\left[  g_{\ast}\left(
X\right)  f^{\prime}\left(  X\right)  -Xf\left(  X\right)  \right]  $ would
have to be close to $0$. In fact, subject to certain technical conditions, the
left-hand side of equation (\ref{eqn_steineqn}) provides a characterization of
the law of $Z$: $\mathbb{E}\left[  g_{\ast}\left(  X\right)  f^{\prime}\left(
X\right)  -Xf\left(  X\right)  \right]  =0$ if and only if $X\overset
{\text{Law}}{=}Z$ (in the equation, information about the law of $Z$ is coded
in $g_{\ast}$). The following proposition states this result in its precise
form. For a quick proof, see Proposition 6.4 in \cite{NP}. The first statement
is Lemma 1 in \cite{Stein} by Stein.

\begin{lemma}
(Stein's Lemma)
\begin{enumerate}
\item If $f$ is continuous, piecewise continuously differentiable, and
$\mathbb{E}\left[  g_{\ast}\left(  Z\right)  \left|  f^{\prime}\left(
Z\right)  \right|  \right]  <\infty$, then
\begin{equation}
\mathbb{E}\left[  g_{\ast}\left(  Z\right)  f^{\prime}\left(  Z\right)
-Zf\left(  Z\right)  \right]  =0\text{.}%
\end{equation}

\item If for every differentiable $f$, $x\mapsto\left|  g_{\ast}\left(
x\right)  f^{\prime}\left(  x\right)  \right|  +\left|  xf\left(  x\right)
\right|  $ is bounded and
\begin{equation}
\mathbb{E}\left[  g_{\ast}\left(  X\right)  f^{\prime}\left(  X\right)
-Xf\left(  X\right)  \right]  =0\text{,}%
\end{equation}
then $X\overset{\text{Law}}{=}Z$.
\end{enumerate}
\end{lemma}

\bigskip

\noindent Let $\mathcal{H=H}_{FM}$ or $\mathcal{H=H}_{W}$. Using (\ref{eqn_steineqn}) on
(\ref{eqn_dmetric}), we have
\begin{equation}
d_{\mathcal{H}}\left(  X,Z\right)  \leq\sup_{f\in\mathscr{F}_{\mathcal{H}}}\left|
\mathbb{E}\left[  g_{\ast}\left(  X\right)  f^{\prime}\left(  X\right)
-Xf\left(  X\right)  \right]  \right|  \label{eqn_firstbound}%
\end{equation}
where the sup is taken over the family $\mathscr{F}_{\mathcal{H}}$ of all Stein equation
solutions $f$ corresponding to $h\in\mathcal{H}$. Here the integration by
parts formulas (\ref{eqn_integpartsWP_L}) and (\ref{eqn_integpartsW}) allow us
to rewrite the term $\mathbb{E}[Xf(X)]$ in terms of the derivatives of $f$ and
the r.v. $g_{X}$, as we pointed out before. For instance, in Wiener space,%
\begin{equation}
d_{\mathcal{H}}\left(  X,Z\right)  \leq\sup_{f\in\mathscr{F}_{\mathcal{H}}}\left|
\mathbb{E}\left[  g_{\ast}\left(  X\right)  f^{\prime}\left(  X\right)
-g_{X}f^{\prime}\left(  X\right)  \right]  \right|  =\sup_{f\in\mathscr{F}_{\mathcal{H}}%
}\left|  \mathbb{E}\left[  f^{\prime}\left(  X\right)  \left(  g_{\ast}\left(
X\right)  -g_{X}\right)  \right]  \right|  \text{.} \label{eqn_basicbound}%
\end{equation}
Thus, to ensure that the distance between $X$ and $Z$ is small, $g_{\ast
}\left(  X\right)  $ should be close to $g_{X}$. We also need to have a good
control of $f^{\prime}\left(  X\right)  $. One way of addressing this, taking
note of Corollary 6.5 in \cite{NP}, is by assuming a universal bound for
$\mathbb{E}\left[  f^{\prime}\left(  X\right)  ^{2}\right]  $ for all
$f\in\mathscr{F}_{\mathcal{H}}$ since
\begin{equation}
d_{\mathcal{H}}\left(  X,Z\right)  \leq\sqrt{\sup_{f\in\mathscr{F}_{\mathcal{H}}}%
\mathbb{E}\left[  f^{\prime}\left(  X\right)  ^{2}\right]  }\times
\sqrt{\mathbb{E}\left[  \left(  g_{\ast}\left(  X\right)  -g_{X}\right)
^{2}\right]  }\text{.}%
\end{equation}

The first factor is intractable since it requires us to consider conditions on
$X$ in relation to all members $f$ of the family $\mathscr{F}_{\mathcal{H}}$. If however we
have a uniform bound for $f^{\prime}$, then we can avoid imposing an
additional restriction on $X$. In this case, we only need worry about how
close $g_{\ast}\left(  X\right)  $ is to $g_{X}$ in $L^{2}\left(
\Omega\right)  $. In fact, such a bound allows us to just consider how close $g_{\ast}\left(  X\right)  $ is to $g_{X}$ in $L^{1}\left(  \Omega\right)  $. It is then interesting to see how information about
the law of $Z$ is contained in its Malliavin derivative. Questions of how
close the law of $X$ is to that of $Z$ is passed on to how close
$\mathbb{E}\left[  \left.  \left\langle DZ,-DL^{-1}Z\right\rangle_\mathfrak{H} \right|
Z=X\right]  $ is to $\mathbb{E}\left[  \left.  \left\langle DX,-DL^{-1}%
X\right\rangle_\mathfrak{H} \right|  X\right]  $. Notice though that this discussion needs
to be modified slightly in Wiener-Poisson space, since the integration by
parts formula (\ref{eqn_integpartsWP_L}) involves also the second derivative. Thus, we need to control (in a uniform way) both the first and second derivatives of the solution of the Stein equation. Due to this extra requirement, as will be seen later,
we will not be able to apply our tools to as wide a scope of target r.v. $Z$,
as we would be able to do in Wiener space.

\subsection{Bound for $f^{\prime}$}\label{fprime}

{\it {\bf The Normal case in Wiener space:}\\
If $Z$ is standard Normal ($g_*(z)=1$), the Stein equation is $f^{\prime}\left(  x\right)
-xf\left(  x\right)  =h\left(  x\right)  -\mathbb{E}\left[  h\left(  Z\right)
\right]  $ and it has solution $f\left(  x\right)  =e^{x^{2}/2}\int_{-\infty
}^{x}\left[  h\left(  y\right)  -\mathbb{E}\left[  h\left(  Z\right)  \right]
\right]  e^{-y^{2}/2}dy$. Stein proved (Lemma
II.3 in \cite{Stein}) that $\left\|  f^{\prime}\right\|  _{\infty}%
\leq2\left\|  h-\mathbb{E}\left[  h\left(  Z\right)  \right]  \right\|
_{\infty}$. In fact, $\left\|  f^{\prime}\right\|  _{\infty}\leq\min\left\{
2\left\|  h-\mathbb{E}\left[  h\left(  Z\right)  \right]  \right\|  _{\infty
},4\left\|  h^{\prime}\right\|  _{\infty}\right\}  $ (see Lemma 2.3
\cite{CB}). For $h\in\mathcal{H}_{FM}$, $\left\|  f^{\prime}\right\|
_{\infty}\leq4$. It follows from  (\ref{eqn_basicbound}) that $d_{FM}\left(  X,Z\right)  \leq k\mathbb{E}\left[
\left|  1-g_{X}\right|  \right]  \leq k\sqrt{\mathbb{E}\left[  \left(
1-g_{X}\right)  ^{2}\right]  }$ with $k=4$. Similar estimates for
$h\in\mathcal{H}_{W}$ lead to a bound for $d_{W}$ of the same form but with
$k=1$ (Lemma 4.2 \cite{C}, Lemma 1.2 \cite{NP}). How close the law of $X$ is
to the standard Normal law depends on how close $g_{X}$ is to $g_{Z}=1$ (in the
$L^{1}$ sense).}

\bigskip

\noindent In the general case, the Stein equation (\ref{eqn_steineqn}) has solution
\begin{equation}
f\left(  x\right)  =\frac{1}{g_{\ast}\left(  x\right)  \rho_{\ast}\left(
x\right)  }\int_{l}^{x}\left[  h\left(  y\right)  -m_{h}\right]  \rho_{\ast
}\left(  y\right)  dy=\frac{-1}{g_{\ast}\left(  x\right)  \rho_{\ast}\left(
x\right)  }\int_{x}^{u}\left[  h\left(  y\right)  -m_{h}\right]  \rho_{\ast
}\left(  y\right)  dy%
\end{equation}
for
$x\in\left(  l,u\right)  $, where $m_{h}:=\mathbb{E}\left[  h\left(  Z\right)  \right]  $.

\bigskip

The proof of the bound for $f^{\prime}$ when $Z$ is Normal can be adapted to
find a constant bound for $g_{\ast}f^{\prime}$ in the non-Normal case. If
$g_{\ast}$ is uniformly bounded below by a positive number, we easily get a uniform
bound for $f^{\prime}$. Unfortunately, this is not always the case. In Table
\ref{table_geestar} we can see several examples of target r.v.'s for which
$g_{\ast}$ can get arbitrarily close to $0$ in its support (for example, when
$Z$ is Gamma and $g_{\ast}\left(  z\right)  =s\left(  z-l\right)  _{+}$).
Kusuoka and Tudor in \cite{KT} (Proposition 3) proved the following
proposition to address this issue. We state it in the following form using
notation and assumptions we have set.

\begin{lemma}
\label{prop_tudorbound}Suppose we have the following conditions on $g_{\ast}$.

\begin{enumerate}

\item If $u<\infty$, then $\varliminf_{x\rightarrow u}g_{\ast}\left(
x\right)  /\left(  u-x\right)  >0$.

\item If $l>-\infty$, then $\varliminf_{x\rightarrow l}g_{\ast}\left(
x\right)  /\left(  x-l\right)  >0$.

\item If $u=\infty$, then $\varliminf_{x\rightarrow u}g_{\ast}\left(
x\right)  >0$.

\item If $l=-\infty$, then $\varliminf_{x\rightarrow l}g_{\ast}\left(
x\right)  >0$.

\end{enumerate}
Then the solution $f$ of the Stein equation (\ref{eqn_steineqn}), for a given
test function $h$ with $\left\|  h\right\|  _{\infty}<\infty$ and $\left\|
h^{\prime}\right\|  _{\infty}<\infty$, has derivative bounded as follows:%
\begin{equation}
\left\|  f^{\prime}\right\|  _{\infty}\leq k\left(  \left\|  h\right\|
_{\infty}+\left\|  h^{\prime}\right\|  _{\infty}\right)
\label{eqn_fprimeboundtudor}%
\end{equation}
where the constant $k$ depends on $Z$ alone, and not on $h$.
\end{lemma}

Unfortunately, conditions \textbf{1} and \textbf{2} are too restrictive.
Consider for instance a r.v. $Z$ with support $\left(  l,\infty\right)  $ and $g_{\ast}\left(
x\right)  =\alpha\left(  x\right)  \left(  x-l\right)  ^{p}$, where
$\alpha(x)$ is uniformly bounded below by some $\alpha_{0}>0$. From Remark
\ref{rem_geepower}, $1\leq p\leq2$ necessarily. Among all $g_{\ast}$ of this
form, Lemma \ref{prop_tudorbound} is thus only able to assure the needed
boundedness of $f^{\prime}$ when $p=1$. For instance, when $Z$ is Inverse Gamma or Lognormal, condition \textbf{2} fails (see the corresponding $g_{\ast}$ in Table \ref{table_geestar}). This stresses the need for less
restrictive conditions on $g_{\ast}$ that would allow us to include these
cases and much more. The first requirement in order to achieve this is a good representation of the derivative $f'$. 

\begin{proposition}
\label{prop_fprime}For $x\in\left(  l,u\right)  $, the derivative $f^{\prime}$
of the solution of the Stein equation (\ref{eqn_steineqn}) is%
\[
f^{\prime}\left(  x\right)  =\frac{1}{g_{\ast}^{2}\left(  x\right)  \rho
_{\ast}\left(  x\right)  }\int_{x}^{u}\int_{l}^{x}\left[  1-\Phi\left(
s\right)  \right]  \Phi\left(  t\right)  \left[  h^{\prime}\left(  t\right)
-h^{\prime}\left(  s\right)  \right]  dtds\text{.}%
\]
where $\Phi(x)=\int_l^x\rho_*(t)dt$ is the cumulative distribution function of $Z$.
\end{proposition}
\begin{proof}
First,%
\begin{align*}
h\left(  x\right)  -m_{h}  &  =\int_{l}^{x}\left[  h\left(  x\right)
-h\left(  s\right)  \right]  \rho_{\ast}\left(  s\right)  ds+\int_{x}%
^{u}\left[  h\left(  x\right)  -h\left(  s\right)  \right]  \rho_{\ast}\left(
s\right)  ds\\
&  =\int_{l}^{x}\left[  \int_{s}^{x}h^{\prime}\left(  t\right)  dt\right]
\rho_{\ast}\left(  s\right)  ds-\int_{x}^{u}\left[  \int_{x}^{s}h^{\prime
}\left(  t\right)  dt\right]  \rho_{\ast}\left(  s\right)  ds\\
&  =\int_{l}^{x}\left[  \int_{l}^{t}\rho_{\ast}\left(  s\right)  ds\right]
h^{\prime}\left(  t\right)  dt-\int_{x}^{u}\left[  \int_{t}^{u}\rho_{\ast
}\left(  s\right)  ds\right]  h^{\prime}\left(  t\right)  dt\\
&  =\int_{l}^{x}\Phi\left(  t\right)  h^{\prime}\left(  t\right)  dt-\int
_{x}^{u}\left[  1-\Phi\left(  t\right)  \right]  h^{\prime}\left(  t\right)
dt
\end{align*}
and so%
\begin{align*}
g_{\ast}\left(  x\right)  \rho_{\ast}\left(  x\right)  f\left(  x\right)   &
=\int_{l}^{x}\left[  h\left(  y\right)  -m_{h}\right]  \rho_{\ast}\left(
y\right)  dy\\
&  =\int_{l}^{x}\left[  \int_{l}^{y}\Phi\left(  t\right)  h^{\prime}\left(
t\right)  dt\right]  \rho_{\ast}\left(  y\right)  dy-\int_{l}^{x}\left[
\int_{y}^{u}\left[  1-\Phi\left(  t\right)  \right]  h^{\prime}\left(
t\right)  dt\right]  \rho_{\ast}\left(  y\right)  dy\\
&  =\int_{l}^{x}\left[  \int_{t}^{x}\rho_{\ast}\left(  y\right)  dy\right]
\Phi\left(  t\right)  h^{\prime}\left(  t\right)  dt\\
&  \quad\quad-\int_{l}^{x}\left[  \int_{l}^{t}\rho_{\ast}\left(  y\right)
dy\right]  \left[  1-\Phi\left(  t\right)  \right]  h^{\prime}\left(
t\right)  dt-\int_{x}^{u}\left[  \int_{l}^{x}\rho_{\ast}\left(  y\right)
dy\right]  \left[  1-\Phi\left(  t\right)  \right]  h^{\prime}\left(
t\right)  dt\\
&  =\int_{l}^{x}\left[  \Phi\left(  x\right)  -\Phi\left(  t\right)  \right]
\Phi\left(  t\right)  h^{\prime}\left(  t\right)  dt-\int_{l}^{x}\Phi\left(
t\right)  \left[  1-\Phi\left(  t\right)  \right]  h^{\prime}\left(  t\right)
dt-\int_{x}^{u}\Phi\left(  x\right)  \left[  1-\Phi\left(  t\right)  \right]
h^{\prime}\left(  t\right)  dt.
\end{align*}
Cancelling some terms and solving for $f$,
\begin{align}\label{eqn_f}
f\left(  x\right)    =-\frac{1-\Phi\left(  x\right)  }{g_{\ast}\left(
x\right)  \rho_{\ast}\left(  x\right)  }\int_{l}^{x}\Phi\left(  t\right)
h^{\prime}\left(  t\right)  dt-\frac{\Phi\left(  x\right)  }{g_{\ast}\left(
x\right)  \rho_{\ast}\left(  x\right)  }\int_{x}^{u}\left[  1-\Phi\left(
t\right)  \right]  h^{\prime}\left(  t\right)  dt\text{.}
\end{align}
Observe that if $x<0$,
$$0=\mathbb{E}[Z]=\int_l^xt\rho_*(t)dt+\int_x^ut\rho_*(t)dt\leq x\Phi(x)+g_*(x)\rho_*(x)$$
while if $x>0$, 
$$0=\mathbb{E}[Z]=\int_l^xt\rho_*(t)dt+\int_x^ut\rho_*(t)dt\geq -g_*(x)\rho_*(x)+x[1-\Phi(x)].$$
Therefore, $0\leq-x\Phi\left(  x\right)  \leq g_{\ast}(x)\rho_{\ast
}(x)\rightarrow0$ as $x\rightarrow l$ and $0\leq x\left[
1-\Phi\left(  x\right)  \right]  \leq g_{\ast}(x)\rho_{\ast}(x)\rightarrow
0$ as $x\rightarrow u$. When we then integrate by parts,%
\begin{align}
\int_{l}^{x}\Phi\left(  t\right)  dt  &  =t\Phi\left(  t\right)  \bigg|%
_{l}^{x}-\int_{l}^{x}t\rho_{\ast}\left(  t\right)  dt=x\Phi\left(  x\right)
+g_{\ast}\left(  x\right)  \rho_{\ast}\left(  x\right) \label{eqn_phi1}\\
\int_{x}^{u}\left[  1-\Phi\left(  t\right)  \right]  dt  &  =t\left[
1-\Phi\left(  t\right)  \right]  \bigg|_{x}^{u}+\int_{x}^{u}t\rho_{\ast
}\left(  t\right)  dt=-x\left[  1-\Phi\left(  x\right)  \right]  +g_{\ast
}\left(  x\right)  \rho_{\ast}\left(  x\right)  \text{.} \label{eqn_phi2}%
\end{align}
Finally,%
\begin{align*}
g_{\ast}\left(  x\right)  f^{\prime}\left(  x\right)   &  =xf\left(  x\right)
+h\left(  x\right)  -m_{h}\\
&  =\left(  -\frac{x\left[  1-\Phi\left(  x\right)  \right]  }{g_{\ast}\left(
x\right)  \rho_{\ast}\left(  x\right)  }+1\right)  \int_{l}^{x}\Phi\left(
t\right)  h^{\prime}\left(  t\right)  dt-\left(  \frac{x\Phi\left(  x\right)
}{g_{\ast}\left(  x\right)  \rho_{\ast}\left(  x\right)  }+1\right)  \int
_{x}^{u}\left[  1-\Phi\left(  t\right)  \right]  h^{\prime}\left(  t\right)
dt\\
&  =\frac{1}{g_{\ast}\left(  x\right)  \rho_{\ast}\left(  x\right)  }\int
_{x}^{u}\left[  1-\Phi\left(  s\right)  \right]  ds\int_{l}^{x}\Phi\left(
t\right)  h^{\prime}\left(  t\right)  dt-\frac{1}{g_{\ast}\left(  x\right)
\rho_{\ast}\left(  x\right)  }\int_{l}^{x}\Phi\left(  t\right)  dt\int_{x}%
^{u}\left[  1-\Phi\left(  s\right)  \right]  h^{\prime}\left(  s\right)  ds
\end{align*}
which leads to the given form of $f^{\prime}$.
\end{proof}

\bigskip

The bound (\ref{eqn_fprimeboundtudor}) is not directly suited for $d_{W}$ where we don't have a prescribed bound on $||h||_{\infty}$. A workaround, as pointed out in \cite{KT}, is that for each $h\in$ $\mathcal{H}_{W}$, we pass on the analysis to a sequence $\left\{  h_{n}\right\}  $ converging to $h$ uniformly in every compact set, where $\left\{  h_{n}\right\}  \subset\left\{  h\in C_{0}^{1}:\left\|
h^{\prime}\right\|  _{\infty}\leq1\right\}  $. However, with the help of the previous lemma, we can overcome this complication by giving a bound for $f'$ in terms of only $||h^{\prime}||_{\infty}$. Recall that if $h$ is Lipschitz, it is a.e. differentiable and $\left\|  h^{\prime}\right\|  _{\infty}%
\leq\left\|  h\right\|  _{L}$. Thus, the upper bound obtained here is
immediately well suited for all $f\in\mathscr{F}_{FM}$ and for all
$f\in\mathscr{F}_{W}$.


\begin{theorem}
\label{prop_fprimebound}If applicable, assume conditions \textbf{3} and
\textbf{4} from Lemma \ref{prop_tudorbound}. Suppose there exists a
positive function$\ \tilde{g}\in \mathcal{C}^{1}(l,u)$ such that
\begin{enumerate}

\item $0<\varliminf_{x\rightarrow u}g_{\ast}\left(  x\right)/\tilde{g}\left(  x\right)\leq\varlimsup_{x\rightarrow u}g_{\ast}\left(  x\right)  /\tilde{g}\left(  x\right)  <\infty\ $ and $\ \tilde{g}'(u^-):=\lim_{x\to u^-}\tilde{g}'(x)\in \boldsymbol{R}$ exists.\footnote{$\boldsymbol{R}$ stands for the extended real numbers, i.e. $\boldsymbol{R}=[-\infty,\infty].$}

\item $0<\varliminf_{x\rightarrow l}g_{\ast}\left(  x\right)/\tilde{g}\left(  x\right)\leq\varlimsup_{x\rightarrow l}g_{\ast}\left(  x\right)  /\tilde{g}\left(  x\right)  <\infty\ $ and $\ \tilde{g}'(l^+):=\lim_{x\to l^+}\tilde{g}'(x)\in \boldsymbol{R}$ exists.

\end{enumerate}
Then the solution $f$ of the Stein equation (\ref{eqn_steineqn}), for a given
test function $h$ with $\left\|h^{\prime}\right\|  _{\infty}<\infty$, has derivative bounded as follows:%
\begin{equation}
\left\|  f^{\prime}\right\|  _{\infty}\leq k\left\|  h^{\prime}\right\|  _{\infty}
\label{eqn_fprimebound}
\end{equation}
where the constant $k$ depends on $Z$ alone, and not on $h$.
\end{theorem}
\begin{proof}
First note that from Proposition \ref{prop_fprime},
\begin{align}\label{eqfprime}
\left|  f^{\prime}\left(  x\right)  \right|  \leq\frac{2\left\|  h^{\prime
}\right\|  _{\infty}}{g_{\ast}^{2}\left(  x\right)  \rho_{\ast}\left(
x\right)  }\int_{x}^{u}\left[  1-\Phi\left(  s\right)  \right]  ds\int_{l}%
^{x}\Phi\left(  t\right)  dt\text{.}%
\end{align}
Fix $l'$ and $u'$ s.t. $l<l'<0<u'<u$. Since $g_*(x)\rho_*(x)$ is continuous and strictly positive on $[l',u']$, it attains its minimum $m:=\inf_{\left[  l^{\prime},u^{\prime}\right]  }g_*(x)\rho_*(x) >0$ on this compact set. Also by continuity of the density $M:=\sup_{\left[  l^{\prime},u^{\prime}\right]} \rho_{\ast}(x)<\infty$, and $g_{\ast}\left(x\right)  =\frac{g_{\ast}(x)\rho_{\ast}(x)}{\rho_{\ast}\left(  x\right)  }\geq\frac{m}{M}>0$ on $\left[  l^{\prime},u^{\prime}\right]  $, so $g_{\ast}^{2}(x)\rho_{\ast}(x)\geq\frac{m^{2}}{M}$. By the continuity and positivity of $I_1(x):=\int_{x}^{u}\left[  1-\Phi\left(  s\right)  \right]  ds$ and $I_2(x):=\int_{l}^{x}\Phi\left(  t\right)  dt$ we conclude that $K:=\sup_{\left[  l^{\prime},u^{\prime}\right]  }\left(
I_{1}(x)\vee I_{2}(x)\right)  <\infty$. By (\ref{eqfprime}), $|f^{\prime}(x)|\leq\frac{2MK^2}{m^2}\left\| h'\right\|  _{\infty}$ on $\left[  l^{\prime},u^{\prime}\right]  $.

Since $l^{\prime}$ and $u^{\prime}$ were arbitrarily chosen, we only need to
prove now that $\varlimsup_{x\rightarrow l}\left|  f^{\prime}\left(  x\right)
\right|  \leq k_{1}||h^{\prime}||_{\infty}$ and $\varlimsup_{x\rightarrow
u}\left|  f^{\prime}\left(  x\right)  \right|  \leq k_{2}||h^{\prime
}||_{\infty}$ for some finite constants $k_{1}$ and $k_{2}$. Due to the
symmetry of the arguments it suffices to prove just one of these limits.
Suppose $l^{\prime}$ was chosen small enough so that $\tilde{g}\in
\mathcal{C}^{1}\left(  l,l^{\prime}\right)  $, and for some constants $0<c\leq C<\infty
$, $cg_{\ast}\left(  x\right)  \leq\tilde{g}\left(  x\right)  \leq Cg_{\ast
}\left(  x\right)  $ on $\left(  l,l^{\prime}\right)  $.

\begin{itemize}

\item {\bf Case 1:} $\boldsymbol{l>-\infty}$.\\
We show that the limit of the right-hand side of (\ref{eqfprime}) is finite
as $x\rightarrow l$. Note that in this case, $\int_{x}^{u}\left[  1-\Phi\left(
s\right)  \right]  ds=g_{\ast}\left(  x\right)  \rho_{\ast}\left(  x\right)
-x\left[  1-\Phi\left(  x\right)  \right]  \rightarrow\left|  l\right|  $.
By L'H\^{o}pital's rule,
\begin{align*}
\varlimsup_{x\rightarrow l}\left|  f^{\prime}\left(  x\right)  \right|   &
\leq2\left\|  h^{\prime}\right\|  _{\infty}\left|  l\right|  \varlimsup
_{x\rightarrow l}\frac{C\int_{l}^{x}\Phi\left(  t\right)  dt}{\tilde{g}\left(
x\right)  g_{\ast}\left(  x\right)  \rho_{\ast}\left(  x\right)  }%
\leq2\left\|  h^{\prime}\right\|  _{\infty}\left|  l\right|  C\varlimsup
_{x\rightarrow l}\frac{\Phi\left(  x\right)  }{-x\tilde{g}\left(  x\right)
\rho_{\ast}\left(  x\right)  +\tilde{g}^{\prime}\left(  x\right)  g_{\ast
}\left(  x\right)  \rho_{\ast}\left(  x\right)  }\\
&  \leq2\left\|  h^{\prime}\right\|  _{\infty}\left|  l\right|  C\varlimsup
_{x\rightarrow l}\frac{\Phi\left(  x\right)  }{\left[  -cx+\tilde{g}^{\prime
}\left(  x\right)  \right]  g_{\ast}\left(  x\right)  \rho_{\ast}\left(
x\right)  }\leq\frac{2\left\|  h^{\prime}\right\|  _{\infty}\left|  l\right|
C}{\tilde{g}^{\prime}\left(  l^+\right)  -cl}\varlimsup_{x\rightarrow l}%
\frac{\rho_{\ast}\left(  x\right)  }{-x\rho_{\ast}\left(  x\right)  }%
=\frac{2\left\|  h^{\prime}\right\|  _{\infty}C}{\tilde{g}^{\prime}\left(
l^+\right)  -cl}\text{.}%
\end{align*}
Since $\tilde{g}\left(  l^+\right):= \lim_{z\to l^+}\tilde{g}_*(z) =0$ and $\tilde{g}\geq0$, we may assume
$l^{\prime}$ is small enough so $\tilde{g}^{\prime}\geq0$ on $\left(
l,l^{\prime}\right)  $. Consequently, $\tilde{g}^{\prime}\left(  l^+\right)
\neq cl<0$.

\item {\bf Case 2:} $\boldsymbol{l=-\infty}$.\\
Since $\varliminf
_{x\rightarrow-\infty}g_{\ast}\left(  x\right)  >0$, we may suppose
$l^{\prime}$ is small enough so that for some constant $m_{0}>0$, $g_{\ast
}\left(  x\right)  \geq m_{0}$ over $\left(  -\infty,l^{\prime}\right)  $.
\begin{align*}
\varlimsup_{x\rightarrow -\infty}\left|  f^{\prime}\left(  x\right)  \right|   &
\leq 2||h'||_\infty\varlimsup
_{x\rightarrow -\infty}\frac{\bigl(g_{\ast}\left(  x\right)  \rho_{\ast}\left(  x\right)
-x\left[  1-\Phi\left(  x\right)  \right]\bigr)\int_{-\infty}^{x}\Phi\left(  t\right)  dt}{g^2_{\ast}\left(  x\right)  \rho_{\ast}\left(  x\right)  }\\
&\leq2\left\|  h^{\prime}\right\|  _{\infty}\biggl(\varlimsup
_{x\rightarrow -\infty}\frac{\int_{-\infty}^{x}\Phi\left(  t\right)  dt}{m_0}+\varlimsup_{x\rightarrow -\infty}\frac{-x\int_{-\infty}^{x}\Phi\left(  t\right)  dt}{g^2_{\ast}\left(  x\right)  \rho_{\ast}\left(  x\right)  }
\biggr) \\
&=2\left\|  h^{\prime}\right\|  _{\infty}\varlimsup_{x\rightarrow -\infty}\frac{|x|\int_{-\infty}^{x}\Phi(t)dt}{g^2_*(x)\rho_*(x)}
\end{align*}
There are two subcases to consider depending on the behavior of $\tilde{g}(x)$ as $x \to -\infty$. From the continuity of $\tilde{g}$ and the existence of $\tilde{g}'(l^+)$,  $L:=\lim_{x\to-\infty}\tilde{g}(x)$ necessarily exists. If $L<\infty$, then $\lim_{x\to-\infty}\frac{\tilde{g}(x)}{|x|}=0$. If $L=\infty$, then by L'H\^{o}pital's rule, $\lim_{x\to-\infty}\frac{\tilde{g}(x)}{|x|}=-\lim_{x\to-\infty}\tilde{g}'(x)=-\tilde{g}'(l^+)$. In either case, $\lim_{x\to-\infty}\frac{\tilde{g}(x)}{|x|}$ exists.
\begin{itemize}
\item {\bf Subcase 1:} $\boldsymbol{\lim_{x\to-\infty}\frac{\tilde{g}(x)}{|x|}=\infty}$\\
Note that by (\ref{eqn_phi1}), $\int_{-\infty}^{x}\Phi(t)dt = x\Phi(x)+g_*(x)\rho_*(x) \le g_*(x)\rho_*(x)$ so
$$\frac{|x|\int_{-\infty}^{x}\Phi(t)dt}{g^2_*(x)\rho_*(x)}\leq C\frac{|x|g_*(x)\rho_*(x)}{\tilde{g}(x)g_*(x)\rho_*(x)} = C\frac{|x|}{\tilde{g}(x)}.$$
Therefore
$$\varlimsup_{x\rightarrow -\infty}\left|  f^{\prime}\left(  x\right)  \right|\leq2\left\|  h^{\prime}\right\|  _{\infty}C\varlimsup_{x\rightarrow -\infty}\frac{|x|}{\tilde{g}(x)}=0<\infty.$$

\item {\bf Subcase 2:} $\boldsymbol{\lim_{x\to-\infty}\frac{\tilde{g}(x)}{|x|}<\infty}$\\
Similarly from (\ref{eqn_phi1}),
$$\frac{|x|\int_{-\infty}^{x}\Phi(t)dt}{g^2_*(x)\rho_*(x)}\leq\frac{\int_{-\infty}^{x}\frac{|x|}{|t|}g_*(t)\rho_*(t)dt}{m_0g_*(x)\rho_*(x)}\leq\frac{\int_{-\infty}^{x}g_*(t)\rho_*(t)dt}{m_0g_*(x)\rho_*(x)}.$$
Therefore,
\begin{align*}
\varlimsup_{x\rightarrow -\infty}\left|  f^{\prime}\left(  x\right)  \right| &\leq\frac{2||h'||_\infty}{m_0}\varlimsup_{x\rightarrow -\infty}\frac{\int_{-\infty}^{x}g_*(t)\rho_*(t)dt}{g_*(x)\rho_*(x)}\leq\frac{2||h'||_\infty}{m_0}\varlimsup_{x\rightarrow -\infty}\frac{g_*(x)\rho_*(x)}{-x\rho_*(x)} \\
&\leq \frac{2||h'||_\infty}{m_0}\varlimsup_{x\rightarrow -\infty}\frac{\tilde{g}(x)}{c|x|}<\infty.
\end{align*}
\end{itemize}
\end{itemize}
The proof that $\varlimsup_{x\rightarrow u}\left|  f^{\prime}\left(  x\right)
\right|  \leq k_{2}||h^{\prime}||_{\infty}$ for some $k_{2}<\infty$ is similar.
\end{proof}

\bigskip

Note that if $g_*$ is uniformly bounded below in a neighborhood of $l>-\infty$ (or for $u<\infty$) then condition {\bf 2} ({\bf 1} in the case of $u$) from Theorem \ref{prop_fprimebound} is not required (see discussion before Lemma \ref{prop_tudorbound}). In the statement of the previous theorem, we can take $\tilde{g}=g_{\ast}$ if $g_{\ast}$ is continuously differentiable (at least locally $\mathcal{C}^1$ close to the endpoints of the support), and in this case the conditions are trivially met. In other words, if we can check that $g_*\in \mathcal{C}^1(l,u)$ then bound (\ref{eqn_fprimebound}) is automatically true (given the existence of $\tilde{g}'(u^-)$ and $\tilde{g}'(l^+)$). These new conditions
are met by all r.v.'s in the Exponential family, Pearson family, and practically any other
r.v. whose density is $\mathcal{C}^{1}$ and is strictly positive in its support. If
$g_{\ast}$ is not continuously differentiable, we can still get the bound but
we are required to approximate $g_{\ast}$ by a continuously differentiable
function $\tilde{g}$ near the endpoints of the support. For example, consider
the Laplace distribution where $g_{\ast}(x)=\frac{1}{c^{2}}(1+c|x|)$ (see
Table \ref{table_geestar}). In this case $g_{\ast}$ is differentiable everywhere except at 0. Therefore we can choose
$\tilde{g}(x)=g_{\ast}(x)$ for all $x\in\left(  -\infty,l^{\prime}\right)
\cup\left(  u^{\prime},\infty\right)  $ (with $-\infty<l^{\prime}<0<u^{\prime
}<\infty$) and $\tilde{g}(x)=\phi(x)$ on $(l^{\prime},u^{\prime})$ where $\phi$
is a smooth function such that $\tilde{g}$ is differentiable at $l^{\prime}$
and $u^{\prime}$.

\begin{description}
\item[Assumption B] We have the following conditions on $g_{\ast}$.

\begin{enumerate}

\item For some positive $\tilde{g}\in \mathcal{C}^{1}\left(l,u\right)$,
\begin{enumerate}
\item $0<\varliminf_{x\rightarrow u}g_{\ast}\left(  x\right)/\tilde{g}\left(  x\right)\leq\varlimsup_{x\rightarrow u}g_{\ast}\left(  x\right)  /\tilde{g}\left(  x\right)  <\infty$. 

\item $0<\varliminf_{x\rightarrow l}g_{\ast}\left(  x\right)/\tilde{g}\left(  x\right)\leq\varlimsup_{x\rightarrow l}g_{\ast}\left(  x\right)  /\tilde{g}\left(  x\right)  <\infty$.

\item $\tilde{g}'(l^+)\ $ and $\ \tilde{g}'(u^-)$ exist.

\end{enumerate}
\item If $u=\infty$, then $\varliminf_{x\rightarrow u}g_{\ast}\left(
x\right)  >0$.

\item If $l=-\infty$, then $\varliminf_{x\rightarrow l}g_{\ast}\left(
x\right)  >0$.

\end{enumerate}

\end{description}

\subsection{Bound for $f^{\prime\prime}$}\label{fprimeprime}

For our convergence in distribution results in Wiener-Poisson space, we need a
boundedness result for $f^{\prime\prime}$. The existence of $f^{\prime\prime}$
demands more conditions on $g_{\ast}$ such as differentiability, which is
understandable since we are requiring greater regularity in the solution of
the Stein equation. In this setting, the existence of $f^{\prime\prime}$ will also immediately force most conditions of Theorem \ref{prop_fprimebound} to be satisfied. If we want to work with $d_{W}$ or $d_{FM}$, we need to consider Lipschitz functions $h$, and for any such test function, we can only hope for it to be differentiable almost everywhere. Consequently, $f''$ must be understood in the almost everywhere sense, i.e., $f''$ is a version of the second derivative of $f$ such that wherever the second derivative does not exist, $f''$ will have a value of 0.

\bigskip

Before setting out to find a bound, we point out the unfortunate fact that our
results here will not apply to as wide a range of target r.v. $Z$ as what
happened for the first derivative. More specifically, we won't be able to give a finite bound for $\left|f^{\prime\prime}\left(  x\right)  \right|  $ when $l>-\infty$ or $u<\infty$,
as we were able to do for $\left|  f^{\prime}\left(  x\right)  \right|  $ in
Theorem \ref{prop_fprimebound}. We actually have a counterexample to
illustrate this: a r.v. $Z$ with support $\left(  l,\infty\right)
\varsubsetneq$ $\mathbb{R}$, such that for some Lipschitz and bounded $h$,
$f^{\prime\prime}\left(  x\right)  $ does not tend to a finite limit as
$x\rightarrow l$. A similar counterexample can be constructed for a r.v. $Z$
with support $\left(  -\infty,u\right)  \varsubsetneq$ $\mathbb{R}$, or with
support $\left(  l,u\right)  \varsubsetneq$ $\mathbb{R}$.

\bigskip

First, we make preliminary computations on $f^{\prime\prime}$. Differentiating
(\ref{eqn_steineqn}) gives us the second derivative%
\[
f^{\prime\prime}\left(  x\right)  =\frac{x-g_{\ast}^{\prime}\left(  x\right)
}{g_{\ast}\left(  x\right)  }f^{\prime}\left(  x\right)  +\frac{1}{g_{\ast
}\left(  x\right)  }f\left(  x\right)  +\frac{1}{g_{\ast}\left(  x\right)
}h^{\prime}\left(  x\right)
\]
which, after considering the form of $f$ in equation (\ref{eqn_f}) and of
$f^{\prime}$ given in Proposition \ref{prop_fprime}, reduces to%

\begin{equation}
f^{\prime\prime}\left(  x\right)  =\frac{A\left(  x\right)
{\displaystyle\int_{l}^{x}}
\Phi\left(  t\right)  h^{\prime}\left(  t\right)  dt+B\left(  x\right)
{\displaystyle\int_{x}^{u}}
\left[  1-\Phi\left(  s\right)  \right]  h^{\prime}\left(  s\right)
ds+g_{\ast}^{2}\left(  x\right)  \rho_{\ast}\left(  x\right)  h^{\prime
}\left(  x\right)  }{g_{\ast}^{3}\left(  x\right)  \rho_{\ast}\left(
x\right)  } \label{eqn_fprime2}%
\end{equation}
where, with the help of (\ref{eqn_phi1}) and (\ref{eqn_phi2}),%
\begin{align}
A\left(  x\right)   &  =\left(  x-g_{\ast}^{\prime}\left(  x\right)  \right)
\int_{x}^{u}\left[  1-\Phi\left(  s\right)  \right]  ds-g_{\ast}\left(
x\right)  \left(  1-\Phi\left(  x\right)  \right) \nonumber\\
&  =g_{\ast}\left(  x\right)  \rho_{\ast}\left(  x\right)  \left(  x-g_{\ast
}^{\prime}\left(  x\right)  \right)  -\left(  x^{2}-xg_{\ast}^{\prime}\left(
x\right)  +g_{\ast}\left(  x\right)  \right)  \left(  1-\Phi\left(  x\right)
\right) \label{eqn_ax}\\
B\left(  x\right)   &  =-\left(  x-g_{\ast}^{\prime}\left(  x\right)  \right)
\int_{l}^{x}\Phi\left(  t\right)  dt-g_{\ast}\left(  x\right)  \Phi\left(
x\right) \nonumber\\
&  =g_{\ast}\left(  x\right)  \rho_{\ast}\left(  x\right)  \left(  g_{\ast
}^{\prime}\left(  x\right)  -x\right)  -\left(  x^{2}-xg_{\ast}^{\prime
}\left(  x\right)  +g_{\ast}\left(  x\right)  \right)  \Phi\left(  x\right)
\text{.} \label{eqn_bx}%
\end{align}
Let $d\left(  x\right)  =g_{\ast}^{3}\left(  x\right)  \rho_{\ast}\left(
x\right)  $ and $n\left(  x\right)  =f^{\prime\prime}\left(  x\right)
d\left(  x\right)  $, the indicated denominator and numerator, respectively,
of $f^{\prime\prime}\left(  x\right)  $. As $x\rightarrow l$, both $d\left(
x\right)  $ and $n\left(  x\right)  $ tend to $0$. If $h^{\prime}$ happens to be differentiable, then by L'H\^{o}pital's rule, $\lim
_{x\rightarrow l}f^{\prime\prime}\left(  x\right)  =\lim_{x\rightarrow
l}n^{\prime}\left(  x\right)  /d^{\prime}\left(  x\right)  $. It can be shown
that $A^{\prime}\left(  x\right)  =\left(  2-g_{\ast}^{\prime\prime}\left(
x\right)  \right)  \int_{x}^{u}\left[  1-\Phi\left(  s\right)  \right]  ds$
and $B^{\prime}\left(  x\right)  =-\left(  2-g_{\ast}^{\prime\prime}\left(
x\right)  \right)  \int_{l}^{x}\Phi\left(  t\right)  dt$. Therefore
\begin{align*}
n^{\prime}\left(  x\right)   &  =A^{\prime}\left(  x\right)  \int_{l}^{x}%
\Phi\left(  t\right)  h^{\prime}\left(  t\right)  dt+A\left(  x\right)
\Phi\left(  x\right)  h^{\prime}\left(  x\right)  +B^{\prime}\left(  x\right)
\int_{x}^{u}\left[  1-\Phi\left(  s\right)  \right]  h^{\prime}\left(
s\right)  ds-B\left(  x\right)  \left[  1-\Phi\left(  x\right)  \right]
h^{\prime}\left(  x\right) \\
&  \quad\quad+\left[  -xg_{\ast}\left(  x\right)  \rho_{\ast}\left(  x\right)
+g_{\ast}^{\prime}\left(  x\right)  g_{\ast}\left(  x\right)  \rho_{\ast
}\left(  x\right)  \right]  h^{\prime}\left(  x\right)  +g_{\ast}^{2}\left(
x\right)  \rho_{\ast}\left(  x\right)  h^{\prime\prime}\left(  x\right) \\
&  =\left(  2-g_{\ast}^{\prime\prime}\left(  x\right)  \right)  \int_{x}%
^{u}\left[  1-\Phi\left(  s\right)  \right]  ds\int_{l}^{x}\Phi\left(
t\right)  h^{\prime}\left(  t\right)  dt-\left(  2-g_{\ast}^{\prime\prime
}\left(  x\right)  \right)  \int_{l}^{x}\Phi\left(  t\right)  dt\int_{x}%
^{u}\left[  1-\Phi\left(  s\right)  \right]  h^{\prime}\left(  s\right)  ds\\
&  \quad\quad+\left[  A\left(  x\right)  \Phi\left(  x\right)  -B\left(
x\right)  \left(  1-\Phi\left(  x\right)  \right)  -\left(  x-g_{\ast}%
^{\prime}\left(  x\right)  \right)  g_{\ast}\left(  x\right)  \rho_{\ast
}\left(  x\right)  \right]  h^{\prime}\left(  x\right)  +g_{\ast}^{2}\left(
x\right)  \rho_{\ast}\left(  x\right)  h^{\prime\prime}\left(  x\right) \\
&  =\left(  2-g_{\ast}^{\prime\prime}\left(  x\right)  \right)  g_{\ast}%
^{2}\left(  x\right)  \rho_{\ast}\left(  x\right)  f^{\prime}\left(  x\right)
+0\cdot h^{\prime}\left(  x\right)  +g_{\ast}^{2}\left(  x\right)  \rho_{\ast
}\left(  x\right)  h^{\prime\prime}\left(  x\right)
\end{align*}
and so%
\begin{align*}
\lim_{x\rightarrow l}f^{\prime\prime}\left(  x\right)   &  =\lim_{x\rightarrow
l}\frac{\left(  2-g_{\ast}^{\prime\prime}\left(  x\right)  \right)  g_{\ast
}^{2}\left(  x\right)  \rho_{\ast}\left(  x\right)  f^{\prime}\left(
x\right)  +g_{\ast}^{2}\left(  x\right)  \rho_{\ast}\left(  x\right)
h^{\prime\prime}\left(  x\right)  }{\left(  2g_{\ast}^{\prime}\left(
x\right)  -x\right)  g_{\ast}^{2}\left(  x\right)  \rho_{\ast}\left(
x\right)  }\\
&  =\lim_{x\rightarrow l}\frac{2-g_{\ast}^{\prime\prime}\left(  x\right)
}{2g_{\ast}^{\prime}\left(  x\right)  -x}f^{\prime}\left(  x\right)
+\lim_{x\rightarrow l}\frac{h^{\prime\prime}\left(  x\right)  }{2g_{\ast
}^{\prime}\left(  x\right)  -x}\text{.}%
\end{align*}
Define the function $h\left(  x\right)  =\frac{4}{3}\left(  x-l\right)
^{3/2}$ on $\left(  l,0\right)  $, $h\left(  x\right)  =\frac{4}{3}\left|
l\right|  ^{3/2}$ on $\left[  0,\infty\right)  $\ and $h\left(  x\right)  =0$
on $\left(  -\infty,l\right]  $. This function is clearly Lipschitz. Note that
$h^{\prime\prime}\left(  x\right)  =\frac{1}{\sqrt{x-l}}$ on $\left(
l,0\right)  $. We now consider the same assumptions from Theorem
\ref{prop_fprimebound} and see that $\varlimsup_{x\rightarrow l}\left|
f^{\prime}\left(  x\right)  \right|  \leq k\left\|  h^{\prime}\right\|
_{\infty}$ and $\lim_{x\rightarrow l}\frac{h^{\prime\prime}\left(  x\right)
}{2g_{\ast}^{\prime}\left(  x\right)  -x}=\infty$. We have thus found a Lipschitz function $h$ for which $\lim_{x\rightarrow l}\left|  f^{\prime\prime}\left(  x\right)  \right|
=\infty$.

\begin{remark}
From the above discussion we can't expect to have a universal bound on the
second derivative of $f$ unless the support of the target r.v. is
$(-\infty,\infty)$. This is consistent with the known NP bound in Wiener-Poisson space developed in \cite{Viquez}, where $Z$ was Normal and
hence had $(-\infty,\infty)$ for support. For the rest of this subsection, we will then assume that $l=-\infty$ and $u=\infty$.
\end{remark}

\begin{theorem}
\label{thm_fprime2}Assume that $g_{\ast}$ is twice differentiable and $g_{\ast}^{\prime\prime}(x)<2$. Suppose too that
$\left|  \frac{x-g_{\ast}^{\prime}\left(  x\right)  }{x^{2}-xg_{\ast}^{\prime
}\left(  x\right)  +g_{\ast}\left(  x\right)  }\right|  $ is bounded as
$x\rightarrow -\infty$ and as $x\rightarrow \infty$. Then the solution $f$ of the Stein equation (\ref{eqn_steineqn}), for a given
test function $h$ with $\left\|
h^{\prime}\right\|  _{\infty}<\infty$, has second derivative bounded as follows:%
\begin{equation}
\left\|  f''\right\|  _{\infty}\leq k\left\|  h^{\prime}\right\|  _{\infty}
\label{eqn_fprimeboundtudor2}%
\end{equation}
where the constant $k$ depends on $Z$ alone, and not on $h$.
\end{theorem}
\begin{proof}
Recall the functions $A$ and $B$ in (\ref{eqn_ax}) and (\ref{eqn_bx}). Using
Lemma 7 in \cite{EV} (note that $\Phi$ there is defined as the upper
probability tail), $A\left(  x\right)  \leq0$ and $B\left(  x\right)  \leq0$.
Therefore, from (\ref{eqn_fprime2}),
\begin{align*}
\left|  f^{\prime\prime}\left(  x\right)  \right|   &  \leq\frac{-A\left(
x\right)  }{g_{\ast}^{3}\left(  x\right)  \rho_{\ast}\left(  x\right)  }%
\int_{l}^{x}\Phi\left(  t\right)  dt\cdot\left\|  h^{\prime}\right\|
_{\infty}+\frac{-B\left(  x\right)  }{g_{\ast}^{3}\left(  x\right)  \rho
_{\ast}\left(  x\right)  }\int_{x}^{u}\left[  1-\Phi\left(  s\right)  \right]
ds\cdot\left\|  h^{\prime}\right\|  _{\infty}+\frac{\left|  h^{\prime}\left(
x\right)  \right|  }{g_{\ast}\left(  x\right)  }\\
\frac{g_{\ast}^{3}\left(  x\right)  \rho_{\ast}\left(  x\right)  \left|
f^{\prime\prime}\left(  x\right)  \right|  }{\left\|  h^{\prime}\right\|
_{\infty}}  &  \leq\left[  -\left(  x-g_{\ast}^{\prime}\left(  x\right)
\right)  \int_{x}^{u}\left[  1-\Phi\left(  s\right)  \right]  ds+g_{\ast
}\left(  x\right)  \left(  1-\Phi\left(  x\right)  \right)  \right]  \int
_{l}^{x}\Phi\left(  t\right)  dt\\
&  \quad\quad+\left[  \left(  x-g_{\ast}^{\prime}\left(  x\right)  \right)
\int_{l}^{x}\Phi\left(  t\right)  dt+g_{\ast}\left(  x\right)  \Phi\left(
x\right)  \right]  \int_{x}^{u}\left[  1-\Phi\left(  s\right)  \right]
ds+g_{\ast}^{2}\left(  x\right)  \rho_{\ast}\left(  x\right) \\
&  =g_{\ast}\left(  x\right)  \left(  1-\Phi\left(  x\right)  \right)
\int_{l}^{x}\Phi\left(  t\right)  dt+g_{\ast}\left(  x\right)  \Phi\left(
x\right)  \int_{x}^{u}\left[  1-\Phi\left(  s\right)  \right]  ds+g_{\ast}%
^{2}\left(  x\right)  \rho_{\ast}\left(  x\right) \\
&  =g_{\ast}\left(  x\right)  \left(  1-\Phi\left(  x\right)  \right)  \left(
g_{\ast}\left(  x\right)  \rho_{\ast}\left(  x\right)  +x\Phi\left(  x\right)
\right)  +g_{\ast}\left(  x\right)  \Phi\left(  x\right)  \left(  g_{\ast
}\left(  x\right)  \rho_{\ast}\left(  x\right)  -x\left[  1-\Phi\left(
x\right)  \right]  \right) \\
& \qquad +g_{\ast}^{2}\left(  x\right)  \rho_{\ast}\left(
x\right) \\
&  =2g_{\ast}^{2}\left(  x\right)  \rho_{\ast}\left(  x\right)  \text{.}%
\end{align*}
Due to the continuity of $g_{\ast}$ and conditions of Assumption B when $l=-\infty$ and $u=\infty$, there is some $m_{0}>0$ such that $g_{\ast}(x)>m_{0}$ for all $x\in\mathbb{R}$. Then, $|f^{\prime\prime}(x)|\leq\frac{2||h^{\prime}||_{\infty}}{g_{\ast}(x)}\leq\frac{2}{m_{0}}||h^{\prime}||_{\infty}=k||h^{\prime}||_{\infty}$.
\end{proof}

\bigskip

One might think at first glance that the conditions of Theorem
\ref{thm_fprime2} are too restrictive. However, a closer look will show that
they are all satisfied by members of the Pearson family having $(-\infty
,\infty)$ as its support. Examples are the Pearson Type IV, Normal, and
Student's T distributions (see Table \ref{table_geestar} to check the
conditions).

\begin{description}
\item[Assumption B$^{\prime}$] Along with Assumption B, the following hold.
\begin{enumerate}

\item $g_*$ is twice differentiable and $g_*''<2$.

\item $\varlimsup_{x\to\pm\infty}\left|  \frac{x-g_{\ast}^{\prime}\left(  x\right)  }{x^{2}-xg_{\ast}^{\prime
}\left(  x\right)  +g_{\ast}\left(  x\right)  }\right|  <\infty$.

\end{enumerate}
\end{description}

\section{\label{sec_Wresults}NP bound in Wiener space}

From the results in subsection \ref{fprime} all solutions of the Stein equation belong to the set $\mathscr{F}_{\mathcal{H}}=\{f\in \mathcal{C}^1(l,u) : ||f'||_{\infty}\leq k\}$, where the constant $k$ depends on the distance $d_{\mathcal{H}}$ used (and so it implicitly depends on the set $\mathcal{H}$).

\begin{theorem}
\label{thm_boundnew}(NP bound) Let $d_{\mathcal{H}}$ be $d_{W}$ or $d_{FM}$. Under
Assumptions A and B,%
\begin{align}
d_{\mathcal{H}}\left(  X,Z\right)  &\leq k\mathbb{E}\left|  g_{\ast}\left(  X\right)
-g_{X}\right|  \leq k\mathbb{E}\left[  \left(  g_{\ast}\left(  X\right)
-g_{X}\right)  ^{2}\right]  ^{1/2} \label{eqn_inequality1}\\
&\leq k\sqrt{\left|  \mathbb{E}\left[  g_{\ast}\left(
X\right)  ^{2}\right]  -\mathbb{E}\left[  g_{\ast}\left(  Z\right)
^{2}\right]  \right|  +\left|  \mathbb{E}\left[  g_*(X)g_X
\right]  -\mathbb{E}\left[  g_{\ast}\left(  Z\right)g_Z  \right]  \right|
+\left|  \mathbb{E}\left[  g_{X}^{2}\right]  -\mathbb{E}\left[  g_{Z}%
^{2}\right]  \right|  }\text{.} \label{eqn_inequality3}
\end{align}

Let $G_{\ast}\left(  x\right)  $ be an antiderivative of $g_{\ast}\left(
x\right)  $. Under Assumptions A$^{\prime}$ and B,
\begin{equation}
d_{\mathcal{H}}\left(  X,Z\right)  \leq k\sqrt{\left|  \mathbb{E}\left[  g_{\ast}\left(
X\right)  ^{2}\right]  -\mathbb{E}\left[  g_{\ast}\left(  Z\right)
^{2}\right]  \right|  +\left|  \mathbb{E}\left[  XG_{\ast}\left(  X\right)
\right]  -\mathbb{E}\left[  ZG_{\ast}\left(  Z\right)  \right]  \right|
+\left|  \mathbb{E}\left[  g_{X}^{2}\right]  -\mathbb{E}\left[  g_{Z}%
^{2}\right]  \right|  }\text{.} \label{eqn_inequality2}%
\end{equation}
In both statements, $k$ is a finite constant depending only on $Z$ and on $d_{\mathcal{H}}$.
\end{theorem}
\begin{proof}
The first bound in (\ref{eqn_inequality1}) follows from (\ref{eqn_basicbound})
and Theorem \ref{prop_fprimebound}. The second bound follows from H\"{o}lder's
Inequality. Let $\Delta=\mathbb{E}\left[  \left(  g_{\ast}\left(  X\right)
-g_{X}\right)  ^{2}\right]  ^{1/2}$. %
Since $\left(  g_{\ast}\left(  Z\right)  -g_{Z}\right)  ^{2}=0$ a.s.,
\[
\Delta^{2}=\mathbb{E}\left[  g_{\ast}\left(  X\right)  ^{2}\right]
-2\mathbb{E}\left[  g_{\ast}\left(  X\right)g_X  \right]  +\mathbb{E}\left[
g_{X}^{2}\right]  -\left(  \mathbb{E}\left[  g_{\ast}\left(  Z\right)
^{2}\right]  -2\mathbb{E}\left[  g_{\ast}\left(  Z\right)g_Z  \right]
+\mathbb{E}\left[  g_{Z}^{2}\right]  \right)
\]
and (\ref{eqn_inequality3}) follows. From (\ref{eqn_biggstar}) and Assumption A$^{\prime}$ we have $\mathbb{E}\left[g_*(F)g_F\right]=\mathbb{E}\left[  FG_{\ast}\left(  F\right)  \right]$, which proves (\ref{eqn_inequality2}).
\end{proof}

\bigskip

The first inequality also follows from Theorem 1 and equation (19) in Kusuoka
and Tudor \cite{KT}. The setup in their paper involves functions $b$ and
$a$. The function $b$ is any function for which $\int_{l}%
^{u}b\left(  x\right)  \rho_{\ast}\left(  x\right)  dx=0$ along with a few
other mild conditions: $b>0$ near $l$, $b<0$ near $u$, $b\rho_{\ast}$ is
continuous and bounded on $\left(  l,u\right)  $. They then defined $a\left(
x\right)  =2\int_{l}^{x}b\left(  y\right)  \rho_{\ast}\left(  y\right)
dy/\rho_{\ast}\left(  x\right)  $. Then for $W$ a standard Brownian motion,
the SDE
\begin{equation}
dY_{t}=b\left(  Y_{t}\right)  dt+\sqrt{a\left(  Y_{t}\right)  }dW_{t}
\label{eqn_sde}%
\end{equation}
has a unique Markovian weak solution with invariant density $\rho_{\ast}$.
With $a$ and $b$ as given above, from Theorem 1 in \cite{KT},
\begin{equation}
d_{\mathcal{H}}\left(  X,Z\right)  \leq k\mathbb{E}\left|  \frac{a\left(  X\right)  }%
{2}-\left\langle DX,DL^{-1}\left\{  b\left(  X\right)  -\mathbb{E}b\left(
X\right)  \right\}  \right\rangle \right|  +k\left|  \mathbb{E}b\left(
X\right)  \right|  \text{.} \label{eqn_tudor}%
\end{equation}
If we take $b\left(  x\right)
=-x$, it follows that $a\left(  x\right)  =2g_{\ast}\left(  x\right)  $. If $X$ is centered, the right-hand side of (\ref{eqn_tudor}) quickly reduces to $k\mathbb{E}\left|
g_{\ast}\left(  X\right)  -g_{X}\right|  $.

\bigskip

While the results in \cite{KT} appear more general, taking $b\left(  x\right)
=-x$ suffices. A careful analysis will reveal that the proofs of their main
results depend only on the density $\rho_{\ast}$ and the choice of $b$. While
each choice of $b$ arguably yields a different diffusion process $Y$, the
invariant density is still $\rho_{\ast}$. Their analytical proofs are in fact
independent of the stochastic differential equation (\ref{eqn_sde}) and the
diffusion process arising from it. For this paper, we only need comparisons
with the law of the reference variable $Z$. To this end, knowing the density
$\rho_{\ast}$ will suffice. The computations using $b\left(  x\right)  =-x$
and $a\left(  x\right)  =2g_{\ast}\left(  x\right)  $ are much easier and this
is reflected in the simplicity of (\ref{eqn_inequality1}) compared to
(\ref{eqn_tudor}).

\bigskip

Furthermore, as shown in the next theorem, the bounds we get from
taking $b\left(  x\right)  =-x$ (see Theorem \ref{thm_boundnew}) are tight.
Indeed, nothing is lost by choosing $b$ this way.

\begin{theorem}
\label{thm_boundtight}(Law Characterization) $X\overset{\text{Law}}{=}Z$ if and only if all of the
following are satisfied.
\end{theorem}

\begin{enumerate}
\item $\mathbb{E}\left[  g_{\ast}\left(  X\right)  ^{2}\right]  =\mathbb{E}%
\left[  g_{\ast}\left(  Z\right)  ^{2}\right]  $

\item $\mathbb{E}\left[  XG_{\ast}\left(  X\right)  \right]  =\mathbb{E}%
\left[  ZG_{\ast}\left(  Z\right)  \right]  $

\item $\mathbb{E}\left[  g_{X}^{2}\right]  =\mathbb{E}\left[  g_{Z}%
^{2}\right]  $
\end{enumerate}
\begin{proof}
If the three conditions are satisfied, Theorem \ref{thm_boundnew} implies
$d\left(  X,Z\right)  =0$.

Now suppose $X\overset{\text{Law}}{=}Z$. They then have the same density
$\rho_{\ast}$ so \textbf{1} and \textbf{2} immediately follow. We next prove
that $g_{X}\overset{\text{Law}}{=}g_{Z}$, imitating the technique Nourdin and
Viens used to prove (\ref{eqn_gstargee}) (see Theorem 3.1 \cite{NV}). Let $f$
be a continuous function with compact support, and $F$ any antiderivative.
\begin{align*}
\mathbb{E}\left[  f\left(  X\right)  g_{X}\right]   &  =\mathbb{E}\left[
XF\left(  X\right)  \right]  =\int_{l}^{u}\left[  x\rho_{\ast}\left(
x\right)  \right]  F\left(  x\right)  dx\\
&  =\left.  -F\left(  x\right)  \int_{x}^{u}y\rho_{\ast}\left(  y\right)
dy\right|  _{x\rightarrow l}^{x\rightarrow u}+\int_{l}^{u}f\left(  x\right)
\left[  \int_{x}^{u}y\rho_{\ast}\left(  y\right)  dy\right]  dx\\
&  =\int_{l}^{u}f\left(  x\right)  \frac{\int_{x}^{u}y\rho_{\ast}\left(
y\right)  dy}{\rho_{\ast}\left(  x\right)  }\rho_{\ast}\left(  x\right)
dx=\mathbb{E}\left[  f\left(  X\right)  \frac{\int_{X}^{u}y\rho_{\ast}\left(
y\right)  dy}{\rho_{\ast}\left(  X\right)  }\right]
\end{align*}
so $g_{X}=\int_{X}^{u}y\rho_{\ast}\left(  y\right)  dy/\rho_{\ast}\left(
X\right)  $ a.s. This has the same distribution as $\int_{Z}^{u}y\rho_{\ast
}\left(  y\right)  dy/\rho_{\ast}\left(  Z\right)  $, equal to $g_{Z}$ a.s.,
so \textbf{3} then follows.
\end{proof}

\begin{remark}
We see that $\mathbb{E}\left[  g_{\ast}\left(  Z\right)  ^{2}\right]
=\mathbb{E}\left[  ZG_{\ast}\left(  Z\right)  \right]  =\mathbb{E}\left[
g_{Z}^{2}\right]  $ (see Lemma \ref{lem_expec}). Thus, for $X$ to have the
same law as $Z$, it is necessary and sufficient that $\mathbb{E}\left[
g_{\ast}\left(  X\right)  ^{2}\right]  $, $\mathbb{E}\left[  XG_{\ast}\left(
X\right)  \right]  $ and $\mathbb{E}\left[  g_{X}^{2}\right]  $ (which a
priori need not be all the same) are all equal to $\mathbb{E}\left[  g_{Z}%
^{2}\right]  $. The three conditions in Theorem \ref{thm_boundtight} are
stated in their current form due to the symmetry involved.
\end{remark}

That $\mathbb{E}\left[  g_{\ast}\left(  Z\right)  ^{2}\right]  =\mathbb{E}%
\left[  ZG_{\ast}\left(  Z\right)  \right]  =\mathbb{E}\left[  g_{Z}%
^{2}\right]  $ are all equal depends on the specific structure of $Z$ itself,
and it is rooted in how $g_{\ast}$ (and thus $G_{\ast}$ as well) is defined in
terms of the law of $Z$. Specifically, it is because $g_{\ast}\left(
Z\right)  =g_{Z}$ that we are able to use the integration by parts formula
(\ref{eqn_integpartsW}) on $g_{\ast}\left(  Z\right)  $. If we evaluate the
function $g_{\ast}$ at the random variable $X$ , we cannot expect $g_{\ast
}\left(  X\right)  $ to be equal to $g_{X}$ because $g_{\ast}$ is an object
that ``belongs'' to $Z$. However, if $X$ and $Z$ are to be ``almost'' the same
in law, we would expect $X$ to ``almost'' satisfy the same relations/equations
for $Z$, e.g. $\mathbb{E}\left[  g_{\ast}\left(  X\right)  ^{2}\right]
``="\mathbb{E}\left[  XG_{\ast}\left(  X\right)  \right]  $. If $g_{\ast}$ is
a polynomial, then this amounts to checking that the moments of $X$ satisfy
the same conditions met by the moments of $Z$. Granted, this method of moments
is not sufficient. Hence, the need for condition \textbf{3}, $\mathbb{E}\left[
g_{X}^{2}\right]  =\mathbb{E}\left[  g_{Z}^{2}\right]  $, in Theorem
\ref{thm_boundtight}.

\bigskip

The following versions of Theorem \ref{thm_boundtight} and Theorem
\ref{thm_boundnew} for sequences are useful.

\begin{corollary}
\label{cor_boundsequence}$X_{n}\rightarrow Z$ in distribution if all of the
following are satisfied.
\end{corollary}

\begin{enumerate}
\item $\mathbb{E}\left[  g_{\ast}\left(  X_{n}\right)  ^{2}\right]
\rightarrow\mathbb{E}\left[  g_{\ast}\left(  Z\right)  ^{2}\right]  $

\item $\mathbb{E}\left[  g_{\ast}\left(  X_{n}\right)g_{X_n}  \right]
\rightarrow\mathbb{E}\left[  g_{\ast}\left(  Z\right)g_Z  \right]  $ (under Assumption A)

\vspace{.1cm}

$\mathbb{E}\left[  X_nG_{\ast}\left(  X_{n}\right)  \right]
\rightarrow\mathbb{E}\left[  ZG_{\ast}\left(  Z\right)  \right]  $ (under Assumption A$^{\prime}$)

\item $\mathbb{E}\left[  g_{X_{n}}^{2}\right]  \rightarrow\mathbb{E}\left[
g_{Z}^{2}\right]  $
\end{enumerate}

\begin{corollary}
\label{cor_Lsequence}$X_{n}\rightarrow Z$ in distribution if $\ g_{\ast}\left(  X_{n}\right)  -g_{X_{n}}\rightarrow0$ in $L^{1}(\Omega)$
\end{corollary}

\begin{remark}
If we normalize so that $\operatorname*{Var}X=\operatorname*{Var}Z$, condition
\textbf{3} in Theorem \ref{thm_boundtight}\ can be replaced by
$\operatorname*{Var}g_{X}=\operatorname*{Var}g_{Z}$ since $\mathbb{E}\left[
g_{X}\right]  =\operatorname*{Var}X$. This also allows us to replace\ the term
$\left|  \mathbb{E}\left[  g_{X}^{2}\right]  -\mathbb{E}\left[  g_{Z}%
^{2}\right]  \right|  $ in Theorem \ref{thm_boundnew} by $\left|
\operatorname*{Var}g_{X}-\operatorname*{Var}g_{Z}\right|  $. In Corollary
\ref{cor_boundsequence}, we can replace condition \textbf{3} by
$\operatorname*{Var}g_{X_{n}}\rightarrow\operatorname*{Var}g_{Z}$ if
$\mathbb{E}\left[  X_{n}^{2}\right]  \rightarrow\mathbb{E}\left[
Z^{2}\right]  $.
\end{remark}

If $Z$ is Normal with variance $\sigma^{2}$ so $g_{\ast}\left(  y\right)
=\sigma^{2}$, $G_{\ast}\left(  y\right)  =\sigma^{2}y$ and $g_{Z}=\sigma^{2}$.
If $\operatorname*{Var}X=\sigma^{2}$, then%
\begin{equation}
d_{\mathcal{H}}\left(  X,Z\right)  \leq k\sqrt{\left|  \sigma^{4}-\sigma^{4}\right|
+\sigma^{2}\left|  \mathbb{E}\left[  X^{2}\right]  -\mathbb{E}\left[
Z^{2}\right]  \right|  +\left|  \operatorname*{Var}g_{X}-\operatorname*{Var}%
g_{Z}\right|  }=k\sqrt{\operatorname*{Var}g_{X}} \label{eqn_distnormal}%
\end{equation}
where $k=4$ if $d_{\mathcal{H}}=d_{FM}$ and $k=1$ if $d_{\mathcal{H}}=d_{W}$. This retrieves Theorem 3.3
in \cite{NPsurv}. If we have a bound on $\operatorname*{Var}g_{X}$, this may
be used to bound the distance. A Poincar\'{e}-type inequality may be used in
this regard. See \cite{NPR} (also for an explanation of the notation used
below) where they use such a bound on $\operatorname*{Var}g_{X}$ to get the
following result:%
\begin{equation}
d_{\mathcal{H}}\left(  X,Z\right)  \leq\frac{k\sqrt{10}}{2\sigma}\left(  \mathbb{E}\left[
\left\|  D^{2}X\otimes_{1}D^{2}X\right\|  _{\mathfrak{H}^{\otimes2}}^{2}\right]  \right)
^{1/2}\left(  \mathbb{E}\left[  \left\|  DX\right\|  _{\mathfrak{H}}^{4}\right]  \right)
^{1/2}\text{.} \label{eqn_Poincare}%
\end{equation}

\bigskip

This was used in \cite{NPR} and \cite{Viquez} to prove CLTs for functionals of
Gaussian subordinated fields (applied to fBm and the solution of the O-U SDE
driven by fBm, for all $H\in(0,1)$).

\subsection{\label{ssec_gpoly}Convergence when $g_{\ast}$ is a polynomial}

Many of the common random variables belong to the Pearson family of
distributions, all of whose members are characterized by their $g_{\ast}$
being polynomials of degree at most $2$, i.e. $g_{\ast}\left(  y\right)
=\alpha y^{2}+\beta y+\gamma$ in the support of $Z$. Some member distributions
in this family are Normal ($g_{\ast}$ is constant), Gamma ($g_{\ast}$ has
degree $1$), Beta ($g_{\ast}$ is quadratic with positive discriminant),
Student's T-distribution ($g_{\ast}$ is quadratic with negative discriminant)
and Inverse Gamma ($g_{\ast}$ is quadratic with zero discriminant).

\bigskip

Refer to \cite{DZ} and \cite{Stein} for more information about Pearson
distributions, and \cite{EV} for Stein's method applied to comparisons of probability tails with a Pearson $Z$. From Remark \ref{rem_geepower}, if the support of $Z$ is
unbounded and $g_{\ast}$ is a polynomial, then $Z$ is necessarily Pearson. If
$Z$ has bounded support and $g_{\ast}$ is a polynomial, $g_{\ast}$ may have
degree exceeding $2$ and in this case, $Z$ is not Pearson.

\begin{corollary}
\label{cor_gstarpoly}If $g_{\ast}$ is a polynomial $g_{\ast}\left(  x\right)
=\sum_{k=0}^{m}a_{k}x^{k}$, for the convergence $X_{n}\rightarrow Z$ in
distribution, conditions \textbf{1} and \textbf{2} in Corollary
\ref{cor_boundsequence} can be replaced by these conditions (respectively): $\mathbb{E}\left[
X_{n}^{k}\right]  \rightarrow\mathbb{E}\left[  Z^{k}\right]  $ for $k=1,\ldots,2m$, and  $\mathbb{E}\left[X_{n}^{k}g_{X_n}\right]  \rightarrow\mathbb{E}\left[  Z^{k}g_Z\right]  $ for $k=1,\ldots,m$. Under assumption A$^{\prime}$, the two conditions can be replaced by $\mathbb{E}\left[
X_{n}^{k}\right]  \rightarrow\mathbb{E}\left[  Z^{k}\right]  $ for $k = 1,\ldots,\max\left\{  2m,m+2\right\}  $.
\end{corollary}
\begin{proof}
$g^{2}_{\ast}\left(  x\right)  $ has order $2m$ while $xG_{\ast}\left(
x\right)  $ has order $m+2$. The matching moments ensure condition \textbf{1} in Corollary \ref{cor_boundsequence} is satisfied,
and under Assumption A$^{\prime}$ also condition \textbf{2} is fulfilled.
\end{proof}

\bigskip

Suppose $g_{\ast}\left(  x\right)  =\sum_{k=0}^{m}a_{k}x^{k}$. Note that%
\[
\mathbb{E}\left[  g_{\ast}\left(  Z\right)  ^{2}\right]  =\mathbb{E}\left[
\left(  \sum_{k=0}^{m}a_{k}Z^{k}\right)  ^{2}\right]  =\sum_{k=0}^{2m}\left(
\sum_{i=0}^{k}a_{i}a_{k-i}\right)  \mathbb{E}\left[  Z^{k}\right]
\]
while
\[
\mathbb{E}\left[  ZG_{\ast}\left(  Z\right)  \right]  =\sum_{k=0}^{m}%
\frac{a_{k}}{k+1}\mathbb{E}\left[  Z^{k+2}\right]  \text{.}%
\]

We noted earlier that $\mathbb{E}\left[  g_{\ast}\left(  Z\right)
^{2}\right]  $ and $\mathbb{E}\left[  ZG_{\ast}\left(  Z\right)  \right]  $
are equal. While the polynomial coefficients of the different moments of $Z$
are different, and more moments may be involved in one expression compared to
the other, the coefficients and the moments themselves should take care of
this apparent difference to ensure equality under the expectation.

\bigskip

Suppose $Z$ is Pearson with $g_{Z}=g_{\ast}\left(  Z\right)  =\alpha
Z^{2}+\beta Z+\gamma$. Using Lemma \ref{lem_expec}, we can prove the following
recursive formula for the moments of $Z$: $\mathbb{E}\left[  Z^{r+1}\right]
=\frac{r\beta}{1-r\alpha}\mathbb{E}\left[  Z^{r}\right]  +\frac{r\gamma
}{1-r\alpha}\mathbb{E}\left[  Z^{r-1}\right]  $. Therefore,%
\begin{align*}
\mathbb{E}\left[  g_Z\right]   =\mathbb{E}\left[  Z^{2}\right]   &  =\frac{\gamma}{1-\alpha}\\
2\mathbb{E}\left[  Zg_Z\right]   =\mathbb{E}\left[  Z^{3}\right]   &  =\frac{2\beta\gamma}{\left(
1-\alpha\right)  \left(  1-2\alpha\right)  }\\
3\mathbb{E}\left[  Z^{2}g_Z\right]   =\mathbb{E}\left[  Z^{4}\right]   &  =\frac{6\beta^{2}\gamma+\left(
1-2\alpha\right)  3\gamma^{2}}{\left(  1-\alpha\right)  \left(  1-2\alpha
\right)  \left(  1-3\alpha\right)  }%
\end{align*}
and
\begin{align}
\mathbb{E}\left[  g_{Z}^{2}\right]   &  =\frac{\beta^{2}\gamma\left(
1-\alpha\right)  +\gamma^{2}\left(  1-2\alpha\right)  ^{2}}{\left(
1-\alpha\right)  \left(  1-2\alpha\right)  \left(  1-3\alpha\right)
}\label{eqn_2ndgz}\\
\operatorname*{Var}g_{Z}  &  =\mathbb{E}\left[  g_{\ast}^{2}\left(  Z\right)
\right]  -\left(  \mathbb{E}\left[  g_{\ast}\left(  Z\right)  \right]
\right)  ^{2}=\frac{\beta^{2}\gamma\left(  1-\alpha\right)  ^{2}+2\alpha
^{2}\gamma^{2}\left(  1-2\alpha\right)  }{\left(  1-2\alpha\right)  \left(
1-3\alpha\right)  \left(  1-\alpha\right)  ^{2}}\text{.} \label{eqn_vargz}%
\end{align}

\begin{corollary}
\label{cor_Pearson}Suppose $Z$ is a Pearson random variable and for the
sequence $\left\{  X_{n}\right\}  $, $\operatorname*{Var}X_{n}=\mathbb{E}%
[X_{n}^{2}]=\mathbb{E}[g_{X_n}]\rightarrow\frac{\gamma}{1-\alpha}$. The following are sufficient
conditions so that $X_{n}\rightarrow Z$ in distribution.

\begin{enumerate}
\item When $Z$ is Normal ($\alpha=\beta=0$), $\operatorname*{Var}g_{X_{n}%
}\rightarrow0$.

\item When $Z$ is Gamma ($\alpha=0$), $\operatorname*{Var}g_{X_{n}}\rightarrow\beta
^{2}\gamma$ and
\begin{itemize}

\item under Assumption A, $\mathbb{E}\left[  X_{n}g_{X_n}\right]\rightarrow\beta\gamma$.

\item under Assumption A$^{\prime}$, $2\mathbb{E}\left[  X_{n}g_{X_n}\right]=\mathbb{E}\left[  X_{n}^{3}\right]\rightarrow2\beta\gamma$.

\end{itemize}

\item In the general case where $\alpha\neq0$, $\operatorname*{Var}g_{X_{n}}^{2}\rightarrow\frac{\beta^{2}%
\gamma\left(  1-\alpha\right)  ^{2}+2\alpha^{2}\gamma^{2}\left(
1-2\alpha\right)  }{\left(  1-2\alpha\right)  \left(  1-3\alpha\right)
\left(  1-\alpha\right)  ^{2}}$ and
\begin{itemize}

\item under Assumption A, \\ $2\mathbb{E}\left[
X_{n}g_{X_n}\right],\mathbb{E}\left[
X_{n}^{3}\right]  \rightarrow\frac{2\beta\gamma}{\left(  1-\alpha\right)
\left(  1-2\alpha\right)  }$, and $3\mathbb{E}\left[  X_{n}^{2}g_{X_n}\right],\mathbb{E}\left[  X_{n}^{4}\right]
\rightarrow\frac{6\beta^{2}\gamma+\left(  1-2\alpha\right)  3\gamma^{2}%
}{\left(  1-\alpha\right)  \left(  1-2\alpha\right)  \left(  1-3\alpha\right)
}$.

\item under Assumption A$^{\prime}$, \\ $2\mathbb{E}\left[
X_{n}g_{X_n}\right]=\mathbb{E}\left[
X_{n}^{3}\right]  \rightarrow\frac{2\beta\gamma}{\left(  1-\alpha\right)
\left(  1-2\alpha\right)  }$, and $3\mathbb{E}\left[  X_{n}^{2}g_{X_n}\right]=\mathbb{E}\left[  X_{n}^{4}\right]
\rightarrow\frac{6\beta^{2}\gamma+\left(  1-2\alpha\right)  3\gamma^{2}%
}{\left(  1-\alpha\right)  \left(  1-2\alpha\right)  \left(  1-3\alpha\right)
}$.

\end{itemize}

\end{enumerate}
\end{corollary}
\begin{proof}
Apply Corollary \ref{cor_gstarpoly} directly.
\end{proof}

\bigskip

The first statement is the version for sequences of Corollary 3.4 in
\cite{NV}. Alternatively, we could replace $\operatorname*{Var}g_{X_{n}%
}\rightarrow0$ by $\mathbb{E}\left[  g_{X_{n}}^{2}\right]  \rightarrow
\gamma^{2}$. For the Gamma convergence, we can replace $\operatorname*{Var}%
g_{X_{n}}\rightarrow\beta^{2}\gamma$ by $\mathbb{E}\left[  g_{X_{n}}%
^{2}\right]  \rightarrow\beta^{2}\gamma+\gamma^{2}$. When $\alpha\neq0$, we
can work with (\ref{eqn_2ndgz}) instead of (\ref{eqn_vargz}) so the statement
will be in terms of $\mathbb{E}\left[  g_{X_{n}}^{2}\right]  \rightarrow
\frac{\beta^{2}\gamma\left(  1-\alpha\right)  +\gamma^{2}\left(
1-2\alpha\right)  ^{2}}{\left(  1-\alpha\right)  \left(  1-2\alpha\right)
\left(  1-3\alpha\right)  }$.

\bigskip

The next result follows from Corollary \ref{cor_Lsequence}.

\begin{corollary}
\label{cor_LsequenceP}Suppose $Z$ is a Pearson random variable. $X_{n}\rightarrow Z$ in distribution if $\ g_{X_{n}}-\alpha X_{n}^{2}-\beta X_{n}\rightarrow\gamma$ in $L^{1}\left(  \Omega\right)  $.
\end{corollary}

\subsection{Convergence in a fixed Wiener chaos}

When $X$ is inside a fixed Wiener chaos so $X=I_{q}\left(  f\right)  $, we
have more structure available. For example, $\left\langle DX,-DL^{-1}%
X\right\rangle _\mathfrak{H}=\frac{1}{q}\left\|  DX\right\|  _\mathfrak{H}^{2}$. Therefore, if
$Z\overset{\text{Law}}{=}\mathcal{N}\left(  0,\sigma^{2}\right)  $ and
$\mathbb{E}\left[  \left(  I_{q}\left(  f\right)  \right)  ^{2}\right]
=\sigma^{2}$, (\ref{eqn_distnormal}) gives us the bound
\[
d_{\mathcal{H}}\left(  X,Z\right)  \leq k\sqrt{\operatorname*{Var}g_{X}}\leq k\sqrt
{\operatorname*{Var}\left(  \frac{1}{q}\left\|  DX\right\|  _\mathfrak{H}^{2}\right)
}\text{.}%
\]
One may then use bounds like
\begin{equation}
\operatorname*{Var}\left(  \frac{1}{q}\left\|  DX\right\|  _\mathfrak{H}^{2}\right)
\overset{\text{(a)}}{=}\frac{1}{q^{2}}\mathbb{E}\left[  \left(  \left\|
DX\right\|  _\mathfrak{H}^{2}-q\sigma^{2}\right)  ^{2}\right]  \overset{\text{(b)}%
}{\leq}\frac{q-1}{3q}\left(  \mathbb{E}\left[  X^{4}\right]  -3\sigma
^{4}\right)  \label{eqn_boundnormal}%
\end{equation}
to further cap the distance. Equality (a) follows from $\mathbb{E}\left[
\frac{1}{q}\left\|  DX\right\|  _\mathfrak{H}\right]  =\mathbb{E}\left[  g_{X}\right]
=\sigma^{2}$ and inequality (b) from Lemma 3.5 in \cite{NPsurv}. These are
quite important and known results which yield CLTs for
functionals on a fixed Wiener chaos. For instance, if we have a sequence
$\left\{  X_{n}\right\}  =\left\{  I_{q}\left(  f_{n}\right)  \right\}  $
where $\mathbb{E}\left[  \left(  I_{q}\left(  f_{n}\right)  \right)
^{2}\right]  \rightarrow\sigma^{2}$, then the following conditions are equivalent:
\begin{enumerate}
\item $X_{n}\rightarrow Z$ in distribution;

\item $\mathbb{E}\left[  X_{n}^{4}\right]  \rightarrow3\sigma^{4}$;

\item $\left\|  f_{n}\otimes_{r}f_{n}\right\|  _{H^{\otimes\left(
2q-2r\right)  }}\rightarrow0$ for all $r=1,\ldots,q-1$;

\item $\left\|  DX_{n}\right\|  _\mathfrak{H}^{2}\rightarrow q\sigma^{2}$ in
$L^{2}\left(  \Omega\right)  $;

\item $\left\|  D^{2}X_{n}\otimes_{1}D^{2}X_{n}\right\|  _{\mathfrak{H}^{\otimes2}}%
^{2}\rightarrow0$ in $L^{2}\left(  \Omega\right)  $.
\end{enumerate}

\bigskip

See \cite{NuP} for $\left(  1\right)  \Longleftrightarrow\left(  2\right)
\Longleftrightarrow\left(  3\right)  $, \cite{NuO} for $\left(  1\right)
\Longleftrightarrow\left(  4\right)  $, and \cite{NPR} for $\left(  1\right)
\Longleftrightarrow\left(  5\right)  $. These in some sense highlight the
tightness of inequality (\ref{eqn_inequality2}) with the help of bounds like
(\ref{eqn_Poincare}) and (\ref{eqn_boundnormal}).

\begin{corollary}
\label{cor_Xchaos}If $X_{n}=I_{q}\left(  f_{n}\right)  $ with $q\geq1$, then
condition \textbf{3} in Corollary \ref{cor_boundsequence} can be replaced by
$\mathbb{E}\left[  \left\|  DX_{n}\right\|  ^{4}_\mathfrak{H}\right]  \rightarrow
q^{2}\mathbb{E}\left[  g_{\ast}^{2}\left(  Z\right)  \right]  $.
\end{corollary}
\begin{proof}
This is a direct consequence of $\left\langle DX_{n},-DL^{-1}X_{n}%
\right\rangle _\mathfrak{H}=\frac{1}{q}\left\|  DX_{n}\right\|  ^{2}_\mathfrak{H}$ and
$\mathbb{E}\left[  g_{Z}^{2}\right]  =\mathbb{E}\left[  g_{\ast}^{2}\left(
Z\right)  \right]  $.
\end{proof}

\bigskip

From this and Corollary \ref{cor_LsequenceP}, we have the following result for
the convergence in a fixed Wiener chaos to a Pearson random variable.

\begin{corollary}
Let $Z$ be Pearson with $g_{\ast}\left(  z\right)  =\alpha z^{2}+\beta
z+\gamma$ in its support. Fix $q\geq2$. Suppose $X_{n}=I_{q}\left(
f_{n}\right)  $ and $\mathbb{E}\left[  X_{n}^{2}\right]  \rightarrow
\frac{\gamma}{1-\alpha}$. If $\left\|  DX_{n}\right\|  ^{2}_\mathfrak{H}-q\alpha X_{n}%
^{2}-q\beta X_{n}\rightarrow q\gamma$ in $L^{1}\left(  \Omega\right)  $, then $X_{n}\rightarrow Z$ in distribution.
\end{corollary}

\begin{remark}
Special cases of the above corollary are known results.

\begin{itemize}
\item Let $Z$ be Normal with variance $1$, i.e. $g_{\ast}\left(  z\right)
=1$. Suppose $\mathbb{E}\left[  X_{n}^{2}\right]  \rightarrow1$. Then $X_{n}\rightarrow
Z$ in distribution if $\left\|  DX_{n}\right\|_\mathfrak{H}  ^{2}\rightarrow q$ in
$L^{2}\left(  \Omega\right)  $. See \cite{NuO}.

\item Let $Z$ be Gamma with $g_{\ast}(z)=\left(  2z+2v\right)  _{+}$, i.e.
$\beta=2$ and $\gamma=2v$, where the parameters are chosen for consistency
with the discussion in \cite{NPnonclt}. Suppose $\mathbb{E}\left[  X_{n}%
^{2}\right]  \rightarrow2v$. Then $X_{n}\rightarrow Z$ in distribution if
$\left\|  DX_{n}\right\|_\mathfrak{H}  ^{2}-2qX_{n}\rightarrow2qv$ in $L^{2}\left(
\Omega\right)  $.
\end{itemize}
\end{remark}

\bigskip

The result in the first item of this remark is known as the Nualart$-$Ortiz-Latorre criterion. In \cite{VT}, the authors used it to prove that
$$C\sqrt{N}\ln(N)\left(\widehat{H}_N-H\right)\xrightarrow[N\to\infty]{}\mathcal{N}(0,1)$$
where $\widehat{H}_N$ is an estimator of the Hurst parameter $H$ for fBm when $H\in\bigl	(\frac{1}{2},\frac{1}{3}\bigr)$ (see \cite{VT} for details).

\subsection{Bilinear functionals of Gaussian subordinated fields}

Let $X_{t}$ be a centered Gaussian stationary process with covariance function
$C\left(  t\right)  =\mathbb{E}\left[  X_{0}X_{t}\right]  =\mathbb{E}\left[
X_{s}X_{s+t}\right]  $. Let $f:\mathbb{R}\rightarrow\mathbb{R}$ be a $\mathcal{C}^{2}$ non-constant
function such that with $Z\sim\mathcal{N}\left(  0,C\left(  0\right)  \right)
$, $\mathbb{E}\left[  \left|  f\left(  Z\right)  \right|  \right]  <\infty$
and $\mathbb{E}\left[  \left|  f^{\prime\prime}\left(  Z\right)  \right|
^{4}\right]  <\infty$. Write $\mu_{f}=\mathbb{E}\left[  f\left(  Z\right)  \right]  $. It was proved in \cite{NPR} and \cite{Viquez} (under mild conditions) that
$$H_T:=\frac{1}{V(T)}\int_{[0,T]}\left(  f\left(X_{s}\right)  -\mu_{f}\right)ds\xrightarrow[T\to\infty]{}\mathcal{N}(0,\Sigma^2)$$
where $V(T)$ is a normalization function with specific properties. In particular, for the case of the increments of fBm ($X_t=B^H_{t+1}-B^H_t$) with Hurst parameter $H\in(1/2,1)$, we have $V(T)=T^H$ and $\Sigma^{2}=2\bigl(\mathbb{E}[Zf(Z)]\bigr)^{2}$. As an ilustration of how to employ this tool, and to emphasize the advantage of using inequality (\ref{eqn_inequality3}) over (\ref{eqn_inequality1}), we will prove that $F_T:=\left(\frac{H_T}{\Sigma}\right)^2-\mathbb{E}\left[\left(\frac{H_T}{\Sigma}\right)^2\right]$ converges to a (centered) chi-squared r.v. as $T\rightarrow\infty$, for the case of the increments of fBm when $H\in(1/2,1)$.

\bigskip

\noindent {\bf Preliminary computations and notation:}\\
Following the setup in \cite{Nualart} (Section 5.1.3), we can write $B^H_{t+1}-B^H_t=I_1\bigl(K_H(t,\cdot)\bigr)$ where $K_H(t,s)=c_H s^{\frac{1}{2}-H}\int_{s\vee t}^{t+1}(u-s)^{H-\frac{3}{2}}u^{H-\frac{1}{2}}du\mathbf{1}_{\left[  0,t+1\right]}$ and $c_H=\left[\frac{H(2H-1)}{\beta(2-2H,H-\frac{1}{2})}\right]^{1/2}$. The integral $I_1$ is with respect to a Wiener process $W$ generating the same filtration as $B^H$, with the two processes related by $B_t = \int_0^t K_H(t,s)dW_s$.

Let's make the simplifying assumption $C\left(  0\right)  =1$ so $Z$,
$X_{s}\sim\mathcal{N}\left(  0,1\right)  $. To simplify notation, we will write $\int_{T^{q}}$ for $\int_{\left[0,T\right]  ^{q}}$ and $\mathbf{K}_{t}$ for $K_H(t,\cdot)$. Also let $\mathfrak{H}=L^2([0,T])$. Then $DX_{t}=\mathbf{K}_{t}$ and $\mathbb{E}[X_sX_t]=\mathbb{E}\left[I_1(\mathbf{K}_{s})I_1(\mathbf{K}_{t})\right]=\left\langle \mathbf{K}_{s},\mathbf{K}_{t}\right\rangle _{\mathfrak{H}}=C\left(  \left|  s-t\right|\right)  =:C_{st}$. Define for $T>0$ and $s,t\in\left[0,T\right]  $ the functionals
\begin{align*}
F_{st} &  =\left(  f\left(  X_{s}\right)  -\mu_{f}\right)  \left(  f\left(
X_{t}\right)  -\mu_{f}\right)  \\
\widetilde{F}_{T} &  =\left(  \frac{H_T}{\Sigma}\right)  ^{2}=\frac{1}{\Sigma^{2}T^{2H}}%
\int_{T^{2}}\left(  f\left(  X_{s}\right)  -\mu_{f}\right)  \left(  f\left(
X_{t}\right)  -\mu_{f}\right)  dsdt=\frac{1}{\Sigma^{2}T^{2H}}\int_{T^{2}%
}F_{st}dsdt\\
F_{T} &  =\widetilde{F}_{T}-\mathbb{E}\left[  \widetilde{F}_{T}\right]  \text{.}%
\end{align*}
Let $\mathbf{s=}\left(  s_{1},\ldots,s_{P}\right)  \in\left[  0,T\right]
^{P}$. We will use $\epsilon(s_i,s_j)$ to denote a nonnegative integer exponent
indexed by a pair of variables from $\mathbf{s}$. Define
\[
L\left(  T\right)  =\frac{1}{T^{PH}}\int_{T^{N}}\prod_{s_{i}\neq s_{j}}\left|
C_{s_{i}s_{j}}^{\epsilon(s_i,s_j)}\right|  d\mathbf{s=}\frac
{1}{T^{PH}}\left\|  \prod_{s_{i}\neq s_{j}}C_{s_{i}s_{j}}^{\epsilon(s_i,s_j)}\right\|  _{1}%
\]
where inside the integral is the product of $Q=\binom{P}{2}$ factors. For example, if $P=4$ and
$\mathbf{s=}\left(  s,t,u,v\right)  $, then
\begin{equation}
L\left(  T\right)  =\frac{1}{T^{4H}}\int_{T^{4}}\left|  C_{st}^{\epsilon(s,t)}C_{su}^{\epsilon(s,u)}C_{sv}^{\epsilon(s,v)}%
C_{tu}^{\epsilon(t,u)}C_{tv}^{\epsilon(t,v)}C_{uv}^{\epsilon(u,v)}\right|  dsdtdudv\text{.}\label{ex_eqn_LT}%
\end{equation}
Recall that $C_{st}^{\epsilon(s,t)}=\bigl(C\left(
\left|  s-t\right|  \right)\bigr)^{\epsilon(s,t)}  $, so the integration in (\ref{ex_eqn_LT}) is
being done only on the subscripts and not on the superscripts (since these are
fixed exponents indexed only by the variables over which we're integrating).

\begin{proposition} 
\label{calculations}
With the previous notation,
\begin{enumerate}

\item For $q\geq1$, take $c_{q}q!=\mathbb{E}\left[  H_{q}\left(  Z\right) f\left(  Z\right)  \right]$ where  $H_q$ is the $q^{\text{th}}$ Hermite polynomial. Then,
$$f\left(  X_t\right)  =\mu_{f}+\sum_{q=1}^{\infty}c_{q}I_{q}\left( \mathbf{K}_{t}^{\otimes q} \right).$$

\item $F_{st}$ has Wiener chaos decomposition
\begin{align*}
F_{st}  &  =\sum_{n=2}^{\infty}\sum_{m=1}^{n-1}c_{m}c_{n-m}I_{n}\left(
\mathbf{K}_{s}^{\otimes m}\otimes\mathbf{K}_{t}^{\otimes\left(  n-m\right)
}\right)  +\sum_{n=0}^{\infty}\sum_{m=0}^{n}\sum_{r=1}^{\infty}d\left(
m+r,n-m+r,r\right)  I_{n}\left(  \mathbf{K}_{s}^{\otimes m}\otimes
\mathbf{K}_{t}^{\otimes\left(  n-m\right)  }\right)  C_{st}^{r}\\
&  =\sum_{n=0}^{\infty}\sum_{m=0}^{n}\sum_{r=0}^{\infty}e\left(  m,n,r\right)
d\left(  m+r,n-m+r,r\right)  I_{n}\left(  \mathbf{K}_{s}^{\otimes m}%
\otimes\mathbf{K}_{t}^{\otimes\left(  n-m\right)  }\right)  C_{st}^{r}%
\end{align*}
where $d\left(  k,j,r\right)  =c_{k}c_{j}r!\binom{k}{r}\binom{j}{r}$, and
\[
e\left(  m,n,r\right)  =\left\{
\begin{array}
[c]{cl}%
0 & \text{if }r=0\text{ and }m\in\left\{  0,n\right\}  \text{ }\\
1 & \text{otherwise}%
\end{array}
\right.  \text{.}%
\]

\item For $0\leq a$, $b\leq n$,%
\[
\left\langle \mathbf{K}_s^{\otimes a}\widetilde{\otimes
}\mathbf{K}_t^{\otimes\left(  n-a\right)  }%
,\mathbf{K}_u^{\otimes b}\widetilde{\otimes}\mathbf{K}%
_v^{\otimes\left(  n-b\right)  }\right\rangle
_{\mathfrak{H}^{\otimes n}}=\frac{1}{\binom{n}{a}\binom{n}{b}}\sum_{p}\binom
{n}{p,a-p,b-p,n-a-b+p}C_{su}^{p}C_{sv}^{a-p}C_{tu}^{b-p}C_{tv}^{n-a-b+p}%
\]
where the summation is taken over all $p$ for which $\max\left(
0,a+b-n\right)  \leq p\leq\min\left(  a,b\right)  $.

\item Fix an integer $P\geq2$. Let $S=\sum \epsilon\left(  s_{i},s_{j}\right)  $ be the sum of the
exponents in $L\left(  T\right)  $.
\begin{itemize}

\item If $S>P/2$, then $\lim_{T\rightarrow
\infty}L\left(  T\right)  =0$.

\item If $S=P/2$, then $\varlimsup_{T\rightarrow
\infty}L\left(  T\right)  <\infty$.

\end{itemize}
\end{enumerate}

\end{proposition}

\noindent \begin{proof}
To prove the first point we expand $f$ in terms of Hermite polynomials:
$$f\left(  z\right)  =c_0+\sum_{q=1}^{\infty}c_{q}H_{q}\left(  z\right)  $$
with $c_{q}q!=\mathbb{E}\left[  H_{q}\left(  Z\right)f\left(  Z\right)  \right]  $ for all $q\geq0$. Since for any $h\in L^2([0,T])$ we have the relation $H_q(I_1(h))=I_q(h^{\otimes q})$, then the result follows.

\bigskip

For the second point we have that $\mathbf{K}_t^{\otimes k}\otimes_{r}\mathbf{K}_s^{\otimes j}=\left\langle \mathbf{K}_t,\mathbf{K}_s\right\rangle _{\mathfrak{H}}^{r}\left(
\mathbf{K}_t^{\otimes\left(  k-r\right)  }\otimes
\mathbf{K}_s^{\otimes\left(  j-r\right)  }\right)
=C_{ts}^{r}\left(  \mathbf{K}_t^{\otimes\left(
k-r\right)  }\otimes\mathbf{K}_s^{\otimes\left(
j-r\right)  }\right)  $. Therefore, from the previous point and the product formula (\ref{eqn_productW}),
\begin{align*}
F_{st}  &  =\sum_{k=1}^{\infty}c_{k}H_{k}\left(  X_{s}\right)  \sum
_{j=1}^{\infty}c_{j}H_{j}\left(  X_{t}\right)  =\sum_{k,j=1}^{\infty}%
c_{k}c_{j}I_{k}\left(  \mathbf{K}_s^{\otimes k}\right)
I_{j}\left(  \mathbf{K}_t^{\otimes j}\right)\hspace{1cm}  \\
&  =\sum_{k,j=1}^{\infty}c_{k}c_{j}\sum_{r=0}^{k\wedge j}r!\binom{k}{r}%
\binom{j}{r}I_{k+j-2r}\left(  \mathbf{K}_s^{\otimes
k}\otimes_{r}\mathbf{K}_t^{\otimes j}\right)
=\sum_{k,j=1}^{\infty}c_{k}c_{j}\sum_{r=0}^{k\wedge j}r!\binom{k}{r}\binom
{j}{r}C_{st}^{r}I_{k+j-2r}\left(  \mathbf{K}_s%
^{\otimes\left(  k-r\right)  }\otimes\mathbf{K}_t^{\otimes\left(  j-r\right)  }\right)  \text{.}%
\end{align*}
Write $c_{st}\left(  k,j,r\right)  =d\left(  k,j,r\right)  C_{st}%
^{r}I_{k+j-2r}\left(  \mathbf{K}_s^{\otimes\left(
k-r\right)  }\otimes\mathbf{K}_t^{\otimes\left(
j-r\right)  }\right)  $ where $d\left(  k,j,r\right)  =c_{k}c_{j}r!\binom
{k}{r}\binom{j}{r}$. Then%
\[
F_{st}=\sum_{k=1}^{\infty}\sum_{j=1}^{\infty}\sum_{r=0}^{k\wedge j}%
c_{st}\left(  k,j,r\right)  =\sum_{k=1}^{\infty}\sum_{j=1}^{\infty}%
c_{st}\left(  k,j,0\right)  +\sum_{k=1}^{\infty}\sum_{j=1}^{\infty}\sum
_{r=1}^{k\wedge j}c_{st}\left(  k,j,r\right).
\]
Applying Fubini's theorem for sums we have,
\begin{align*}
\sum_{k=1}^{\infty}\sum_{j=1}^{\infty}c_{st}\left(  k,j,0\right)   &
=\sum_{m=1}^{\infty}\sum_{j=1}^{\infty}c_{m}c_{j}I_{m+j}\left(  \mathbf{K}%
_s^{\otimes m}\otimes\mathbf{K}_t^{\otimes j}\right)  =\sum_{m=1}^{\infty}\sum_{n=m+1}^{\infty}c_{m}%
c_{n-m}I_{n}\left(  \mathbf{K}_s^{\otimes m}%
\otimes\mathbf{K}_t^{\otimes\left(  n-m\right)  }\right)
\\
&  =\sum_{n=2}^{\infty}\sum_{m=1}^{n-1}c_{m}c_{n-m}I_{n}\left(  \mathbf{K}%
_s^{\otimes m}\otimes\mathbf{K}_t^{\otimes\left(  n-m\right)  }\right)
\end{align*}
and%
$$\sum_{j=1}^{\infty}\sum_{r=1}^{k\wedge j}    =\sum_{j=1}^{k}\sum_{r=1}%
^{j}+\sum_{j=k+1}^{\infty}\sum_{r=1}^{k}=\sum_{r=1}^{k}\sum_{j=r}^{k}%
+\sum_{r=1}^{k}\sum_{j=k+1}^{\infty}=\sum_{r=1}^{k}\sum_{j=r}^{\infty}.$$
Therefore,
\begin{align*}
\sum_{k=1}^{\infty}\sum_{j=1}^{\infty}\sum_{r=1}^{k\wedge j}c_{st}\left(
k,j,r\right)   &  =\sum_{k=1}^{\infty}\sum_{r=1}^{k}\sum_{j=r}^{\infty}%
c_{st}\left(  k,j,r\right)  =\sum_{r=1}^{\infty}\sum_{m=0}^{\infty}\sum
_{n=m}^{\infty}c_{st}\left(  m+r,n-m+r,r\right) \\
&  =\sum_{r=1}^{\infty}\sum_{n=0}^{\infty}\sum_{m=0}^{n}c_{st}\left(
m+r,n-m+r,r\right) \\
&  =\sum_{n=0}^{\infty}\sum_{m=0}^{n}\sum_{r=1}^{\infty}d\left(
m+r,n-m+r,r\right)  C_{st}^{r}I_{n}\left(  \mathbf{K}_s^{\otimes m}\otimes\mathbf{K}_t^{\otimes\left(
n-m\right)  }\right)  \text{,}%
\end{align*}
establishing the second point.

\bigskip

Point 3 requires counting the possible combinations in the inner product. Note first that the symmetric tensor product $\mathbf{K}_t^{\otimes
a}\widetilde{\otimes}\mathbf{K}_s^{\otimes\left(  n-a\right)
}$ has $\binom{n}{a}$ distinct terms. Take any particular term $\alpha$ and
list down all its $n$ factors $\mathbf{K}_t$ and
$\mathbf{K}_s$ in the order in which they appear. Now
take any term $\beta$ from $\mathbf{K}_u^{\otimes
b}\widetilde{\otimes}\mathbf{K}_v^{\otimes\left(  n-b\right)
}$ and list down all its factors (in order) below those of $\alpha$. Let $p$
be the number of $\left(  \mathbf{K}_t,\mathbf{K}%
_u\right)  $ pairings. Thus, the number of pairings of
the type $\left(  \mathbf{K}_t,\mathbf{K}_v\right)  $, $\left(  \mathbf{K}_s%
,\mathbf{K}_u\right)  $ and $\left(  \mathbf{K}_s,\mathbf{K}_v\right)  $ are $a-p$, $b-p$ and
$n-a-b+p$, respectively. Finally, the number of pairs $\left(  \alpha
,\beta\right)  $ which have $p$ matching $\mathbf{K}_t$
and $\mathbf{K}_u$ is $\binom{n}{p,a-p,b-p,n-a-b+p}$.

\bigskip

Finally, to prove the fourth point we make use of Proposition 3 (point 4) in \cite{Viquez} which states that,
\[
\int_{T^{2}}\left|  C_{st}^{\epsilon}\right|  dsdt=O\left(  \frac{T^{2H}}{T^{\left(
2-2H\right)  \left(  \epsilon-1\right)  }}\right)
\]
so
\[
\left\|  C_{st}^{\epsilon}\right\|  _{Q}=\left[  \int_{T^{P}}\left|  C_{st}%
^{Q\epsilon}\right|  d\mathbf{s}\right]  ^{1/Q}=\left[  T^{P-2}\int_{T^{2}}\left|
C_{st}^{Q\epsilon}\right|  dsdt\right]  ^{1/Q}=O\left(  \frac{T^{\left(
P-2+2H\right)  /Q}}{T^{\left(  2-2H\right)  \left(  Q\epsilon-1\right)  /Q}}\right)
\text{.}%
\]
By the generalized H\"{o}lder inequality for integrals,
\begin{align*}
L\left(  T\right)   &  =\frac{1}{T^{PH}}\left\|  \prod_{s_{i}\neq s_{j}%
}C_{s_{i}s_{j}}^{\epsilon(s_{i},s_{j})}\right\|  _{1}\leq\frac{1}%
{T^{PH}}\prod_{s_{i}\neq s_{j}}\left\|  C_{s_{i}s_{j}}^{\epsilon(s_{i},s_{j})}\right\|  _{Q}\\
&  =O\left(  \frac{1}{T^{PH}}\prod_{s_{i}\neq s_{j}}\frac{T^{\left(
P-2+2H\right)  /Q}}{T^{\left(  2-2H\right)  \left(  Q\epsilon(s_{i},s_{j})  -1\right)  /Q}}\right)  =O\left(  \frac{1}{T^{\left(
1-H\right)  \left(  2S-P\right)  }}\right)
\end{align*}
and the result follows.
\end{proof}

\bigskip

\begin{theorem}
Suppose $f:\mathbb{R}\to\mathbb{R}$ is of class $\mathcal{C}^2$ such that $\Sigma^2:=2\bigl(\mathbb{E}[Zf(Z)]\bigr)^2\neq0$. Then as $T\to\infty$,
$$F_T\xrightarrow[]{\text{Law}}\chi^2.$$
\end{theorem}

\bigskip

\noindent \begin{proof}
Throughout this proof we will use the symbol $\sim$ to relate two expressions which have the same limit as $T\rightarrow\infty$. Also, $\boldsymbol{\kappa}$ will represent a constant whose value may change from one equation to another and might depend on the summation indices.

\bigskip

Since $\chi^2$ is a member of the Pearson family, we can make use of Corollary \ref{cor_Pearson} to prove this theorem. This will require proving convergence of the second moment of $F_T$ and of $g_{F_T}$, plus the moment convergence of $F_Tg_{F_T}$. We take note of the following facts: $\Sigma^{2}=2c_{1}^{2}$ ($c_1=\mathbb{E}[Zf(Z)]$), and $T^{-2H}\int_{T^2}C_{ts}dtds\to2$ as $T\to\infty$ (see Proposition 3 in \cite{Viquez}). 

\begin{itemize}

\item {\bf Convergence of the Second Moment of $\boldsymbol{F_T}$:}\\
From point 2 of Proposition \ref{calculations},%

\begin{equation}
DF_{st}=\sum_{n=1}^{\infty}\sum_{m=0}^{n}\sum_{r=0}^{\infty}ne\left(
m,n,r\right)  d\left(  m+r,n-m+r,r\right)  I_{n-1}\left(  \mathbf{K}%
_{s}^{\otimes m}\otimes\mathbf{K}_{t}^{\otimes\left(  n-m\right)  }\right)
C_{st}^{r} \label{ex_eqn_DFst}%
\end{equation}
and%
\begin{equation}
-DL^{-1}F_{uv}=\sum_{N=1}^{\infty}\sum_{M=0}^{N}\sum_{R=0}^{\infty}e\left(
M,N,R\right)  d\left(  M+R,N-M+R,R\right)  I_{N-1}\left(  \mathbf{K}%
_{u}^{\otimes M}\otimes\mathbf{K}_{v}^{\otimes\left(  N-M\right)  }\right)
C_{uv}^{R}\text{.} \label{ex_eqn_DLFuv}%
\end{equation}
Each indicated multiple integral of order $n-1$, with kernel in $\mathfrak{H}%
^{\otimes n}$, is to be interpreted as
\begin{align}
I_{n-1}\left(  \mathbf{K}_{s}^{\otimes m}\otimes\mathbf{K}_{t}^{\otimes\left(
n-m\right)  }\right)   &  =\frac{m}{n}I_{n-1}\left(  \mathbf{K}_{s}%
^{\otimes\left(  m-1\right)  }\otimes\mathbf{K}_{t}^{\otimes\left(
n-m\right)  }\right)  \mathbf{K}_{s}+\frac{n-m}{n}I_{n-1}\left(
\mathbf{K}_{s}^{\otimes m}\otimes\mathbf{K}_{t}^{\otimes\left(  n-m-1\right)
}\right)  \mathbf{K}_{t}\nonumber\\
&  =\sum_{a=m-1}^{m}\frac{k^{\left(  m,n,a\right)  }}{n}I_{n-1}\left(
\mathbf{K}_{s}^{\otimes a}\otimes\mathbf{K}_{t}^{\otimes\left(  n-1-a\right)
}\right)  \mathbf{f}_{st}^{\left(  m,a\right)  } \label{ex_eqn_In1}%
\end{align}
where $k^{\left(  m,n,a\right)  }\mathbf{f}_{st}^{\left(  m,a\right)
}=\left\{
\begin{array}
[c]{ll}%
m\mathbf{K}_{s} & \text{if }a=m-1\geq0\\
\left(  n-m\right)  \mathbf{K}_{t} & \text{if }a=m
\end{array}
\right.  $.

Therefore,$\ \mathbb{E}\left[  \left\langle DF_{st},-DL^{-1}%
F_{uv}\right\rangle _{\mathfrak{H}}\right]  $ will be a summation containing
terms of the form
\begin{align}
&  \boldsymbol{\kappa}\mathbb{E}\left[  \left\langle I_{n-1}\left(  \mathbf{K}_{s}^{\otimes
a}\otimes\mathbf{K}_{t}^{\otimes\left(  n-1-a\right)  }\right)  \mathbf{f}%
_{st}^{\left(  m,a\right)  },I_{N-1}\left(  \mathbf{K}_{u}^{\otimes A}%
\otimes\mathbf{K}_{v}^{\otimes\left(  N-1-A\right)  }\right)  \mathbf{f}%
_{uv}^{\left(  M,A\right)  }\right\rangle _{\mathfrak{H}}\right]  C_{st}%
^{r}C_{uv}^{R}\label{ex_eqn_correxpsum1}\\
&  =\mathbf{1}_{\{n=N\}}\boldsymbol{\kappa}\left\langle \mathbf{K}_{s}^{\otimes a}\widetilde{\otimes
}\mathbf{K}_{t}^{\otimes\left(  n-1-a\right)  },\mathbf{K}_{u}^{\otimes
A}\widetilde{\otimes}\mathbf{K}_{v}^{\otimes\left(  n-1-A\right)  }\right\rangle
_{\mathfrak{H}^{\otimes\left(  n-1\right)  }}\left\langle \mathbf{f}%
_{st}^{\left(  m,a\right)  },\mathbf{f}_{uv}^{\left(  M,A\right)
}\right\rangle _{\mathfrak{H}}C_{st}^{r}C_{uv}^{R}\label{ex_eqn_correxpsum2}\\
&  =\mathbf{1}_{\{n=N\}}\sum_{p}\boldsymbol{\kappa}C_{su}^{p}C_{sv}^{a-p}C_{tu}^{A-p}C_{tv}%
^{n-1-a-A+p}\left\langle \mathbf{f}_{st}^{\left(  m,a\right)  },\mathbf{f}%
_{uv}^{\left(  M,A\right)  }\right\rangle _{\mathfrak{H}}C_{st}^{r}C_{uv}^{R}
\label{ex_eqn_correxpsum3}%
\end{align}
where from (\ref{ex_eqn_correxpsum2}) to (\ref{ex_eqn_correxpsum3}), we used
the third point of Proposition \ref{calculations}.

\bigskip

Observe that $\left\langle \mathbf{f}_{st}^{\left(  m,a\right)  }%
,\mathbf{f}_{uv}^{\left(  M,A\right)  }\right\rangle _{\mathfrak{H}}$ is any of
$C_{su}$, $C_{sv}$, $C_{tu}$ or $C_{tv}$. In (\ref{ex_eqn_correxpsum3}), the
exponents of all six types of correlations then add up to $S=n+r+R=N+r+R$ (at
this point, always take $N=n$, otherwise the term is 0). In the summations (\ref{ex_eqn_DFst}) and
(\ref{ex_eqn_DLFuv}) appearing in $\left\langle DF_{st},-DL^{-1}%
F_{uv}\right\rangle _{\mathfrak{H}}$, since $e\left(  m,n,r\right)  e\left(
M,N,R\right)  =0$ if $S\leq1$, the remaining terms are those for which
$S\geq2$. Therefore, using the moments formula of Proposition \ref{lem_expec},
\begin{align}
\mathbb{E}\left[  F_{T}^{2}\right]   &  =\mathbb{E}\left[  g_{F_{T}}\right]
=\mathbb{E}\left[  \left\langle DF_{T},-DL^{-1}F%
_{T}\right\rangle _{\mathfrak{H}}\right]  =\frac{1}{\Sigma^{4}T^{4H}}%
\int_{T^{4}}\mathbb{E}\left[  \left\langle DF_{st},-DL^{-1}F_{uv}\right\rangle
_{\mathfrak{H}}\right]  dsdtdudv\label{ex_eqn_FT2a}\\
&  \sim\frac{1}{\Sigma^{4}T^{4H}}\int_{T^{4}%
}\mathbb{E}\left[  \left\langle 2d\left(  1,1,0\right)  I_{1}\left(
\mathbf{K}_{s}\otimes\mathbf{K}_{t}\right)  ,d\left(  1,1,0\right)
I_{1}\left(  \mathbf{K}_{u}\otimes\mathbf{K}_{v}\right)  \right\rangle
_{\mathfrak{H}}\right]  dsdtdudv\label{ex_eqn_FT2b}%
\end{align}
where in (\ref{ex_eqn_FT2b}), we applied point 4 of Proposition \ref{calculations} ($P=4$) on
(\ref{ex_eqn_correxpsum3}) and (\ref{ex_eqn_FT2a}): for those terms
contributed by $\mathbb{E}\left[  \left\langle DF_{st},-DL^{-1}F_{uv}%
\right\rangle _{\mathfrak{H}}\right]  $ where $S>2$, the limit is $0$. The
limit in (\ref{ex_eqn_FT2a}) is the nonzero value we get for the remaining
case $S=2$; specifically, $\left(  n,N,m,M,r,R\right)  =\left(
2,2,1,1,0,0\right)  $. Since%
\begin{multline}
\left\langle I_{1}\left(  \mathbf{K}_{s}\otimes\mathbf{K}_{t}\right)
,I_{1}\left(  \mathbf{K}_{u}\otimes\mathbf{K}_{v}\right)  \right\rangle
_{\mathfrak{H}}= \frac{I_{1}\left(  \mathbf{K}_{s}\right)  I_{1}\left(
\mathbf{K}_{u}\right)  C_{tv}+I_{1}\left(  \mathbf{K}_{t}\right)  I_{1}\left(
\mathbf{K}_{u}\right)  C_{sv}}{4} \\ + \frac{I_{1}\left(  \mathbf{K}_{s}\right)  I_{1}\left(
\mathbf{K}_{v}\right)  C_{tu}+I_{1}\left(  \mathbf{K}_{t}\right)  I_{1}\left(
\mathbf{K}_{v}\right)  C_{su}}{4}\text{,}\label{ex_eqn_I1kernel2}%
\end{multline}
then%
\[
\mathbb{E}\left[  F_{T}^{2}\right]  \sim\frac{2c_{1}^{4}}{\Sigma^{4}T^{4H}%
}\int_{T^{4}}\frac{1}{4}\left[  2C_{su}C_{tv}+2C_{sv}C_{tu}\right]
dsdtdudv\rightarrow\frac{2c_{1}^{4}}{\Sigma^{4}}\left(  2\right)  ^{2}%
=2\text{.}%
\]
Therefore,
$$\boxed{\mathbb{E}\left[  F_{T}^{2}\right]  \to 2}$$

\item {\bf Convergence of $\boldsymbol{F_Tg_{F_T}}$:}\\
Notice that
\[
\mathbb{E}\left[  F_{T}g_{F_{T}}\right]
=\mathbb{E}\left[  \widetilde{F}_{T}g_{F_{T}}\right]  -\mathbb{E}\left[
\widetilde{F}_{T}\right]  \mathbb{E}\left[  g_{F_{T}}\right]  \sim\mathbb{E}%
\left[  \widetilde{F}_{T}g_{F_{T}}\right]  -2
\]
since%
\[
\mathbb{E}\left[  \widetilde{F}_{T}\right]  =\frac{1}{\Sigma^{2}T^{2H}}\int
_{T^{2}}\mathbb{E}\left[  F_{st}\right]  dsdt=\sum_{r=1}^{\infty}%
\frac{d\left(  r,r,r\right)  }{\Sigma^{2}T^{2H}}\int_{T^{2}}C_{st}^{r}%
dsdt\sim\frac{d\left(  1,1,1\right)  }{\Sigma^{2}T^{2H}}\int_{T^{2}}%
C_{st}dsdt\rightarrow\frac{2c_{1}^{2}}{\Sigma^{2}}=1\text{.}%
\]
Now we need to investigate%
\begin{equation}
\mathbb{E}\left[  \widetilde{F}_{T}g_{F_{T}}\right]  =\frac{1}{\Sigma^{6}T^{6H}%
}\int_{T^{6}}\mathbb{E}\left[  \left\langle DF_{st},-DL^{-1}F_{uv}%
\right\rangle _{\mathfrak{H}}F_{wx}\right]  dsdtdudvdwdx\text{.}%
\label{ex_eqn_E3a}%
\end{equation}
The expression inside the expectation is a summation with generic term%
\begin{align}
\hspace{-2cm}&\left\langle I_{n-1}\left(  \mathbf{K}_{s}^{\otimes m}\otimes\mathbf{K}%
_{t}^{\otimes\left(  n-m\right)  }\right)  C_{st}^{r},I_{N-1}\left(
\mathbf{K}_{u}^{\otimes M}\otimes\mathbf{K}_{v}^{\otimes\left(  N-M\right)
}\right)  C_{uv}^{R}\right\rangle _{\mathfrak{H}}I_{n^{\prime}}\left(
\mathbf{K}_{w}^{\otimes m^{\prime}}\otimes\mathbf{K}_{x}^{\otimes\left(
n^{\prime}-m^{\prime}\right)  }\right)  C_{wx}^{r^{\prime}}\nonumber\\
\hspace{-2cm} & \hspace{1cm} =\sum_{a,A}\boldsymbol{\kappa}I_{n-1}\left(  \mathbf{K}_{s}^{\otimes a}\otimes\mathbf{K}%
_{t}^{\otimes\left(  n-1-a\right)  }\right)  I_{N-1}\left(  \mathbf{K}%
_{u}^{\otimes A}\otimes\mathbf{K}_{v}^{\otimes\left(  N-1-A\right)  }\right)
I_{n^{\prime}}\left(  \mathbf{K}_{w}^{\otimes m^{\prime}}\otimes\mathbf{K}%
_{x}^{\otimes\left(  n^{\prime}-m^{\prime}\right)  }\right)  \left\langle
\mathbf{f}_{st}^{\left(  m,a\right)  },\mathbf{f}_{uv}^{\left(  M,A\right)
}\right\rangle _{\mathfrak{H}}C_{st}^{r}C_{uv}^{R}C_{wx}^{r^{\prime}%
}\label{ex_eqn_E3generic}%
\end{align}
where $a\in\left\{  m-1,m\right\}  $, $A\in\left\{  M-1,M\right\}  $ and
$\left\langle \mathbf{f}_{st}^{\left(  m,a\right)  },\mathbf{f}_{uv}^{\left(
M,A\right)  }\right\rangle _{\mathfrak{H}}$ could be any of $C_{su}$, $C_{sv}$,
$C_{tu}$ or $C_{tv}$ (we used (\ref{ex_eqn_In1}) here).%
\begin{align}
\hspace{-2cm} &  \mathbb{E}\left[  I_{n-1}\left(  \mathbf{K}_{s}^{\otimes a}\otimes
\mathbf{K}_{t}^{\otimes\left(  n-1-a\right)  }\right)  I_{N-1}\left(
\mathbf{K}_{u}^{\otimes A}\otimes\mathbf{K}_{v}^{\otimes\left(  N-1-A\right)
}\right)  I_{n^{\prime}}\left(  \mathbf{K}_{w}^{\otimes m^{\prime}}%
\otimes\mathbf{K}_{x}^{\otimes\left(  n^{\prime}-m^{\prime}\right)  }\right)
\right]  \nonumber\\
\hspace{-2cm}&  =\sum_{f=0}^{\left(  n\wedge N\right)  -1}\boldsymbol{\kappa}\left( \left[  \left(
\mathbf{K}_{s}^{\otimes a}\widetilde{\otimes}\mathbf{K}_{t}^{\otimes\left(
n-1-a\right)  }\right)  \widetilde{\otimes}_{f}\left(  \mathbf{K}_{u}^{\otimes
A}\widetilde{\otimes}\mathbf{K}_{v}^{\otimes\left(  N-1-A\right)  }\right)
\right]  \widetilde{\otimes}_{n+N-2-2f}\left[  \mathbf{K}_{w}^{\otimes m^{\prime}%
}\widetilde{\otimes}\mathbf{K}_{x}^{\otimes\left(  n^{\prime}-m^{\prime}\right)
}\right]  \right)  \mathbf{1}_{\{n^{\prime}=n+N-2-2f\}}\nonumber\\
\hspace{-2cm}&  =\sum_{\{2f=n+N-n^{\prime}-2\}}\sum_{\{\epsilon\in\boldsymbol{A}_f\}}\boldsymbol{\kappa} C_{su}^{\epsilon(s,u)}%
C_{sv}^{\epsilon(s,v)}C_{tu}^{\epsilon(t,u)}C_{tv}^{\epsilon(t,v)}\nonumber\\
\hspace{-2cm}&  \hspace{1cm} \times\left\langle \mathbf{K}_{s}^{\otimes \left(a-\epsilon(s,u)
-\epsilon(s,v)\right)}\widetilde{\otimes}\mathbf{K}_{t}^{\otimes \left(n-1-a-\epsilon(t,u)
-\epsilon(t,v)\right)}\widetilde{\otimes}\mathbf{K}_{u}^{\otimes \left(A-\epsilon(s,u)
-\epsilon(t,u)\right)}\widetilde{\otimes}\mathbf{K}_{v}^{\otimes\left(  N-1-A-\epsilon(s,v)  -\epsilon(t,v)\right)  },\mathbf{K}_{w}^{\otimes m^{\prime}%
}\widetilde{\otimes}\mathbf{K}_{x}^{\otimes\left(  n^{\prime}-m^{\prime}\right)
}\right\rangle _{\mathfrak{H}^{\otimes n^{\prime}}}\nonumber\\
\hspace{-2cm}&  =\sum_{\{2f=n+N-n^{\prime}-2\}}\sum_{\{\epsilon\in\boldsymbol{A}_f,\xi\in\boldsymbol{B}^{\boldsymbol{A}}_f\}}\boldsymbol{\kappa} C_{su}^{\epsilon(s,u)}%
C_{sv}^{\epsilon(s,v)}C_{tu}^{\epsilon(t,u)}C_{tv}^{\epsilon(t,v)}C_{sw}^{\xi(s,w)}C_{sx}^{\xi(s,x)}%
C_{tw}^{\xi(t,w)}C_{tx}^{\xi(t,x)}C_{uw}^{\xi(u,w)}C_{ux}^{\xi(u,x)}C_{vw}^{\xi(v,w)}%
C_{vx}^{\xi(v,x)}\label{ex_eqn_E3corr}%
\end{align}
where $\boldsymbol{A}_f$ is the collection of exponents $\epsilon(t,u)=p$, $\epsilon(t,v)=S_1-p$, $\epsilon(s,u)=S_2-p$ and $\epsilon(s,v)=f-S_1-S_2+p$, with $p$, $S_1$, $S_2$ running over all integers such that $\max(0,S_1+S_2-f)\leq p \leq \min(S_1,S_2)$, $\max(0,n-1-a-f)\leq S_1\leq a$ and $\max(0,N-1-A-f)\leq S_1\leq A$. $\boldsymbol{B}^{\boldsymbol{A}}_f$ is defined similarly, just applying recursively point 3 of Proposition \ref{calculations}, so for instance we have $\xi(t,w)+\xi(t,x)+\xi(s,w)+\xi(s,x)+\xi(u,w)+\xi(u,x)+\xi(v,w)+\xi(v,x)=n^{\prime}$. From this
point on, we take $2f=n+N-n^{\prime}-2$. Combining (\ref{ex_eqn_E3generic})
and (\ref{ex_eqn_E3corr}), we see that the expectation inside the integral of
(\ref{ex_eqn_E3a}) is a summation
\[
\sum_{n=1}^{\infty}\sum_{m=0}^{n}\sum_{r=0}^{\infty}\sum_{N=1}^{\infty}%
\sum_{M=0}^{N}\sum_{R=0}^{\infty}\sum_{n^{\prime}=0}^{\infty}\sum_{m^{\prime
}=0}^{n^{\prime}}\sum_{r^{\prime}=0}^{\infty}\sum_{a,A}\sum_{\{2f=n+N-n^{\prime}-2\}}\sum_{\{\epsilon\in\boldsymbol{A}_f,\xi\in\boldsymbol{B}^{\boldsymbol{A}}_f\}}%
\]
with each term consisting of $15$ different types of correlations whose
exponents add up to $S=f+n^{\prime}+1+r+R+r^{\prime}$. From this it follows
that $n+N+n^{\prime}+2\left(  r+R+r^{\prime}\right)  =2S$. Systematically
listing down all possible values of the indices in this summation will show
that $e\left(  m,n,r\right)  e\left(  M,N,R\right)  e\left(  m^{\prime
},n^{\prime},r^{\prime}\right)  =0$ when $S\leq2$. For those terms in the
summation for which $S>3$, the limit they contribute in (\ref{ex_eqn_E3a}) is
$0$ (using the fourth point in Proposition \ref{calculations} with $P=6$ variables\thinspace$s$, $t$, $u$,
$v$, $w$, $x$). For those terms having $S=3$, a careful consideration of the
indices such that $e\left(  m,n,r\right)  e\left(  M,N,R\right)  e\left(
m^{\prime},n^{\prime},r^{\prime}\right)  >0$ leads to either $\left(
m,n,r,M,N,R,m^{\prime},n^{\prime},r^{\prime}\right)  =\left(
1,2,0,1,2,0,0,0,1\right)  $ or $\left(  1,2,0,1,2,0,1,2,0\right)  $.
Therefore, we can continue (\ref{ex_eqn_E3a}) as%
\begin{align*}
\mathbb{E}\left[  \widetilde{F}_{T}g_{F_{T}}\right]   &  \sim\frac{2c_{1}^{6}%
}{\Sigma^{6}T^{6H}}\int_{T^{6}}\mathbb{E}\left[  \left\langle I_{1}\left(
\mathbf{K}_{s}\otimes\mathbf{K}_{t}\right)  ,I_{1}\left(  \mathbf{K}%
_{u}\otimes\mathbf{K}_{v}\right)  \right\rangle _{\mathfrak{H}}\right]
C_{wx}dsdtdudvdwdx\\
&  \quad\quad+\frac{2c_{1}^{6}}{\Sigma^{6}T^{6H}}\int_{T^{6}}\mathbb{E}\left[
\left\langle I_{1}\left(  \mathbf{K}_{s}\otimes\mathbf{K}_{t}\right)
,I_{1}\left(  \mathbf{K}_{u}\otimes\mathbf{K}_{v}\right)  \right\rangle
_{\mathfrak{H}}I_{2}\left(  \mathbf{K}_{w}\otimes\mathbf{K}_{x}\right)
\right]  dsdtdudvdwdx\\
&  =L_{1}+L_{2}\text{.}%
\end{align*}
Using (\ref{ex_eqn_I1kernel2}), we have%
\begin{align*}
L_{1} &  =\frac{2c_{1}^{6}}{\Sigma^{6}}\cdot\frac{1}{T^{2H}}\int_{T^{2}}%
C_{wx}dwdx\cdot\frac{1}{T^{4H}}\int_{T^{4}}\mathbb{E}\left[  \left\langle
I_{1}\left(  \mathbf{K}_{s}\otimes\mathbf{K}_{t}\right)  ,I_{1}\left(
\mathbf{K}_{u}\otimes\mathbf{K}_{v}\right)  \right\rangle _{\mathfrak{H}%
}\right]  dsdtdudv\\
&  \sim\frac{2c_{1}^{6}}{\Sigma^{6}}\cdot(2)\cdot\frac{1}{T^{4H}}\int_{T^{4}%
}\frac{2C_{su}C_{tv}+2C_{sv}C_{tu}}{4}dsdtdudv\rightarrow\frac{2c_{1}^{6}}{\Sigma^{6}}\cdot(2)\cdot\left(  2\right)  ^{2}=2\biggl(\frac{2c_{1}^{2}}{\Sigma^{2}}\biggr)^3=2\text{.}%
\end{align*}
To compute $L_{2}$, we use (\ref{ex_eqn_I1kernel2}) again so%
\begin{align*}
&  \mathbb{E}\left[  \left\langle I_{1}\left(  \mathbf{K}_{s}\otimes
\mathbf{K}_{t}\right)  ,I_{1}\left(  \mathbf{K}_{u}\otimes\mathbf{K}%
_{v}\right)  \right\rangle _{\mathfrak{H}}I_{2}\left(  \mathbf{K}_{w}%
\otimes\mathbf{K}_{x}\right)  \right]  \\
& \hspace{1.5cm} =\frac{1}{4}\mathbb{E}\left[  I_{1}\left(  \mathbf{K}_{s}\right)
I_{1}\left(  \mathbf{K}_{u}\right)  I_{2}\left(  \mathbf{K}_{w}\otimes
\mathbf{K}_{x}\right)  \right]  C_{tv}+\frac{1}{4}\mathbb{E}\left[
I_{1}\left(  \mathbf{K}_{t}\right)  I_{1}\left(  \mathbf{K}_{u}\right)
I_{2}\left(  \mathbf{K}_{w}\otimes\mathbf{K}_{x}\right)  \right]  C_{sv}\\
&  \hspace{2.5cm}+\frac{1}{4}\mathbb{E}\left[  I_{1}\left(  \mathbf{K}_{s}\right)
I_{1}\left(  \mathbf{K}_{v}\right)  I_{2}\left(  \mathbf{K}_{w}\otimes
\mathbf{K}_{x}\right)  \right]  C_{tu}+\frac{1}{4}\mathbb{E}\left[
I_{1}\left(  \mathbf{K}_{t}\right)  I_{1}\left(  \mathbf{K}_{v}\right)
I_{2}\left(  \mathbf{K}_{w}\otimes\mathbf{K}_{x}\right)  \right]
C_{su}\text{.}%
\end{align*}
The first expectation simplifies to%
\begin{align*}
\mathbb{E}\left[  I_{1}\left(  \mathbf{K}_{s}\right)  I_{1}\left(
\mathbf{K}_{u}\right)  I_{2}\left(  \mathbf{K}_{w}\otimes\mathbf{K}%
_{x}\right)  \right]   &  =\sum_{r=0}^{1}r!\binom{1}{r}\binom{1}{r}%
\mathbb{E}\left[  I_{2-2r}\left(  \mathbf{K}_s\otimes
_{r}\mathbf{K}_u\right)  I_{2}\left(  \mathbf{K}_w\otimes\mathbf{K}_x\right)  \right]  \\
&  =2!\left\langle \mathbf{K}_{s}\widetilde{\otimes}\mathbf{K}_{u},\mathbf{K}%
_{w}\widetilde{\otimes}\mathbf{K}_{x}\right\rangle _{\mathfrak{H}^{\otimes2}%
}=2\frac{2C_{sw}C_{ux}+2C_{sx}C_{uw}}{4}=C_{sw}C_{ux}+C_{sx}C_{uw}%
\end{align*}
so%
\begin{align*}
L_{2} &  =\frac{2c_{1}^{6}}{\Sigma^{6}T^{6H}}\int_{T^{6}}\mathbb{E}\left[
\left\langle I_{1}\left(  \mathbf{K}_{s}\otimes\mathbf{K}_{t}\right)
,I_{1}\left(  \mathbf{K}_{u}\otimes\mathbf{K}_{v}\right)  \right\rangle
_{\mathfrak{H}}I_{2}\left(  \mathbf{K}_{w}\otimes\mathbf{K}_{x}\right)
\right]  dsdtdudvdwdx\\
&  =\frac{2c_{1}^{6}}{4\Sigma^{6}T^{6H}}\int_{T^{6}}\left[
\begin{array}
[c]{c}%
\left(  C_{sw}C_{ux}+C_{sx}C_{uw}\right)  C_{tv}+\left(  C_{tw}C_{ux}%
+C_{tx}C_{uw}\right)  C_{sv}\\
+\left(  C_{sw}C_{vx}+C_{sx}C_{vw}\right)  C_{tu}+\left(  C_{tw}C_{vx}%
+C_{tx}C_{vw}\right)  C_{su}%
\end{array}
\right]  dsdtdudvdwdx\\
&  =\frac{c_{1}^{6}}{2\Sigma^{6}T^{6H}}\cdot8\int_{T^{6}}C_{sw}C_{ux}%
C_{tv}dsdtdudvdwdx\rightarrow\frac{4c_{1}^{6}}{\Sigma^{6}}\left(  2\right)
^{3}=4\text{.}%
\end{align*}
Finally,
\[
\boxed{\mathbb{E}\left[  F_{T}g_{F_{T}}\right]\rightarrow\left(  2+4\right)
-2=4\text{.}}%
\]

\item {\bf Convergence of the Second Moment of $\boldsymbol{g_{F_T}}$:}
\begin{align}
\mathbb{E}\left[  g_{F_{T}}^{2}\right]  &=\mathbb{E}\left[  \left\langle
DF_{T},-DL^{-1}F_{T}\right\rangle _{\mathfrak{H}}^{2}\right]\\
&=\frac{1}{\Sigma^{8}T^{8H}}\int_{T^{8}}\mathbb{E}\left[  \left\langle
DF_{st},-DL^{-1}F_{uv}\right\rangle _{\mathfrak{H}}\left\langle DF_{wx}%
,-DL^{-1}F_{yz}\right\rangle _{\mathfrak{H}}\right]  dsdtdudvdwdxdydz\text{.}
\label{ex_eqn_g2}%
\end{align}
Inside the expectation, a generic term is%
\begin{align*}
&  \left\langle I_{n-1}\left(  \mathbf{K}_{s}^{\otimes m}\otimes\mathbf{K}%
_{t}^{\otimes\left(  n-m\right)  }\right)  C_{st}^{r},I_{N-1}\left(
\mathbf{K}_{u}^{\otimes M}\otimes\mathbf{K}_{v}^{\otimes\left(  N-M\right)
}\right)  C_{uv}^{R}\right\rangle _{\mathfrak{H}}\\
& \hspace{1.5cm} \times\left\langle I_{n^{\prime}-1}\left(  \mathbf{K}_{w}^{\otimes
m^{\prime}}\otimes\mathbf{K}_{x}^{\otimes\left(  n^{\prime}-m^{\prime}\right)
}\right)  C_{wx}^{r^{\prime}},I_{N^{\prime}-1}\left(  \mathbf{K}_{y}^{\otimes
M^{\prime}}\otimes\mathbf{K}_{z}^{\otimes\left(  N^{\prime}-M^{\prime}\right)
}\right)  C_{yz}^{R^{\prime}}\right\rangle_\mathfrak{H}\\
&\hspace{.5cm}  =\sum_{a,A,a',A'}\boldsymbol{\kappa}I_{n-1}\left(  \mathbf{K}_{s}^{\otimes a}\otimes\mathbf{K}%
_{t}^{\otimes\left(  n-1-a\right)  }\right)  I_{N-1}\left(  \mathbf{K}%
_{u}^{\otimes A}\otimes\mathbf{K}_{v}^{\otimes\left(  N-1-A\right)  }\right)
I_{n^{\prime}-1}\left(  \mathbf{K}_{w}^{\otimes a^{\prime}}\otimes
\mathbf{K}_{x}^{\otimes\left(  n^{\prime}-1-a^{\prime}\right)  }\right)  \\
&  \hspace{1.5cm} \quad\quad\times I_{N^{\prime}-1}\left(  \mathbf{K}_{y}^{\otimes A^{\prime}}\otimes
\mathbf{K}_{z}^{\otimes\left(  N^{\prime}-1-A^{\prime}\right)  }\right)\left\langle \mathbf{f}_{st}^{\left(  m,a\right)
},\mathbf{f}_{uv}^{\left(  M,A\right)  }\right\rangle _{\mathfrak{H}%
}\left\langle \mathbf{f}_{wx}^{\left(  m^{\prime},a^{\prime}\right)
},\mathbf{f}_{yz}^{\left(  M^{\prime},A^{\prime}\right)  }\right\rangle
_{\mathfrak{H}}C_{st}^{r}C_{uv}^{R}C_{wx}^{r^{\prime}}C_{yz}^{R^{\prime}%
}\text{.}%
\end{align*}

The expectation of the product of the four multiple integrals is of the form%
\[
\mathbb{E}\left[  I_{n-1}\left(  a\right)  I_{N-1}\left(  b\right)
I_{n^{\prime}-1}\left(  c\right)  I_{N^{\prime}-1}\left(  d\right)  \right]
=\sum_{p,q}\boldsymbol{\kappa}\left\langle a\widetilde{\otimes}_{p}b,c\widetilde{\otimes}_{q}d\right\rangle _{\mathfrak{H}%
^{\otimes\left(  n+N-2-2p\right)  }}\mathbf{1}_{\{n+N-2p=n^{\prime}+N^{\prime}%
-2q\}}\text{.}%
\]
$\mathbb{E}\left[  \left\langle DF_{st},-DL^{-1}F_{uv}\right\rangle
_{\mathfrak{H}}\left\langle DF_{wx},-DL^{-1}F_{yz}\right\rangle _{\mathfrak{H}%
}\right]  $ then is a summation
\[
\sum_{n=1}^{\infty}\sum_{m=0}^{n}\sum_{r=0}^{\infty}\sum_{N=1}^{\infty}%
\sum_{M=0}^{N}\sum_{R=0}^{\infty}\sum_{n^{\prime}=1}^{\infty}\sum_{m^{\prime
}=0}^{n^{\prime}}\sum_{r^{\prime}=0}^{\infty}\sum_{N^{\prime}=1}^{\infty}%
\sum_{M^{\prime}=0}^{N^{\prime}}\sum_{R^{\prime}=0}^{\infty}\sum_{a,A,a',A'}\sum
_{\{n+N-2p=n^{\prime}+N^{\prime}-2q\}}%
\]
consisting of $\binom{8}{2}=28$ different types of correlations whose
exponents add up to $S=p+q+\left(  n+N-2-2p\right)  +2+\left(  r+R+r^{\prime
}+R^{\prime}\right)  $. Along with the condition $n+N-2p=n^{\prime}+N^{\prime
}-2q$, we have $2S=n+N+n^{\prime}+N^{\prime}+2\left(  r+R+r^{\prime}%
+R^{\prime}\right)  $. By a careful consideration of the indices, $e\left(
m,n,r\right)  e\left(  M,N,R\right)  e\left(  m^{\prime},n^{\prime},r^{\prime
}\right)  e\left(  M^{\prime},N^{\prime},R^{\prime}\right)  =0$ if $S\leq3$.
By point 4 in Proposition \ref{calculations}, using $P=8$ for the number of variables in
(\ref{ex_eqn_g2}), we see that the integrand terms for which $S>4$ contribute
nothing to the limit as $T\rightarrow\infty$. If $S=4$ and $e\left(
m,n,r\right)  e\left(  M,N,R\right)  e\left(  m^{\prime},n^{\prime},r^{\prime
}\right)  e\left(  M^{\prime},N^{\prime},R^{\prime}\right)  =1$, we only have
$\left(  m,n,r\right)  =\left(  M,N,R\right)  =\left(  m^{\prime},n^{\prime
},r^{\prime}\right)  =\left(  M^{\prime},N^{\prime},R^{\prime}\right)
=\left(  1,2,0\right)  $. Therefore,%
\[
\mathbb{E}\left[  g_{F_{T}}^{2}\right]  \sim\frac{4c_{1}^{8}}{\Sigma^{8}%
T^{8H}}\int_{T^{8}}\mathbb{E}\left[  \left\langle I_{1}\left(  \mathbf{K}%
_{s}\otimes\mathbf{K}_{t}\right)  ,I_{1}\left(  \mathbf{K}_{u}\otimes
\mathbf{K}_{v}\right)  \right\rangle _{\mathfrak{H}}\left\langle I_{1}\left(
\mathbf{K}_{w}\otimes\mathbf{K}_{x}\right)  ,I_{1}\left(  \mathbf{K}%
_{y}\otimes\mathbf{K}_{z}\right)  \right\rangle _{\mathfrak{H}}\right]
dsdtdudvdwdxdydz.
\]
We use (\ref{ex_eqn_I1kernel2}) on the integrand:%
\begin{align*}
& \mathbb{E}\left[  \left\langle I_{1}\left(  \mathbf{K}_{s}\otimes
\mathbf{K}_{t}\right)  ,I_{1}\left(  \mathbf{K}_{u}\otimes\mathbf{K}%
_{v}\right)  \right\rangle _{\mathfrak{H}}\left\langle I_{1}\left(
\mathbf{K}_{w}\otimes\mathbf{K}_{x}\right)  ,I_{1}\left(  \mathbf{K}%
_{y}\otimes\mathbf{K}_{z}\right)  \right\rangle _{\mathfrak{H}}\right]  \\
& =\frac{1}{16}\mathbb{E}\left[
\begin{array}
[c]{c}%
\left\{  I_{1}\left(  \mathbf{K}_{s}\right)  I_{1}\left(  \mathbf{K}%
_{u}\right)  C_{tv}+I_{1}\left(  \mathbf{K}_{t}\right)  I_{1}\left(
\mathbf{K}_{u}\right)  C_{sv}+I_{1}\left(  \mathbf{K}_{s}\right)  I_{1}\left(
\mathbf{K}_{v}\right)  C_{tu}+I_{1}\left(  \mathbf{K}_{t}\right)  I_{1}\left(
\mathbf{K}_{v}\right)  C_{su}\right\}  \\
\times\left\{  I_{1}\left(  \mathbf{K}_{w}\right)  I_{1}\left(  \mathbf{K}%
_{y}\right)  C_{xz}+I_{1}\left(  \mathbf{K}_{x}\right)  I_{1}\left(
\mathbf{K}_{y}\right)  C_{wz}+I_{1}\left(  \mathbf{K}_{w}\right)  I_{1}\left(
\mathbf{K}_{z}\right)  C_{xy}+I_{1}\left(  \mathbf{K}_{x}\right)  I_{1}\left(
\mathbf{K}_{z}\right)  C_{wy}\right\}
\end{array}
\right]  \text{.}%
\end{align*}
Therefore,%
\begin{align*}
\mathbb{E}\left[  g_{F_{T}}^{2}\right]    & \sim\frac{4c_{1}^{8}}{\Sigma
^{8}T^{8H}}\cdot\frac{1}{16}\cdot16\int_{T^{8}}\mathbb{E}\left[  I_{1}\left(
\mathbf{K}_{s}\right)  I_{1}\left(  \mathbf{K}_{u}\right)  I_{1}\left(
\mathbf{K}_{w}\right)  I_{1}\left(  \mathbf{K}_{y}\right)  \right]
C_{tv}C_{xz}dsdtdudvdwdxdydz\\
& \sim\frac{4c_{1}^{8}}{\Sigma^{8}T^{4H}}\int_{T^{4}%
}\mathbb{E}\left[  I_{1}\left(  \mathbf{K}_{s}\right)  I_{1}\left(
\mathbf{K}_{u}\right)  I_{1}\left(  \mathbf{K}_{w}\right)  I_{1}\left(
\mathbf{K}_{y}\right)  \right]  dsdudwdy\cdot\frac{1}{T^{4H}}\int_{T^4}C_{tv}C_{xz}dtdvdxdz
\end{align*}%
We have
\begin{align*}
\mathbb{E}\left[  I_{1}\left(  \mathbf{K}_{s}\right)  I_{1}\left(
\mathbf{K}_{u}\right)  I_{1}\left(  \mathbf{K}_{w}\right)  I_{1}\left(
\mathbf{K}_{y}\right)  \right]    & =\sum_{p=0}^{1}\sum_{q=0}^{1}%
\mathbb{E}\left[  I_{2-2p}\left(  \mathbf{K}_{s}\otimes_{p}\mathbf{K}%
_{u}\right)  I_{2-2q}\left(  \mathbf{K}_{w}\otimes_{p}\mathbf{K}_{y}\right)
\right]  \\
& =\mathbb{E}\left[  I_{2}\left(  \mathbf{K}_{s}\otimes\mathbf{K}_{u}\right)
I_{2}\left(  \mathbf{K}_{w}\otimes\mathbf{K}_{y}\right)  \right]
+\mathbb{E}\left[  I_{0}\left(  \mathbf{K}_{s}\otimes_{1}\mathbf{K}%
_{u}\right)  I_{0}\left(  \mathbf{K}_{w}\otimes_{1}\mathbf{K}_{y}\right)
\right]  \\
& =2\left\langle \mathbf{K}_{s}\widetilde{\otimes}\mathbf{K}_{u},\mathbf{K}%
_{w}\widetilde{\otimes}\mathbf{K}_{y}\right\rangle _{\mathfrak{H}^{\otimes2}%
}+C_{su}C_{wy}\\
& =2\left\langle \frac{\left(  \mathbf{K}_{s}\otimes\mathbf{K}_{u}\right)
+\left(  \mathbf{K}_{u}\otimes\mathbf{K}_{s}\right)  }{2},\frac{\left(
\mathbf{K}_{w}\otimes\mathbf{K}_{y}\right)  +\left(  \mathbf{K}_{y}%
\otimes\mathbf{K}_{w}\right)  }{2}\right\rangle _{\mathfrak{H}^{\otimes2}%
}+C_{su}C_{wy}\\
& =\frac{C_{sw}C_{uy}}{2}+\frac{C_{uw}C_{sy}}{2}+\frac{C_{sy}C_{uw}}{2}%
+\frac{C_{uy}C_{sw}}{2}+C_{su}C_{wy}\\
& =C_{sw}C_{uy}+C_{uw}C_{sy}+C_{su}C_{wy}%
\end{align*}
and so%
\[
\mathbb{E}\left[  g_{F_{T}}^{2}\right]  \sim\frac{4c_{1}^{8}}{\Sigma^{8}}\cdot\frac{1}{T^{4H}}\cdot3\int_{T^{4}}C_{sw}C_{uy}dsdudwdy\cdot\frac{1}{T^{4H}}\int_{T^4}C_{tv}C_{xz}dtdvdxdz\rightarrow
\frac{4c_{1}^{8}}{\Sigma^{8}}\cdot3\left(  2\right)
^{2}\left(  2\right)  ^{2}=12
\]
proving that
$$\boxed{\mathbb{E}\left[g_{F_T}^2\right]\to12.}$$

\end{itemize}

Let $\beta=2$ and $\gamma=2$. We've shown that%
\begin{align*}
\mathbb{E}\left[  F_{T}^{2}\right]    & \rightarrow2=\gamma\\
\mathbb{E}\left[  F_{T}g_{F_{T}}\right] & \rightarrow4=\beta\gamma\\
\mathbb{E}\left[  g_{F_{T}}^{2}\right]    & \rightarrow12=\beta^{2}%
\gamma+\gamma^{2}\text{.}%
\end{align*}
Therefore, $F_{T}$ converges to a (centered) Gamma random variable.
Specifically, since $\beta=2$, it converges to a (centered)
Chi-squared random variable with one degree of freedom (see Table \ref{table_geestar}).
\end{proof}

\bigskip

Notice that this example is not trivial since is not a process in a fixed Wiener chaos. Also, it shows the importance of equation (\ref{eqn_inequality2}) in Theorem \ref{thm_boundnew}, since equation (\ref{eqn_inequality1}) is very intractable in this case.

\section{\label{sec_WPresults}NP bound in Wiener-Poisson space}

In Wiener-Poisson space, if we repeat the process before equation
(\ref{eqn_basicbound}) and use (\ref{eqn_integpartsWP_L}), the correct
integration by parts formula, we get
\begin{align}
d_{\mathcal{H}}\left(  X,Z\right)  \leq\sup_{f\in\mathscr{F}_{\mathcal{H}}%
}\left|  \mathbb{E}\left[  f^{\prime}\left(  X\right)  \left(  g_{\ast}\left(
X\right)  -g_{X}\right)  \right]  + \mathbb{E}\left[  \left\langle \int_0^{D_zF}f''(F+xu)x(D_zF-u)du,-DL^{-1}%
X\right\rangle _\mathfrak{H}\right]\right|\label{eqn_basicboundWP}%
\end{align}
It becomes evident that we need to find universal bounds on the first and
second derivatives of $f$. Recall from subsection \ref{fprimeprime} that we only have such bounds when $l=-\infty$ and $u=\infty$. With this in mind, we have $\mathscr{F}_{\mathcal{H}}=\{f\in\mathcal{C}^1 : f' \text{ is Lipschitz, } ||f'||_{\infty}<k_1, ||f''||_{\infty}<k_2\}$, where $k_1$ and $k_2$ depend only on the distance $d_{\mathcal{H}}$. The following is a generalization of Theorem 2 in
\cite{Viquez} (where $Z$ was standard Normal) and an extension of Theorem
\ref{thm_boundnew} to Wiener-Poisson space.

\begin{theorem}
\label{thm_boundnewWP}(NP bound) Let $d_{\mathcal{H}}$ be $d_{W}$ or $d_{FM}$. Under Assumptions A and B$^{\prime}$,
\begin{align*}
d_{\mathcal{H}}\left(  X,Z\right)   &  \leq k\biggl(\mathbb{E}\left|  g_{\ast}\left(  X\right)
-g_X\right|  +\mathbb{E}\left[
\left\langle \left|  x\left(  DX\right)  ^{2}\right|  ,\left|  -DL^{-1}%
X\right|  \right\rangle _\mathfrak{H}\right]\biggr)
\end{align*}
where $k$ is a finite constant depending only on $Z$ and on $d_{\mathcal{H}}$.
\end{theorem}
\begin{proof}
This follows immediately from (\ref{eqn_basicboundWP}) since $|\langle a,b\rangle|\leq\langle|a|,|b|\rangle$ and $\left|\int f\right|\leq\int|f|$.
\end{proof}

\bigskip

This upper bound was first developed for Poisson space in \cite{PUST}, where was used to prove several CLTs for Poisson functionals. In \cite{Viquez} it was used to prove CLTs for Wiener-Poisson functionals.

\begin{corollary}
\label{cor_LsequenceWP}$X_{n}\rightarrow Z$ in distribution if both statements
are true.

\begin{enumerate}
\item $g_{\ast}\left(  X_{n}\right)  -g_{X_n}\rightarrow0$ in $L^{1}(\Omega)$.

\item $\left\langle \left|  x\left(  DX_{n}\right)  ^{2}\right|  ,\left|
-DL^{-1}X_{n}\right|  \right\rangle _\mathfrak{H}\rightarrow0$ in $L^{1}(\Omega)$.
\end{enumerate}
\end{corollary}

\bigskip

\noindent For convergence results inside a fixed Wiener chaos, the following preliminary
computations are needed.

\begin{proposition}
\label{lem_WPcomp}Let $X_n=I_{q}\left(  f_n\right)  $, with $\mathbb{E}\left[  X_{n}%
^{2}\right]  =q!\left\|  f_{n}\right\|  _{\mathfrak{H}^{\otimes q}}^{2}\rightarrow1$. Assume that for $r=0,\ldots,q-1$ and $s=0,\ldots,q-r$,
$\left\|  f_{n}\otimes_{r}^{s}f_{n}\right\|  _{\mathfrak{H}^{\otimes\left(
2q-2r-s\right)}} \mathbf{1}_{\{s=0,r\neq0\}\cup\{s\neq0,r=0\}}\rightarrow 0$. Then as $n\to\infty$,

\begin{enumerate}

\item $\mathbb{E}\left[  \left\|  DX_n\right\|  _\mathfrak{H}^{4}\right]\rightarrow q^2$;

\item $||DX_n||^2_\mathfrak{H}\rightarrow q$ in $L^2(\Omega)$;

\item $\mathbb{E}\left[  \int_{\mathbb{R}^{+}\times\mathbb{R}}x^{2}\left(  D_{z}%
X_{n}\right)  ^{4}d\mu\left(  z\right)  \right]  \rightarrow0$;

\item $\mathbb{E}\left[  X_n^{4}\right]\rightarrow 3$.

\end{enumerate}
\end{proposition}
\begin{proof}
Since $D_{z}X_n=qI_{q-1}\left(  f_n\left(  z,\cdot\right)  \right)  $, we can apply the product formula (\ref{eqn_productWP}) to get
\begin{align*}
\left\|  DX_n\right\|  _\mathfrak{H}^{2}  &  =\left\langle DX_n,DX_n\right\rangle _\mathfrak{H}%
=q^{2}\int I_{q-1}\left(  f_n\left(  z,\cdot\right)  \right)  I_{q-1}\left(
f_n\left(  z,\cdot\right)  \right)  d\mu\left(  z\right) \\
&  =q^{2}\int\sum_{r=0}^{q-1}\sum_{s=0}^{q-1-r}r!s!\binom{q-1}{r}^{2}%
\binom{q-1-r}{s}^{2}I_{2q-2-2r-s}\bigl(  f_n\left(  z,\cdot\right)  \otimes
_{r}^{s}f_n\left(  z,\cdot\right)  \bigr)  d\mu\left(  z\right) \\
&  =q^{2}\sum_{p=1}^{q}\sum_{s=0}^{q-p}\left(  p-1\right)  !s!\binom{q-1}%
{p-1}^{2}\binom{q-p}{s}^{2}\int I_{2q-2p-s}\left(  f_n\left(  z,\cdot\right)
\otimes_{p-1}^{s}f_n\left(  z,\cdot\right)  \right)  d\mu\left(  z\right) \\
&  =\sum_{p=1}^{q}\sum_{s=0}^{q-p}pp!s!\binom{q}{p}^{2}\binom{q-p}{s}%
^{2}I_{2q-2p-s}\left(  \int f_n\left(  z,\cdot\right)  \otimes_{p-1}^{s}f_n\left(
z,\cdot\right)  d\mu\left(  z\right)  \right) \\
&  =\sum_{r=1}^{q}\sum_{s=0}^{q-r}rr!s!\binom{q}{r}^{2}\binom{q-r}{s}%
^{2}I_{2q-2r-s}\left(  f_n\otimes_{r}^{s}f_n\right)  \text{.}%
\end{align*}

Also by orthogonality of chaoses,%
\begin{align*}
\mathbb{E}\left[  \left\|  DX_n\right\|  _\mathfrak{H}^{4}\right]  &=\sum_{r,R=1}^{q}%
\sum_{s=0}^{q-r}\sum_{S=0}^{q-R}rRr!R!s!S!\binom{q}{r}^{2}\binom{q}{R}%
^{2}\binom{q-r}{s}^{2}\binom{q-R}{S}^{2}\mathbb{E}\left[  I_{2q-2r-s}\left(
f_n\otimes_{r}^{s}f_n\right)  I_{2q-2R-S}\left(  f_n\otimes_{R}^{S}f_n\right)
\right]\\
&\overset{\left(  \text{*}\right)  }{\leq}\sum_{%
\genfrac{.}{.}{0pt}{}{r,R=1}{r\neq R}%
}^{q}\sum_{s=0}^{q-r}\sum_{S=0}^{q-R}rRr!R!s!S!\binom{q}{r}^{2}\binom{q}%
{R}^{2}\binom{q-r}{s}^{2}\binom{q-R}{S}^{2}\\
&  \hspace{3cm}\times\mathbf{1}_{\{2r+s=2R+S\}}\left(  2q-2r-s\right)  !\left|\left| f_n\otimes_{r}^{s}f_n\right|\right|_{\mathfrak{H}^{\otimes\left(  2q-2r-s\right)}}\left|\left|f_n\otimes%
_{R}^{S}f_n\right|\right|_{\mathfrak{H}^{\otimes\left(  2q-2r-s\right)  }}\\
&  \hspace{1cm}+\sum_{r=1}^{q-1}\sum_{s=0}^{q-r}r^{2}\left(  r!\right)
^{2}\left(  s!\right)  ^{2}\binom{q}{r}^{4}\binom{q-r}{s}^{4}\left(
2q-2r-s\right)  !\left\|  f_n\otimes_{r}^{s}f_n\right\|  _{\mathfrak{H}^{\otimes
\left(  2q-2r-s\right)  }}^{2}+q^{2}\left(  q!\left\|  f_n\right\|  _{\mathfrak{H}^{\otimes
q}}^{2}\right)  ^{2}
\end{align*}
In $\left(  \text{*}\right)  $, we used $||\widetilde
{g}||_\mathfrak{H}\leq||g||_\mathfrak{H}$ for nonsymmetric $g$ (this follows by a simple application of the triangle inequality), and H$\ddot{\text{o}}$lder's inequality
in the following:
\begin{align*}
\mathbb{E}\left[  I_{2q-2r-s}\left(  f_n\otimes_{r}^{s}f_n\right)  I_{2q-2R-S}%
\left(  f_n\otimes_{R}^{S}f_n\right)  \right]  &=\mathbf{1}_{\{2r+s=2R+S\}}\left(
2q-2r-s\right)  !\left\langle f_n\widetilde{\otimes}_{r}^{s}f_n,f_n\widetilde{\otimes}%
_{R}^{S}f_n\right\rangle _{\mathfrak{H}^{\otimes\left(  2q-2r-s\right)  }}\\
&\leq\mathbf{1}_{\{2r+s=2R+S\}}\left(
2q-2r-s\right)  !\left|\left| f_n\widetilde{\otimes}_{r}^{s}f_n\right|\right|_{\mathfrak{H}^{\otimes\left(  2q-2r-s\right)}}\left|\left|f_n\widetilde{\otimes}%
_{R}^{S}f_n\right|\right|_{\mathfrak{H}^{\otimes\left(  2q-2r-s\right)  }}\text{.}
\end{align*}
$\mathbb{E}\left[  \left\|  DX_{n}\right\|  _\mathfrak{H}^{4}\right]  \rightarrow
q^{2}$ then follows from the assumptions on the kernels' contractions, proving the first point.

\bigskip

On the other hand,
\begin{align*}
\mathbb{E}\left[  \left(  \left\|  DX_{n}\right\|  _\mathfrak{H}^{2}-q\right)
^{2}\right]   &  =\mathbb{E}\left[  \left\|  DX_{n}\right\|  _\mathfrak{H}%
^{4}-2q\left\|  DX_{n}\right\|  _\mathfrak{H}^{2}+q^{2}\right] 
  =\mathbb{E}\left[  \left\|  DX_{n}\right\|  _\mathfrak{H}^{4}\right]  -2q\cdot
q\mathbb{E}\left[  X_{n}^{2}\right]  +q^{2}\rightarrow0
\end{align*}
so $\left\|  DX_{n}\right\|  _\mathfrak{H}^{2}\rightarrow q$ in $L^{2}\left(
\Omega\right)  $ proving the second point.

\bigskip

For the third point we have,
\begin{align*}
\left(  D_{z}X_n\right)  ^{2}  & =q^{2}\sum_{r=0}^{q-1}\sum_{s=0}%
^{q-1-r}r!s!\binom{q-1}{r}^{2}\binom{q-1-r}{s}^{2}I_{2q-2-2r-s}\bigl(
f_n\left(  z,\cdot\right)  \otimes_{r}^{s}f_n\left(  z,\cdot\right)  \bigr) \\
\left(  D_{z}X_n\right)  ^{4}  &  =q^{4}\sum_{r=0}^{q-1}\sum_{R=0}^{q-1}%
\sum_{s=0}^{q-1-r}\sum_{S=0}^{q-1-R}r!R!s!S!\binom{q-1}{r}^{2}\binom{q-1}%
{R}^{2}\binom{q-1-r}{s}^{2}\binom{q-1-R}{S}^{2}\\
&  \hspace{2cm}\times I_{2q-2-2r-s}\bigl(  f_n\left(  z,\cdot\right)  \otimes
_{r}^{s}f_n\left(  z,\cdot\right)  \bigr)  I_{2q-2-2R-S}\left(  f_n\left(
z,\cdot\right)  \otimes_{R}^{S}f_n\left(  z,\cdot\right)  \right) \\
\mathbb{E}\left[  \int x^{2}\left(  D_{z}X_n\right)  ^{4}d\mu\left(  z\right)
\right]   &  =q^{4}\sum_{r=0}^{q-1}\sum_{R=0}^{q-1}\sum_{s=0}^{q-1-r}%
\sum_{S=0}^{q-1-R}r!R!s!S!\binom{q-1}{r}^{2}\binom{q-1}{R}^{2}\binom{q-1-r}%
{s}^{2}\binom{q-1-R}{S}^{2}\\
&  \quad\quad\times\int\mathbb{E}\left[  I_{2q-2-2r-s}\bigl(  xf_n\left(
z,\cdot\bigr)  \otimes_{r}^{s}f_n\left(  z,\cdot\right)  \right)
I_{2q-2-2R-S}\left(  xf_n\left(  z,\cdot\right)  \otimes_{R}^{S}f_n\left(
z,\cdot\right)  \right)  \right]  d\mu\left(  z\right)  \text{.}%
\end{align*}
The expectation, when $2r+s=2R+S$, is bounded by
\begin{align*}
\left(  2q-2r-s-2\right)!  &\left|  \left\langle xf_n\left(  z,\cdot\right)\widetilde{\otimes}_{r}^{s}f_n\left(  z,\cdot\right)  ,xf_n\left(  z,\cdot\right)\widetilde{\otimes}_{R}^{S}f_n\left(  z,\cdot\right)  \right\rangle _{\mathfrak{H}^{\otimes\left(  2q-2R-S-2\right)  }}\right| \\
&  \leq\left(  2q-2r-s-2\right)  !\left\|  xf_n\left(  z,\cdot\right)\widetilde{\otimes}_{r}^{s}f_n\left(  z,\cdot\right)  \right\|  _{\mathfrak{H}^{\otimes\left(2q-2r-s-2\right)  }}\left\|  xf_n\left(  z,\cdot\right)  \widetilde{\otimes}_{R}^{S}f_n\left(  z,\cdot\right)  \right\|  _{\mathfrak{H}^{\otimes\left(  2q-2R-S-2\right)
}}\text{.}%
\end{align*}
Modulo the constant factor $\left(  2q-2r-s-2\right)  !$, the integral of the
expectation is bounded by%
\begin{align*}
\int\left\|  xf_n\left(  z,\cdot\right)  \widetilde{\otimes}_{r}^{s}f_n\left(z,\cdot\right)  \right\|  _{\mathfrak{H}^{\otimes\left(  2q-2r-s-2\right)  }}&\left\|xf_n\left(  z,\cdot\right)  \widetilde{\otimes}_{R}^{S}f_n\left(  z,\cdot\right)\right\|  _{\mathfrak{H}^{\otimes\left(  2q-2R-S-2\right)  }}d\mu\left(  z\right) \\
&  =\left\langle \left\|  xf_n\left(  z,\cdot\right)  \widetilde{\otimes}_{r}^{s}f_n\left(  z,\cdot\right)  \right\|  _{\mathfrak{H}^{\otimes\left(  2q-2r-s-2\right)}},\left\|  xf_n\left(  z,\cdot\right)  \widetilde{\otimes}_{R}^{S}f_n\left(z,\cdot\right)  \right\|  _{\mathfrak{H}^{\otimes\left(  2q-2R-S-2\right)  }}\right\rangle _\mathfrak{H}\\
&  \leq\bigl\|  \left\|  xf_n\left(  z,\cdot\right)  \otimes_{r}^{s}f_n\right\|_{H^{\otimes\left(  2q-2r-s-2\right)  }}\bigr\|  _\mathfrak{H}  \times\left\| || xf_n\left(  z,\cdot\right)  \otimes_{R}^{S}f_n\left(  z,\cdot\right)  || _{\mathfrak{H}^{\otimes\left(  2q-2R-S-2\right)  }}\right\|  _\mathfrak{H}
\end{align*}

We'll work out the first factor:%
\begin{align*}
\bigl\|  \left\|  xf_n\left(  z,\cdot\right)  \otimes_{r}^{s}f_n\left(
z,\cdot\right)  \right\|  _{\mathfrak{H}^{\otimes\left(  2q-2r-s-2\right)  }}\bigr\|
_\mathfrak{H}^{2}  &  =\int\left\|  xf_n\left(  z,\cdot\right)  \otimes_{r}^{s}f_n\left(
z,\cdot\right)  \right\|  _{\mathfrak{H}^{\otimes\left(  2q-2r-s-2\right)  }}^{2}%
d\mu\left(  z\right) \\
&  =\int\left\|  \left(  f_n\otimes_{r}^{s+1}f_n\right)  \left(  z,\cdot\right)
\right\|  _{\mathfrak{H}^{\otimes\left(  2q-2r-s-2\right)  }}^{2}d\mu\left(  z\right)
=\left\|  f_n\otimes_{r}^{s+1}f_n\right\|  _{\mathfrak{H}^{\otimes\left(  2q-2r-s-1\right)
}}^{2}\text{.}%
\end{align*}
Finally,
\begin{align*}
\mathbb{E}\left[  \int x^{2}\left(  D_{z}X_n\right)  ^{4}d\mu\left(  z\right)
\right]   &  \leq q^{4}\sum_{r,R=0}^{q-1}\sum_{s=0}^{q-1-r}\sum_{S=0}%
^{q-1-R}r!R!s!S!\binom{q-1}{r}^{2}\binom{q-1}{R}^{2}\binom{q-1-r}{s}^{2}%
\binom{q-1-R}{S}^{2}\\
&  \quad\quad\times\mathbf{1}_{\{2r+s=2R+S\}}\left(  2q-2r-s-2\right)  !\left\|
f_n\otimes_{r}^{s+1}f_n\right\|  _{\mathfrak{H}^{\otimes\left(  2q-2r-s-1\right)  }}%
^{2}\left\|  f_n\otimes_{R}^{S+1}f_n\right\|  _{\mathfrak{H}^{\otimes\left(
2q-2R-S-1\right)  }}^{2}\\
&  =q^{4}\sum_{r,R=0}^{q-1}\sum_{t=1}^{q-r}\sum_{T=1}^{q-R}r!R!\left(
t-1\right)  !\left(  T-1\right)  !\binom{q-1}{r}^{2}\binom{q-1}{R}^{2}%
\binom{q-1-r}{t-1}^{2}\binom{q-1-R}{T-1}^{2}\\
&  \quad\quad\times\mathbf{1}_{\{2r+t=2R+T\}}\left(  2q-2r-t-1\right)  !\left\|
f_n\otimes_{r}^{t}f_n\right\|  _{\mathfrak{H}^{\otimes\left(  2q-2r-t\right)  }}^{2}%
\times\left\|  f_n\otimes_{R}^{T}f_n\right\|  _{\mathfrak{H}^{\otimes\left(  2q-2R-T\right)
}}^{2}%
\end{align*}
The third point then follows.

\bigskip

Finally, for the fourth point, we use the moments formula Proposition \ref{lem_expec} to get

\[
\mathbb{E}\left[  X_n^{4}\right]  =\frac{3}{q}\mathbb{E}\left[  X_n^{2}\left\|
DX_n\right\|  _{H}^{2}\right]  +\frac{3}{q}\mathbb{E}\left[  \left\langle
x\left(  DX_n\right)  ^{3},X_n+\theta_{z}xDX_n\right\rangle _\mathfrak{H}\right]  =\frac
{3}{q}U_n+\frac{3}{q}\left(  V_n+W_n\right)
\]
where
\begin{align*}
V_n  &  =\mathbb{E}\left[  \left\langle x\left(  DX_n\right)  ^{3},X_n\right\rangle
_\mathfrak{H}\right]  =\mathbb{E}\left[  \left\langle x\left(  DX_n\right)  ^{2},X_n\left(
DX_n\right)  \right\rangle _\mathfrak{H}\right] \\
W_n  &  =\mathbb{E}\left[  \left\langle x\left(  DX_n\right)  ^{3},\theta
_{z}xDX_n\right\rangle _\mathfrak{H}\right]  =\mathbb{E}\left[  \int\theta_{z}%
x^{2}\left(  D_{z}X_n\right)  ^{4}d\mu\left(  z\right)  \right]
\end{align*}
and $U_n=\mathbb{E}\left[  X_n^{2}\left\|  DX_n\right\|  _\mathfrak{H}^{2}\right]  $. It is sufficient to prove that $U_n\to q$, $V_n\to0$ and $W_n\to0$ as $n\to\infty$.

\bigskip

To compute $U_{n}$, note that $X_n^{2}=\sum_{r=0}^{q}\sum_{s=0}%
^{q-r}r!s!\binom{q}{r}^{2}\binom{q-r}{s}^{2}I_{2q-2r-s}\left(  f_n\otimes
_{r}^{s}f_n\right)  $. Using our expression for $\left\|  DX_n\right\|  _\mathfrak{H}^{2}$
above,
\begin{align*}
\hspace{-1cm} U_n  &  =\sum_{r=0}^{q}\sum_{R=1}^{q}\sum_{s=0}^{q-r}\sum_{S=0}^{q-R}%
Rr!R!s!S!\binom{q}{r}^{2}\binom{q}{R}^{2}\binom{q-r}{s}^{2}\binom{q-R}{S}%
^{2}\mathbb{E}\left[  I_{2q-2r-s}\left(  f_n\otimes_{r}^{s}f_n\right)
I_{2q-2R-S}\left(  f_n\otimes_{R}^{S}f_n\right)  \right] \\
&  =\sum_{r=0}^{q}\sum_{R=1}^{q}\sum_{s=0}^{q-r}\sum_{S=0}^{q-R}%
Rr!R!s!S!\binom{q}{r}^{2}\binom{q}{R}^{2}\binom{q-r}{s}^{2}\binom{q-R}{S}%
^{2}\mathbf{1}_{\{2r+s=2R+S\}}\left(  2q-2r-s\right)  !\left\langle f_n\widetilde
{\otimes}_{r}^{s}f_n,f_n\widetilde{\otimes}_{R}^{S}f_n\right\rangle _{\mathfrak{H}^{\otimes\left(
2q-2r-s\right)  }}\\
&=\sum_{%
\genfrac{.}{.}{0pt}{}{r=0,R=1}{r\neq R}%
}^{q}\sum_{s=0}^{q-r}\sum_{S=0}^{q-R}Rr!R!s!S!\binom{q}{r}^{2}\binom{q}%
{R}^{2}\binom{q-r}{s}^{2}\binom{q-R}{S}^{2}\mathbf{1}_{\{2r+s=2R+S\}}\left(  2q-2r-s\right)  !\left\langle f_n\widetilde
{\otimes}_{r}^{s}f_n,f_n\widetilde{\otimes}_{R}^{S}f_n\right\rangle _{\mathfrak{H}^{\otimes\left(
2q-2r-s\right)  }}\\
&  \quad\quad+\sum_{r=1}^{q-1}\sum_{s=0}^{q-r}r\left(  r!\right)
^{2}\left(  s!\right)  ^{2}\binom{q}{r}^{4}\binom{q-r}{s}^{4}\left(
2q-2r-s\right)  !\left\|  f_n\widetilde{\otimes}_{r}^{s}f_n\right\|  _{\mathfrak{H}^{\otimes
\left(  2q-2r-s\right)  }}^{2}+q\left(  q!\left\|  f_n\right\|  _{\mathfrak{H}^{\otimes
q}}^{2}\right)  ^{2}
\end{align*}
We can again apply H$\ddot{\text{o}}$lder's inequality on the inner product,
and conclude that all the terms go to $0$ except the last term which goes to
$q$. Therefore, $U_{n}\rightarrow q$ as $n\rightarrow\infty$.

\bigskip

Observe that
\begin{align*}
\left|  V_{n}\right|   &  \leq\mathbb{E}\left[  \left\|  x\left(
DX_{n}\right)  ^{2}\right\|  _\mathfrak{H}\left\|  X_{n}\left(  DX_{n}\right)
\right\|  _\mathfrak{H}\right]  \leq\sqrt{\mathbb{E}\left[  \left\|  x\left(
DX_{n}\right)  ^{2}\right\|  _\mathfrak{H}^{2}\right]  }\sqrt{\mathbb{E}\left[
\left\|  X_{n}\left(  DX_{n}\right)  \right\|  _\mathfrak{H}^{2}\right]  }\\
&  =\sqrt{\mathbb{E}\left[  \int x^{2}\left(  D_{z}X_{n}\right)  ^{4}%
d\mu\left(  z\right)  \right]  }\sqrt{\mathbb{E}\left[  \left\|  X_{n}\left(
DX_{n}\right)  \right\|  _\mathfrak{H}^{2}\right]  }%
\end{align*}

Note that
\[
\mathbb{E}\left[  \left\|  X_{n}\left(  DX_{n}\right)  \right\|  _\mathfrak{H}%
^{2}\right]  =\mathbb{E}\left[  X_{n}^{2}\left\|  DX_{n}\right\|  _\mathfrak{H}%
^{2}\right]  =U_{n}\rightarrow q\text{.}
\]
From the third point, $\mathbb{E}\left[  \int_{\mathbb{R}^{+}\times\mathbb{R}%
}x^{2}\left(  D_{z}X_{n}\right)  ^{4}d\mu\left(  z\right)  \right]
\rightarrow0$ so $V_{n}\rightarrow0$ as $n\rightarrow\infty$.

\bigskip

Finally for $W_n$,
\[
|W_{n}|=\mathbb{E}\left[  \int\theta_{z}x^{2}\left(  D_{z}X_n\right)  ^{4}%
d\mu\left(  z\right)  \right]  \leq\mathbb{E}\left[  \int x^{2}\left(
D_{z}X_n\right)  ^{4}d\mu\left(  z\right)  \right] \longrightarrow0 \text{.}%
\]
Putting them together we get the fourth point: $\mathbb{E}[X_n^4]\to3$ as $n\to\infty$.
\end{proof}

\bigskip

In Wiener space, convergence in a fixed Wiener chaos to a standard normal
distribution is characterized by the convergence of the fourth moments to $3$
or of the convergence of the norm of certain contractions to $0$ (see list
preceding Corollary \ref{cor_Xchaos}). We would then like to see if the same
situation holds in Wiener-Poisson space. At this point, this appears to be an
open question. We then finish with the following theorem which shows
convergence in distribution and of the fourth moments to $3$ if certain
contractions converge to $0$.

\begin{theorem}
Let $X_n=I_{q}\left(  f_n\right)  $, with $\mathbb{E}\left[  X_{n}%
^{2}\right]  =q!\left\|  f_{n}\right\|  _{\mathfrak{H}^{\otimes q}}^{2}\rightarrow1$. Assume that for $r=0,\ldots,q-1$ and $s=0,\ldots,q-r$,
$\left\|  f_{n}\otimes_{r}^{s}f_{n}\right\|  _{\mathfrak{H}^{\otimes\left(
2q-2r-s\right)}} \mathbf{1}_{\{s=0,r\neq0\}\cup\{s\neq0,r=0\}}\rightarrow0$. Then as $n\to\infty$,
\begin{itemize}

\item $\mathbb{E}[X_n^4]\to3$.

\item $X_{n}\rightarrow \mathcal{N}\left(  0,1\right)  $ in distribution.

\end{itemize}
\end{theorem}
\begin{proof}
The first assertion is the fourth point in Proposition \ref{lem_WPcomp}. For the second
point, we refer to Corollary \ref{cor_LsequenceWP} to see that it suffices to
prove $g_{\ast}\left(  X_{n}\right)  -g_{X_{n}}\rightarrow0$ in $L^{2}\left(
\Omega\right)  $ and $\left\langle \left|  x\right|  \left(  DX_{n}\right)
^{2},\left|  DX_{n}\right|  \right\rangle _\mathfrak{H}\rightarrow0$ in $L^{1}(\Omega)$
as $n\rightarrow\infty$. These are immediate when we note that
$$g_*(X_n)-g_{X_n}=1-\frac{1}{q}||DX_n||^2_\mathfrak{H}\rightarrow 0 \ \text{ in } L^2(\Omega)\hspace{1cm} \text{(by point 2 of Proposition \ref{lem_WPcomp})}$$
and
\begin{align*}
\mathbb{E}\left[  \left\langle \left|  x\right|  \left(  DX_{n}\right)
^{2},\left|  DX_{n}\right|  \right\rangle _\mathfrak{H}\right]   &  \leq\mathbb{E}%
\left[  \left\|  x\left(  DX_{n}\right)  ^{2}\right\|  _\mathfrak{H}\left\|
DX_{n}\right\|  _\mathfrak{H}\right]  \leq\sqrt{\mathbb{E}\left[  \left\|  x\left(
DX_{n}\right)  ^{2}\right\|  _\mathfrak{H}^{2}\right]  }\sqrt{\mathbb{E}\left[
\left\|  DX_{n}\right\|  _\mathfrak{H}^{2}\right]  }\\
&  =\sqrt{\mathbb{E}\left[  \int x^{2}\left(  DX_{n}\right)  ^{4}d\mu\right]
}\sqrt{q\mathbb{E}\left[  X_{n}^2\right]  }\rightarrow0 \hspace{1cm} \text{(by point 3 of Proposition \ref{lem_WPcomp}).}%
\end{align*}
\end{proof}

\end{document}